\documentclass[11pt]{amsart}
\usepackage{amsfonts,amsthm,amsmath,amssymb,etoolbox,comment,graphicx,array,tabu,commath,latexsym,environ,tikz,tikz-cd,stackengine,diagbox,caption,enumitem}
\usepackage{soul}

\usepackage[colorlinks, linkcolor = black, anchorcolor = black, citecolor = black, urlcolor=black]{hyperref}

\newcommand{\cE}{\mathcal{E}}

\newcommand{\cG}{\mathcal{G}}
\newcommand{\cH}{\mathcal{H}}
\newcommand{\cI}{\mathcal{I}}

\newcommand{\cM}{\mathcal{M}}

\newcommand{\cP}{\mathcal{P}}

\newcommand{\cS}{\mathcal{S}}

\newcommand{\cU}{\mathcal{U}}
\newcommand{\cV}{\mathcal{V}}

\newcommand{\cW}{\mathcal{W}}

\newcommand{\bH}{\mathbf{H}}

\newcommand{\bp}{\mathbf{p}}
\newcommand{\bq}{\mathbf{q}}
\newcommand{\bv}{\mathbf{v}}
\newcommand{\bw}{\mathbf{w}}
\newcommand{\bt}{\mathbf{t}}
\newcommand{\bs}{\mathbf{s}}
\newcommand{\btau}{{\boldsymbol\tau}}
\newcommand{\bmu}{{\boldsymbol\mu}}
\newcommand{\bkappa}
{{\boldsymbol\kappa}}
\newcommand{\bzeta}{{\boldsymbol\zeta}}
\newcommand{\bsigma}{{\boldsymbol\sigma}}

\newcommand{\bomega}
{{\boldsymbol\omega}}
\newcommand{\Db}{D}

\renewcommand{\H}{\mathcal H}
\newcommand{\R}{\mathbb R}

\newcommand{\N}{\mathbb N}

\newcommand{\area}{\operatorname{area}}
\newcommand{\Hess}{\operatorname{Hess}}
\newcommand{\dist}{\operatorname{dist}}

\newcommand{\loc}{\mathrm{loc}}
\newcommand{\UConf}{\mathrm{UConf}}

\renewcommand{\tilde}{\widetilde}
\newcommand{\tr}{\operatorname{tr}}

\newcommand{\x}{\times}

\newcommand{\Ric}{\operatorname{Ric}}

\newcommand{\InjRad}{\mathrm{InjRad}}

\usepackage[style=alphabetic, backend=biber, sorting=nyt, url=false ]{biblatex}
\newtheorem{thm}{Theorem}[section]

\newtheorem{prop}[thm]{Proposition}
\newtheorem{lem}[thm]{Lemma}
\newtheorem{cor}[thm]{Corollary}

\theoremstyle{definition}
\newtheorem{definition}[thm]{Definition}

\theoremstyle{remark}
\newtheorem{remark}[thm]{Remark}

\begin{document}

\title[Electrostatics and minimal surface doublings]{Minimal surface doublings and electrostatics for Schr\"odinger operators }
\author{Adrian Chun-Pong Chu}
\author{Daniel Stern}
\address{Cornell University, Department of Mathematics, 310 Malott Hall, Ithaca, NY 14853}
\email{cc2938@cornell.edu}
\email{daniel.stern@cornell.edu}
\begin{abstract}
Twenty years ago, N. Kapouleas introduced a singular perturbation construction known as ``doubling", which produces sequences of high-genus minimal surfaces converging to a given minimal surface with multiplicity two. Doubling constructions have since been implemented successfully in several settings, with deep work of Kapouleas--McGrath reducing their existence theory to the problem of finding suitable families of ansatz data on the initial minimal surface. 

In this paper, we introduce a variational approach to the existence of minimal doublings, relating the Kapouleas--McGrath construction to the study of nondegenerate critical points for a Coulomb-type interaction energy. By analyzing the minimizers of this energy, we prove that, in a generic closed $3$-manifold, every two-sided, embedded minimal surface of index one admits a sequence of minimal doublings. As a corollary, we find that a generic 3-manifold contains an infinite sequence of embedded minimal surfaces with bounded area and arbitrarily large genus, whose geometry can be described with some precision. 
\end{abstract}

\maketitle

\section{Introduction}\label{sect:intro}

The last fifteen years have been a time of rapid progress in the existence theory for minimal hypersurfaces, with a number of dramatic results originating in Marques and Neves's efforts to develop Morse-theoretic methods for the area functional via Almgren-Pitts min-max theory. Among these, one of the most notable achievements is the proof that any closed Riemannian manifold contains infinitely many minimal hypersurfaces--first proved for Ricci positive metrics in \cite{MN17}, then for generic metrics in \cite{IMN18} (see also \cite{ChoMan1, Zho20}), and finally for arbitrary metrics via a contradiction argument by Song \cite{Son23}, settling a well-known conjecture of Yau \cite[Problem 88]{Yau82}. 

In general, these constructions come with limited information about the resulting minimal surfaces, beyond bounds on the area and Morse index that tend to infinity with the number of parameters in the associated min-max construction. In dimension three, recent refinements of the Simon-Smith min-max theory have also facilitated the construction of minimal surfaces with bounded or prescribed genus in various settings \cite{HK19, Ket19, WangZhou23FourMinimalSpheres,  chuLi2024fiveTori, LiWang2024NineTori, chuLiwang2025}, as in the first-named author's recent proof that every Ricci-positive metric on $S^3$ contains minimal surfaces of arbitrary genus and bounded area \cite{chu2025}. In some special ambient geometries like $\mathbb{R}^3$, $\mathbb{S}^3$, $\mathbb{B}^3$, and Gaussian space $(\mathbb{R}^3,e^{-|x|^2/4}\delta_{ij})$, other methods have also been successful in detecting new families of minimal surfaces, including equivariant min-max constructions \cite{Ket16, Ket16b, CFS22, Ket22, franzKetoverMario2024genusOneCatenoid, BNS25, SW25}, eigenvalue optimization \cite{FS16, KKMS24}, and singular perturbation or gluing techniques  \cite{KapYang, KapouleasKleeneMøller, KapEq1, Wiygul20, KL17, KapPeter19, KW17b, FPZ17, kapouleasMcGrath2023generalDoubling, KZ21, CSW22}, where the above references are by no means exhaustive. Gluing methods in particular have the advantage of producing minimal surfaces with a very precise geometric structure, at the expense of missing examples with small first Betti number. 

The focus of the present paper is a particular gluing construction due to Kapouleas known as \emph{doubling}, which detects near a given minimal surface $\Sigma$ families of minimal surfaces resembling two copies of $\Sigma$ joined by a large number of small catenoidal necks.
Since its introduction in \cite{KapYang}, doubling and generalizations thereof have been successfully employed to produce many new families of minimal surfaces in $\mathbb{S}^3$ \cite{KapEq1,Wiygul20,KapPeter19,KapPeter24}, self-shrinkers for the mean curvature flow \cite{kapouleasMcGrath2023generalDoubling}, and free boundary minimal surfaces in $\mathbb{B}^3$ \cite{FPZ17,KW17b}. In an exciting development, the technique has also recently been extended to dimension four, yielding new minimal hypersurfaces in $\mathbb{S}^4$ \cite{KZ24}. 

Up to now, most doubling constructions have been carried out in settings where both the ambient $3$-manifold and initial minimal surface possess a large symmetry group, but Kapouleas has suggested that it may be possible to implement the construction in much greater generality, potentially giving an alternate approach to Yau's conjecture \cite[Section 1]{KapSurvey}. A significant step in this direction was made in the paper \cite{kapouleasMcGrath2023generalDoubling}, in which Kapouleas and McGrath reduced the existence theory for minimal doublings to the problem of finding suitable families of ansatz data on a given minimal surface $\Sigma$.

The goal of the present paper is to develop quite general doubling constructions that can be implemented in generic settings, by relating the results of \cite{kapouleasMcGrath2023generalDoubling} to the variational theory of a suitable Coulomb-type interaction energy. As our main result, we identify a doubling construction that can be applied to any minimal surface of index one in a generic closed $3$-manifold.

\begin{thm}\label{thm:main_doubling}
Let $N$ be a closed $3$-manifold with a $C^{\infty}$-generic Riemannian metric $g$. Then for every two-sided, embedded minimal surface $\Sigma\subset N$ of index one, there exists a sequence of embedded minimal surfaces $\{M_k\}$ with unbounded genus such that $M_k\to 2\Sigma$ as varifolds.
\end{thm}

The surfaces $M_k$ will be orientable if and only if the initial surface $\Sigma$ is orientable. In the nonorientable case, genus can be taken to refer to the number of $\mathbb{RP}^2$ summands in a connected sum decomposition of $M_k$. Here, as in \cite{IMN18} and elsewhere, we use ``generic" to mean comeager in the space of smooth Riemannian metrics: i.e., the conclusions of the theorem hold within a countable intersection of open, dense subsets of the space of smooth metrics on $N$. Thanks to recent advances in Almgren-Pitts (or Allen-Cahn) min-max theory \cite{ChoMan1,Zho20, MN21}, we recall that any closed $3$-manifold $(N^3,g)$ with a bumpy metric contains a two-sided, embedded minimal surface $\Sigma$ of index one, with area bounded by the first Almgren-Pitts min-max width $\omega_1(N,g)$. As a result, Theorem \ref{thm:main_doubling} implies the following, giving a gluing approach to Yau's conjecture for generic $3$-manifolds.

\begin{cor}\label{thm:maincor}
    For a $C^\infty$-generic Riemannian metric $g$ on a closed $3$-manifold $N$, there exists a sequence of embedded minimal surfaces with unbounded genus and area less than $2\omega_1(N,g)$.
\end{cor}

As is typical for gluing constructions--building once again on the analysis of \cite{kapouleasMcGrath2023generalDoubling}--it is possible to give a much more precise description of the surfaces produced in Theorem \ref{thm:main_doubling} in relation to the initial surface $\Sigma$. To do so, we first introduce some notation: denoting by $G(x,y)$ the Green's function for the Jacobi operator $J=d^*d-(|A|^2+\Ric(\nu,\nu))$ of $\Sigma$ as in Section \ref{subsect:hdef} below, we define the interaction energy $\cE(\mu)\in (-\infty,+\infty]$ of a probability measure $\mu$ on $\Sigma$ by
$$\cE(\mu):=\int_{\Sigma\times \Sigma} G(x,y)d\mu(x)d\mu(y).$$
For any unstable minimal surface $\Sigma$, it is straightforward to check that the infimum
$$\cI:=\inf\{\cE(\mu)\mid \mu(\Sigma)=1\}$$
is a negative number, realized by some \emph{equilibrium measure} $\mu$ whose support is contained in the set $\{|A|^2+\Ric(\nu,\nu)\geq 0\}$. 

\begin{thm}\label{thm: asymptotics}
    The minimal surfaces $M_k$ in Theorem \ref{thm:main_doubling} satisfy the following:
    \begin{itemize}
    \item For each $k$, $M_k$ is a two-sheeted graph over $\Sigma$ away from a collection of $n_k$ approximately catenoidal necks centered at some points $p_1^k,\ldots,p^k_{n_k}\in \Sigma$.
    \item As $k\to \infty$, the areas $\area(M_k)$ satisfy  $$\lim_{k\to\infty}\frac{\log(2\area(\Sigma)-\area(M_k))}{n_k}=4\pi \cI.$$
    \item The probability measure $\frac{1}{n_k}\sum_{i=1}^{n_k}\delta_{p_i^k}$ on $\Sigma$ weak$^*$-converges to an equilibrium measure $\mu$ for $\cE$ as $k\to\infty$.
    \end{itemize}
\end{thm}

\begin{remark}
As discussed in Proposition \ref{prop:eq.char} below, for an index one minimal surface $\Sigma\subset (N^3,g)$ such that $\{|A|^2+\Ric(\nu,\nu)>0\}=\Sigma$, one can show that
$$\cI=-\frac{1}{\int_{\Sigma}(|A|^2+\Ric(\nu,\nu))dA_g},$$
and the equilibrium measure $\mu$ for $\mathcal{E}$ is precisely $\mu=-\cI(|A|^2+\Ric(\nu,\nu))dA_g$. In fact, assuming $|A|^2+\Ric(\nu,\nu)>0$, this is \emph{equivalent} to the statement that $\Sigma$ has Morse index one.
\end{remark}




\subsection{Main ideas: the doubling theorem}
Given a nondegenerate minimal surface $\Sigma^2\subset (N^3,g)$ with Jacobi operator $J$, minimal doublings of $\Sigma$ are constructed in \cite[Sections 3-5]{kapouleasMcGrath2023generalDoubling} from families of \emph{linearized doubling solutions} satisfying a suite of important structural conditions. These linearized doubling solutions are functions $\varphi\in C^{\infty}(\Sigma\setminus\{p_1,\ldots,p_n\})$ solving
\begin{equation}\label{ld.def}
J\varphi=-2\pi\sum_{i=1}^n\tau_i\delta_{p_i}
\end{equation}
for some points $p_1,\ldots,p_n\in\Sigma$ and weights $\tau_1,\ldots,\tau_n>0$, where the weights $\tau_i$ are of roughly comparable size, and much smaller than the separation $\min_{i\neq j}d(p_i,p_j)$. Standard results for Green's functions on surfaces imply that for any such $\varphi$, the difference $\varphi(x)-\tau_i\log(d(p_i,x))$ is $C^1$ near $p_i$, and we recall that the catenoid of waist $\tau>0$ in $\mathbb{R}^3$ can be expressed as a two-sheeted graph $$\{z=\pm \tau\cosh^{-1}(|(x,y)|/\tau)\mid |(x,y)|\geq \tau\},$$
where $\cosh^{-1}(|(x,y)|/\tau)\approx \log(2|(x,y)|/\tau)$ for $|(x,y)|>>\tau$. If we wish to construct a new surface with small mean curvature by gluing the two-sheeted normal graph of $\pm\varphi$ over $\Sigma$ to approximately catenoidal necks of waist $\tau_i$ centered at $p_i$, it is then intuitively reasonable to seek solutions of \eqref{ld.def} for which the difference 
\begin{equation}\label{doub.diff}
\varphi(x)-\tau_i\log(2d(p_i,x)/\tau_i)=\varphi(x)-\tau_i\log(d(p_i,x))+\tau_i\log(\tau_i/2)
\end{equation}
is kept as small as possible in the gluing region. In \cite{kapouleasMcGrath2023generalDoubling}, this difference is controlled through the \emph{mismatch} 
$$\mathcal{M}(\varphi)\in \bigoplus_{i=1}^n\left(T^*_{p_i}\Sigma\oplus\mathbb{R}\right),$$
which associates to each linearized doubling solution $\varphi$ the first order Taylor expansion of \eqref{doub.diff} at each of the singularities $p_1,\ldots,p_n$. 

Our work in the present paper begins with the observation that this mismatch $\cM(\varphi)$ can be identified with the gradient of a Coulomb-type interaction energy. Specifically, given an invertible Schr\"odinger operator $J=d^*d-V$ on $(\Sigma,g)$, denote by $G:\Sigma\times \Sigma\setminus diag(\Sigma)\to \mathbb{R}$ the Green's function satisfying
$$u(x)=\int_{y\in \Sigma} Ju(y)G(x,y)dA_g$$
where
$$diag(\Sigma):=\{(x,x):x\in\Sigma\}\subset\Sigma\x\Sigma,$$
and let $R\in Lip(\Sigma\times \Sigma)$ be the \emph{Robin's function} given by
$$R(x,y):=G(x,y)+\frac{1}{2\pi}\log d(x,y),$$
whose restriction $x\mapsto R(x,x)$ to the diagonal $diag(\Sigma)$ belongs to $C^{\infty}(\Sigma)$ (cf. Appendix \ref{green.app}).
Given $n$-tuples $\bp=(p_1,\ldots,p_n)$, $\btau=(\tau_1,\ldots,\tau_n)$ of distinct points in $\Sigma$ and positive weights, we then define 
$$\cH(\bp,\btau):=\sum_{i=1}^n\sum_{j\neq i}\tau_i\tau_jG(p_i,p_j)+\sum_{i=1}^n\tau_i^2R(p_i,p_i)+\sum_{i=1}^n\frac{\tau_i^2}{4\pi}\left(1-2\log\frac{\tau_i}{2}\right),$$
and observe that this defines a smooth function $\cH: X_n\to (-\infty,\infty)$ on the space $X_n\subset \Sigma^n\times (0,\infty)^n$ of all such pairs. 

Quadratic interaction energies resembling the first two terms 
$$\sum\tau_i\tau_j G(p_i,p_j)+\sum \tau_i^2 R(p_i,p_i)$$
of $\mathcal{H}$ arise naturally in many settings; for example, the analogous energy for the Green's function of the Laplace-Beltrami operator provides a Hamiltonian for the dynamics of point vortices in two-dimensional fluids \cite{DB15, DGK24}. More broadly, the dynamics and equilibria of Coulomb-type interactions play a central role in classical electrostatics, statistical mechanics, and $n$-body problems in Newtonian gravity, among many other areas of math and physics; see, for instance, Serfaty's ICM lecture \cite{SerfatyICM} for a nice overview.

The addition of the nonlinear term $\sum \frac{\tau_i^2}{4\pi}\left(1-\log \frac{\tau_i}{2}\right)$ is less standard, and tailored specifically to the doubling construction of minimal surfaces. Namely, the precise form of $\cH$ defined above is chosen so that  
$$\mathcal{M}(\varphi)=-\pi S_{\btau}\circ\nabla \cH(\bp,\btau),$$
where $\varphi$ is the linearized doubling associated to $(\bp,\btau)$ under \eqref{ld.def} and $S_{\btau}:\bigoplus_{i=1}^nT^*_{p_i}\Sigma\oplus \mathbb{R}^n\to \bigoplus_{i=1}^nT^*_{p_i}\Sigma\oplus \mathbb{R}^n$ is a linear map fixing the $\mathbb{R}^n$ factor and scaling $T^*_{p_i}\Sigma$ by $\tau_i^{-1}$. The first main technical result of the paper recasts \cite[Theorem 5.7]{kapouleasMcGrath2023generalDoubling} in terms of the existence of critical points for $\cH$ satisfying a list of key estimates.

\begin{thm}[cf. Theorem \ref{thm:doubling.precise} below]\label{thm:double.lem-intro}
Let $\Sigma$ be an unstable, two-sided, nondegenerate minimal surface with Jacobi operator $J$ in a closed $3$-manifold $(N^3,g)$. Then for $n$ sufficiently large, if $(\bp,\btau)$ is a critical point for $\mathcal{H}$ satisfying
\begin{enumerate}
    \item  $\displaystyle \frac{1}{n}< \frac{\tau_i}{\bar{\btau}}< n\text{ for all }1\leq i\leq n$ where $\bar{\btau}:=\frac{1}{n}\sum_{i=1}^n\tau_i\in (0,n^{-100,000})$,
    \item $ \displaystyle    \delta(\bp):=\frac{1}{20}\min_{i\neq j}d(p_i,p_j)>\bar{\btau}^{1/100,000}$,
    \item \label{quant.nondeg}
    $\min\{|\lambda|\mid \lambda \in Spec(S_{\btau}D^2\mathcal{H}(\bp,\btau)S_{\btau})\}>\bar{\btau}^{1/1000},$ and
    \item  \label{2eigen.bd} $\lambda_2(J;\Sigma\setminus \bigcup_{i=1}^nB_{\sigma_i/3}(q_i))>0$
for all $(\bq,\bsigma)$ sufficiently close to $(\bp,\btau)$, 
\end{enumerate}
 then there exists an embedded minimal doubling $M_n\subset (N,g)$ of $\Sigma$ with $n$ necks such that 
\begin{equation}\label{area.exp}
    |\area(M)-(2\area(\Sigma)-\frac{1}{4}\mathcal{H}(\bp,\btau))|\leq \frac{1}{n}\mathcal{H}(\bp,\btau).
\end{equation}
\end{thm}

\begin{remark} 
    In fact, the proof of Proposition \ref{prop:p_iLowerBoundDistance} below shows that any critical point satisfying condition (1) automatically satisfies a separation estimate akin to condition (2), so it is possible to state a more streamlined version that removes item (2) from the hypotheses.
\end{remark}

A more detailed version of this theorem is given in Theorem \ref{thm:doubling.precise}. Conditions (1) and (2) are self-explanatory, ensuring that the weights $\tau_i$ are of roughly comparable size, much smaller than $\frac{1}{n}$ or $\min_{i\neq j}d(p_i,p_j)$ in a logarithmic sense. Condition (3) on the spectrum $Spec(S_{\btau}D^2\cH(\bp,\btau)S_{\btau})$ of $S_{\btau}D^2\cH(\bp,\btau)S_{\btau}$ is a nondegeneracy criterion, ensuring the existence of a neighborhood of $(\bp,\btau)$ which is mapped bijectively by the mismatch $(\bq,\bsigma)\mapsto \cM(\varphi_{\bq,\bsigma})$ onto a ball of radius $\sim \tau_i^{1+1/500}$ about the origin in a suitable space $\cP\cong \bigoplus_{i=1}^nT^*_{p_i}\Sigma\oplus \mathbb{R}^n$: note here $\varphi_{\bq,\bsigma}$ denotes the function satisfying $J\varphi_{\bq,\bsigma}=-2\pi\sum_i\sigma_i\delta_{q_i}$. The proof, given in Section \ref{sect:Proofdoubling.precise}, amounts to carefully checking that this neighborhood gives rise to a family of linearized doubling solutions meeting all the criteria necessary to carry out the fixed point argument of \cite[Section 5]{kapouleasMcGrath2023generalDoubling}.

The condition (4) is used to ensure positivity of the associated linearized doublings $\varphi$ sufficiently far from the singularities, which is essential for the \emph{embeddedness} of the resulting minimal doubling. It is satisfied automatically when $\lambda_2(J)>0$--i.e., when $\Sigma$ has Morse index one--which is the primary reason for the restriction to surfaces of index one in our main theorems. Indeed, for the existence of \emph{immersed} minimal doublings, the index one condition could probably be relaxed throughout, giving immersed minimal surfaces of high genus lying near any fixed unstable minimal surface in a generic $3$-manifold. 

\subsection{Main ideas: variational theory for $\cH$ and proof of the main theorems}

In Section \ref{sect:minH}, we introduce a simple variational construction to obtain critical points of $\cH$ for any invertible, indefinite Schr\"odinger operator $J$ on a closed surface $(\Sigma,g)$. The construction can be formulated as a one-parameter min-max construction for the functional $\cH$ or, more conveniently, as a minimization problem for a projectivized energy 
$$\hat\cH(\bp,[\btau]):=\max_{t\in (0,\infty)}\cH(\bp, t\btau)\in (0,\infty),$$
whose critical point theory is equivalent to that of $\cH$. We show that this minimization problem is well-posed on $X_n$
for every $n\in \mathbb{N}$, yielding critical points of $\cH$ of \emph{lowest possible energy}. When $J$ is the Jacobi operator of a minimal embedding $\Sigma\subset (N^3,g)$, it is then reasonable to hope that Theorem \ref{thm:double.lem-intro} can be applied to produce a \emph{minimal doubling of largest possible area} with respect to its genus.

For any invertible, indefinite Schr\"odinger operator $J$, we succeed in showing that these least-energy critical points $(\bp,\btau)$ satisfy estimates of the form
$$c\leq \frac{\tau_i}{\tau_j}\leq C,$$
$$e^{-Cn}\leq \tau_i\leq e^{-cn},$$
and  
$$d(p_i,p_j)>c_{\epsilon}n^{-2-\epsilon}$$
for any $i\neq j$ and $\epsilon>0$, for constants depending on $J$ and the geometry of $\Sigma$. Moreover, drawing inspiration from the analysis of mean-field limits for classical Coulomb energies as in, e.g., \cite[Section 2]{SerfatyLectures}, we find that $\tau_i/\tau_j\to 1$ in an average sense as $n\to\infty$, and the probability measures $\frac{1}{n}\sum_{i=1}^n\delta_{p_i}$ tend to an equilibrium measure $\mu$ for the interaction energy $\cE$ as $n\to\infty$.

Applying this construction to the Jacobi operator of an unstable, two-sided, nondegenerate minimal surface $\Sigma\subset (N^3,g)$, we obtain for each sufficiently large $n$ a critical point $(\bp, \btau)$ satisfying conditions (1) and (2) of Theorem \ref{thm:double.lem-intro}. If $\Sigma$ has Morse index one, then condition (4) holds automatically as well. The remaining challenge, then, is to check that the quantitative nondegeneracy condition (3) of Theorem \ref{thm:double.lem-intro} holds generically for these least-energy critical points of $\cH$, at least along some subsequence $n_k\to\infty$. In Section \ref{sect:wt.nondeg}, we establish quantitative nondegeneracy with respect to the weight variables $\tau_1,\ldots,\tau_n$ in generality, with no perturbation needed, so the key point is to achieve nondegeneracy in the position variables $p_1,\ldots,p_n$.

This last step is carried out in Section \ref{sect:MainProof}. As a key technical ingredient, given a nondegenerate, unstable minimal surface $\Sigma\subset (N^3,g)$ and sufficiently large $n$, we construct by hand a small conformal perturbation $g_n'\in [g]$ with respect to which $\Sigma$ remains minimal, while every least-energy critical point of the interaction energy $\cH_{(g_n',\Sigma)}$ on $X_n$ for the new Jacobi operator and induced metric on $\Sigma$ satisfies all four conditions (1)-(4) of Theorem \ref{thm:double.lem-intro}, essentially by creating small additional wells around the points $p_1,\ldots,p_n$ obtained by $\hat\cH$-minimization in the initial metric. Moreover, $g_n'\to g$ exponentially fast as $n\to\infty$. Since $\Sigma$ is nondegenerate, we can then fix a neighborhood $\cV(g,\Sigma)$ in the set $\Gamma_\infty$ of all smooth metrics on $N$ such that there exists a unique $g'$-minimal surface $\Sigma'$ close to $\Sigma$ in a $C^{\infty}$ sense for every $g'\in \cV(g,\Sigma)$, and use this perturbation construction to show that the open set $\mathcal{V}_k(g,\Sigma)$ of metrics $g'\in \cV(g,\Sigma)$ such that the least-energy critical points of $\cH_{(g',\Sigma')}$ satisfy conditions (1)-(4) of Theorem \ref{thm:double.lem-intro} \emph{for some} $n\geq k$ is dense in $\cV(g,\Sigma)$. 

Using this argument, for any fixed $k\in \mathbb{N}$ and $\epsilon>0$, we then show that the set of smooth metrics $g$ for which each of the (finitely many) minimal surfaces $\Sigma\subset (N,g)$ satisfying 
$$\lambda_1(J_{\Sigma})+\epsilon<0<\lambda_2(J_{\Sigma})-\epsilon$$
and
$$\area(\Sigma)+|A_{\Sigma}|_{C^0}<\frac{1}{\epsilon}$$
satisfies conditions (1)-(4) of Theorem \ref{thm:double.lem-intro} for some $n_k\geq k$ is open and dense in the space of smooth metrics. Taking the intersection over $\epsilon\in \mathbb{Q}\cap (0,1)$ and $k\in \mathbb{N}$, and intersecting with the space of bumpy metrics (whose genericity was proved by White), we arrive finally at a generic set of metrics $g$ such that for any index one minimal surface $\Sigma\subset (N,g)$, there exists a sequence $n_k\to\infty$ for which the lowest-energy critical points $(\bp,\btau)$ of $\cH_{(g,\Sigma)}$ satisfy the hypotheses of Theorem \ref{thm:double.lem-intro}, so Theorems \ref{thm:main_doubling} and \ref{thm: asymptotics} follow.

\subsection{Discussion and future directions}

Since the energies $\cH$ and $\hat\cH$ are invariant under the obvious action of the permutation group $S_n$ on $X_n$, they can be regarded as functions on the quotient $X_n/S_n$. This quotient is easily seen to contain as a deformation retract the unordered configuration space $\UConf_n(\Sigma)$ of $n$ points on the surface $\Sigma$, whose topology has been the subject of intensive investigation \cite{Fuks70Cohomology,vershinin99braid,napolitano03cohomology,church12homologicalStability,church15FIModule,drummondKnudsen16betti,knudsen17bettiFactorizationHomology,millerWilson19higher}. Given the rich topology of these spaces, it is natural to seek out further doubling constructions by min-max or Morse-theoretic techniques for $\hat\cH$, resulting in doublings of lower area than those constructed here. As a concrete goal, it would be very interesting to identify, for a given $k\geq 2$, a variational construction for $\hat\cH$ for which condition (4) of Theorem \ref{thm:double.lem-intro} can be verified in some generality on surfaces $\Sigma$ of Morse index $k$, yielding \emph{embedded} minimal doublings for minimal surfaces of higher Morse index.

As stated, the results of this paper cannot be applied directly to surfaces $\Sigma$ possessing continuous families of isometries commuting with $J$. However, in many cases of interest, it should be possible to employ a version of Theorem \ref{thm:double.lem-intro} that works modulo isometries, as in most of the explicit doubling constructions in the literature. For the special case of the equator $\Sigma=\mathbb{S}^2\subset \mathbb{S}^3$, we expect that the minimal doublings given by $\hat\cH$-minimization coincide with those obtained by equivariant eigenvalue optimization in \cite[Theorem 1.1]{KKMS24} with $\Gamma=\mathbb{Z}_2$, for which coarse asymptotics were obtained in \cite{KMS25}. If true, this would answer the natural analog in the hemisphere $\mathbb{S}^3_+$ of a question \cite[Question 5]{KS24-steklov} raised by Karpukhin and the second-named author regarding the free boundary minimal surfaces of genus zero in $\mathbb{B}^3$ obtained by maximizing the first Steklov eigenvalue.

While we have focused in this article on doubling constructions for closed minimal surfaces in closed $3$-manifolds, it would be interesting to develop analogous variational techniques for other generalizations of the doubling construction--e.g., to the free boundary setting, and to higher-multiplicity stackings. In light of the recent work of Kapouleas-Zou \cite{KZ24}, there is also the prospect of extending these tools to higher dimension, where currently few methods are available for producing minimal hypersurfaces with prescribed topology.

\subsection{Organization} In Section \ref{sect:KMtoEnergyFunctional}, we introduce the interaction energies $\cH$ and $\hat\cH$, collect some preliminary estimates, and state  a more precise version (Theorem \ref{thm:doubling.precise})  of Theorem \ref{thm:double.lem-intro}. In Section \ref{sect:Proofdoubling.precise}, we explain in detail how to derive Theorem \ref{thm:doubling.precise} from the arguments of \cite{kapouleasMcGrath2023generalDoubling}. In Section \ref{sect:minH}, we give a detailed analysis of the $\hat{\cH}$-minimization problem for arbitrary indefinite, invertible Schr\"odinger operators, showing that for $n$ large, the problem is well-posed, and gives rise to critical points $(\bp,\btau)$ of $\cH$ whose weights are proportional $\tau_i\sim \tau_j$ and decay exponentially fast as $n\to\infty$, with suitable lower bounds on $d(p_i,p_j)$; as a byproduct of our analysis, we also establish convergence of $\sum_{i=1}^n\frac{1}{n}\delta_{p_i}$ in the large $n$ limit to equilibrium measures, which we then characterize. At the end of Section \ref{sect:minH}, in the index one case, we also obtain general lower bounds for the Hessian $D^2\cH(\bp,\btau)$ at the minimizing critical points $(\bp,\btau)$ with respect to variations of the weight parameters $\btau$. Finally, in Section \ref{sect:MainProof}, we complete the proofs of Theorems \ref{thm:main_doubling}, \ref{thm:maincor}, and \ref{thm: asymptotics}, with the main technical ingredient being a perturbation construction used to achieve the requisite lower bounds for $D^2\cH(\bp,\btau)$ with respect to variations of the spatial variables $\bp$. In the Appendix, we recall how to derive some of the relevant estimates for Green's functions of Schr\"odinger operators, and explain the effects of an ambient conformal change on the Jacobi operator of a minimal surface.


\subsection*{Acknowledgments}
AC would like to thank Nicolaos Kapouleas for helpful discussions about doubling while he was in residence at SLMath in Fall 2024, supported by NSF grant DMS-1928930. Both authors are grateful to Peter McGrath for valuable conversations about the paper \cite{kapouleasMcGrath2023generalDoubling} and doubling constructions more generally, as well as comments on an earlier draft of this paper.   AC would  also like to thank  Antonios Zitridis for discussions about mean field theory. During the completion of this project, the research of AC was partially supported by an AMS-Simons travel grant, and the research of DS was partially supported by the NSF grant DMS 2404992 and the Simons Foundation.

\section{Basic properties of the interaction energies $\cH$ and $\hat\cH$}\label{sect:KMtoEnergyFunctional}

\subsection{Definition and first variation of $\cH$}\label{subsect:hdef}
Let $(\Sigma^2,g)$ be a closed surface with a smooth metric $g$, and let $V\in C^{\infty}(\Sigma)$ be a smooth function such that the associated Schr\"odinger operator
$$J=J_{g,V}:=d_g^*d-V$$
has trivial kernel, and is therefore invertible, by standard Fredholm theory. Before introducing the interaction energy $\mathcal{H}$ associated to $J$, we first need to collect some basic properties of the Green's function for $J$ on $\Sigma$. The following should be well-known, but for completeness we include some details of the proof in Section \ref{green.app} of the appendix.

\begin{lem}\label{green.lem}
    On the complement of $diag(\Sigma):=\{(x,x)\mid x\in \Sigma\}$, there exists a unique symmetric function 
    $$G\in C^{\infty}(\Sigma\times \Sigma\setminus diag(\Sigma))$$
    such that for each $p\in\Sigma$, the function $\phi_p\in C^{\infty}(\Sigma\setminus\{p\})$ given by $\phi_p(x)=G(p,x)=G(x,p)$ is the unique solution of
    $$J\phi_p=\delta_p,$$
    and for any $u\in C^{\infty}(\Sigma)$, we have
    $$u(x)=\int_{\Sigma}G(x,y)(Ju)(y)dA_g(y).$$
    Moreover, the \emph{Robin function}
    \begin{equation}
    R(x,y):=G(x,y)+\frac{1}{2\pi}\log d_g(x,y)
    \end{equation}
    extends to a globally defined function in $Lip(\Sigma\times \Sigma)$, which is $C^1$ near the diagonal in $\Sigma\times \Sigma$. Finally, for each $k\in \mathbb{N}$, there is a uniform bound
    $$\sup_{(x,y)\in \Sigma\times \Sigma}d_g(x,y)^k|\nabla^k G(x,y)|<\infty.$$
\end{lem}

\begin{remark}\label{robin.rk}
For later use, note that the Lipschitz norms of the Robin function $R$ and suprema of $d_g(x,y)^k|\nabla^kG(x,y)|$ vary continuously under smooth deformations of the metric $g$ and potential $V$. Though the stated regularity for the Robin function $R(x,y)$ is more or less optimal in general, in Lemma \ref{robin.reg} of the appendix, we check that the function $R_D:\Sigma\to \mathbb{R}$ given by the restriction
$$R_D(x):=R(x,x)$$
of $R$ to the diagonal satisfies $R_D\in C^{\infty}(\Sigma)$, with $C^k$ norms $\|R_D\|_{C^k}$ again varying continuously with respect to smooth variations of $g$ and $V$.
\end{remark}

Now, for each $n\in \mathbb{N}$, consider the $n$-fold product 
$$\overline{X}_n=\Sigma^n\times [0,\infty)^n$$
consisting of pairs $(\bp,\btau)$ of $n$-tuples $\bp=(p_1,\ldots,p_n)\in \Sigma^n$ and nonnegative weights $\btau=(\tau_1,\ldots,\tau_n)\in [0,\infty)^n$. Our interaction energies will be defined on the open subset
\begin{equation}
    X_n:=\{(\bp,\btau)\in \Sigma^n\x (0,\infty)^n: \text{ }p_i\ne p_j\textrm{ if }i\ne j\}
\end{equation}
of $\Sigma^n\times [0,\infty)^n$ in which all weights are positive and $\bp$ consists of $n$ \emph{distinct} points. Note that the permutation group $S_n$ acts in a natural way on $X_n$, such that the quotient $X_n/S_n$ can be identified with the space of measures on $\Sigma$ supported at exactly $n$ points, and admits an obvious deformation retraction onto a copy of the unordered configuration space $\UConf_n(\Sigma)$ of $n$ points in $\Sigma$.

We can now define the interaction energy central to our results.

\begin{definition}\label{h.def}
Define an energy functional $\H:X_n\to\R$
by
$$\cH(\bp,\btau):=\sum^n_{i=1}\sum_{j\neq i}\tau_i\tau_jG(p_i,p_j)+\sum_{i=1}^n\tau_i^2R(p_i,p_i)+\sum_{i=1}^n\frac{\tau_i^2}{4\pi}\left (1-2 \log \frac {\tau_i}2\right).$$
\end{definition}

As discussed in the introduction, the functional $\cH$ is defined in such a way that its gradient is closely related to the mismatch of linearized doubling solutions. More precisely, we have the following.

\begin{prop}\label{prop:criticalPointH}
    Given any point $(\bp,\btau)\in X_n$, let $\phi=\sum_i\tau_i\phi_{p_i}$ be the unique solution of $J\phi=\sum_i\tau_i\delta_{p_i}$. Then
    $$\nabla^{p_i}\cH(\bp,\btau) = 2\tau_i \nabla[\phi+\frac {\tau_i}{2\pi}\log(2 d_{p_i})](p_i)$$
    and
    $$\nabla^{\tau_i}{\cH}(\bp,\btau)=2[\phi+\frac{\tau_i}{2\pi}\log d_{p_i}](p_i)-\frac{\tau_i}{\pi}\log (\tau_i/2).$$
    In particular, $(\bp,\btau)$ is critical for $\mathcal{H}$ if and only if the function
    $$\phi(x)+\frac{\tau_i}{2\pi}\log(2d(p_i,x)/\tau_i)$$
    vanishes to first order at $p_i$ for each $1\leq i \leq n$.
\end{prop}
\begin{proof}
Differentiating $\cH$ in the $p_i$ variable and exploiting the symmetry of $G$ and $R$, we see that
\begin{align*}
&\nabla^{p_i}\cH(\bp,\btau)\\
 &=\sum_{j\neq i}\tau_i\tau_j(\nabla^{x} G(p_i,p_j)+\nabla^{y}G(p_j,p_i))+\tau_i^2({\nabla^x}R(p_i,p_i)+{\nabla^y}R(p_i,p_i))\\
&=2\tau_i\sum_{j\neq i}\tau_j\nabla^{y}G(p_j,p_i)+2\tau^2_i{\nabla^y}R(p_i,p_i)\\
&=2\tau_i \sum_{j\neq i}\tau_j \nabla\phi_{p_j}(p_i)+2\tau_i^2 \nabla[\phi_{p_i}+\frac {1}{2\pi}\log d_{p_i}](p_i)\\
&= 2\tau_i \nabla[\phi+\frac {\tau_i}{2\pi}\log(d_{p_i})](p_i).
\end{align*}
Then, differentiating $\cH$ in $\tau_i$, we see that
\begin{align*}
    &\nabla^{\tau_i}{\cH}(\bp,\btau)\\
    &=2\sum_{j\neq i}\tau_jG(p_j,p_i)+2\tau_i R(p_i,p_i)-\frac{\tau_i}{\pi}\log (\tau_i/2)\\
    &=2\sum_{j\neq i}\tau_i\phi_{p_j}(p_i)+2\tau_i[\phi_{p_i}+\frac 1{2\pi}\log d_{p_i}](p_i)-\frac{\tau_i}{\pi}\log (\tau_i/2)\\
    &=2[\phi+\frac{\tau_i}{2\pi}\log d_{p_i}](p_i)-\frac{\tau_i}{\pi}\log (\tau_i/2).
\end{align*}
This gives the desired result.
\end{proof}

When $\Sigma\subset (N^3,g)$ is a nondegenerate, two-sided minimal surface with Jacobi operator 
$$J=d_g^*d-(|A|^2+\Ric(\nu,\nu)),$$
and $\phi=\sum_i \tau_i \phi_{p_i}$ is the Green's function associated to $(\bp,\btau)\in X_n$, then the function $\varphi=-2\pi \phi$ gives a ``linearized doubling solution" as in \cite[Definition 3.6]{kapouleasMcGrath2023generalDoubling}. In the terminology of \cite[Definition 3.10]{kapouleasMcGrath2023generalDoubling}, the ``mismatch" of such a linearized doubling solution $\varphi$ is given by the first-order Taylor expansion of $\varphi-\tau_i\log(2d(p_i,x)/\tau_i)$ at each of the points $p_1,\ldots,p_n$, so Proposition \ref{prop:criticalPointH} relates this mismatch to the gradient $\nabla \mathcal{H}(\bp,\btau)$.

To successfully carry out the doubling construction of \cite{kapouleasMcGrath2023generalDoubling}, one needs not only a single linearized doubling solution with mismatch zero, but a family of linearized doubling solutions realizing any sufficiently small mismatch, while satisfying certain estimates for the weights $\tau_i$ and distances $d(p_i,p_j)$. Thus, we are led to consider not only critical points for the energies $\mathcal{H}$, but sufficiently \emph{nondegenerate} critical points, in a quantitative sense that we make precise in the following subsection.

\subsection{Statement of the doubling theorem}

Before giving a precise statement of the doubling theorem, we introduce a few more pieces of notation. First, for any $(\bp,\btau)\in X_n$, denote by $\mathcal{P}_{\bp}$ the tangent space
$$\mathcal{P}_{\bp}:=(\bigoplus_{i=1}^nT_{p_i}\Sigma)\oplus \mathbb{R}^n\cong T_{(\bp,\btau)}X_n.$$
Note that $\mathcal{P}_{\bp}$ inherits from the metric on $\Sigma$ a natural $\ell^2$ inner product: namely, given $(\bv,\bs)=(v_1,\ldots,v_n,s_1,\ldots,s_n)$ and $(\bw,\bt)=(w_1,\ldots,w_n,t_1,\ldots,t_n)\in\mathcal{P}_{\bp}$, we set
$$\langle (\bv,\bs),(\bw,\bt)\rangle_{\ell^2}:=\sum_{i=1}^n\langle v_i,w_i\rangle+\sum_{i=1}^ns_it_i.$$
While it is sometimes convenient to work with the associated $\ell^2$ norm on $\mathcal{P}_{\bp}$, to keep notation more consistent with \cite{kapouleasMcGrath2023generalDoubling}, for the most part we will use the $\ell^{\infty}$-type norm
$$|(\bv,\bs)|_{\ell^{\infty}}:=\max\{\sqrt{|v_i|^2+|s_i|^2}\mid 1\leq i\leq n\}$$
as our default norm on $\mathcal{P}_{\bp}$, keeping in mind the obvious inequalities
$$|(\bv,\bs)|_{\ell^{\infty}}\leq |(\bv,\bs)|_{\ell^2}\leq \sqrt{n}|(\bv,\bs)|_{\ell^{\infty}}.$$
Similarly, we find it convenient to equip $\Sigma^n$ and $(0,\infty)^n$ with the $\ell^{\infty}$-type metrics
$$d(\bp,\bq):=\max\{d(p_i,q_i)\mid 1\leq i\leq n\}$$
and
$$|\bv-\bw|:=|\bv-\bw|_{\ell^{\infty}}.$$

Next, at a critical point $(\bp,\btau)$ for $\mathcal{H}$, note that there is no ambiguity in defining the Hessian as a symmetric bilinear form
$$D^2\mathcal{H}(\bp,\btau):\mathcal{P}_{\bp}\otimes \mathcal{P}_{\bp}\to \mathbb{R},$$
which we can identify, via the $\ell^2$ inner product, with a self-adjoint map
$$D^2\mathcal{H}(\bp,\btau): \mathcal{P}_{\bp}\to \mathcal{P}_{\bp}.$$
Finally, defining a rescaling map
$$S_{\btau}: \mathcal{P}_{\bp}\to \mathcal{P}_{\bp}$$
by
$$S_{\btau}(\bv,\bs):=(\tau_1^{-1}v_1,\ldots,\tau_n^{-1}v_n,s_1,\ldots,s_n),$$
it is easy to see that $S_{\btau}$ is also self-adjoint with respect to the $\ell^2$ inner product, and therefore the composition
$$S_{\btau}\circ D^2\cH(\bp,\btau)\circ S_{\btau}: \mathcal{P}_{\bp}\to \mathcal{P}_{\bp}$$
is self-adjoint as well, hence diagonalizable with respect to an $\ell^2$-orthonormal basis. 

With this notation in place, we can now state a precise version of the doubling theorem.

\begin{thm}\label{thm:doubling.precise}
Let $(N,g)$ be a closed Riemannian $3$-manifold, and let
$\Sigma$ be a closed, two-sided, unstable, nondegenerate minimal surface. Letting $\mathcal{H}$ be the interaction energy associated to the Jacobi operator $J$ as above, there exists a constant $\tilde{n}(g,\Sigma)\in \mathbb{N}$ with the follow property: For any  $n\geq \tilde{n}$, suppose $(\bp,\btau)\in X_n$ is a critical point for $\mathcal{H}$ such that:
\begin{enumerate}
    \item\label{tau.comparable}  $\displaystyle \frac{1}{n}< \frac{\tau_i}{\bar{\btau}}< n\text{ for all }1\leq i\leq n$, and  $\bar{\btau}:=\frac{1}{n}\sum_{i=1}^n\tau_i\in (0,n^{-100,000})$.
    \item\label{del.low} $ \displaystyle    \delta(\bp):=\frac{1}{20}\min_{i\neq j}d(p_i,p_j)>\bar{\btau}^{1/100,000}$.
    \item \label{quant.nondeg}
    $\min\{|\lambda|\mid \lambda \in Spec(S_{\btau}D^2\mathcal{H}(\bp,\btau)S_{\btau})\}>\bar{\btau}^{1/1000}.$
    \item  \label{2eigen.bd} 
For all $(\bq,\bsigma)\in X_n$ with $d(\bp,\bq)+\bar{\btau}^{-1}|\bsigma-\btau|<\delta^3$, we have 
$$\lambda_2(J;\Sigma\setminus \bigcup_{i=1}^nB_{\sigma_i/3}(q_i))>0.$$
\end{enumerate}
 Then there exists an embedded minimal doubling $M\subset (N,g)$ of $\Sigma$ with $n$ necks and area $\area(M)<2\area(\Sigma)$, contained in the $\bar{\btau}^{1/2}$-neighborhood of $\Sigma$. More precisely, $M$ satisfies
 $$|\area(M)-(2\area(\Sigma)-\frac{1}{4}\mathcal{H}(\bp,\btau))|\leq \frac{1}{n}\mathcal{H}(\bp,\btau)$$
 and for some $(\bq,\bsigma)$ with $d(\bq,\bsigma)+\bar{\btau}^{-1}|\bsigma-\btau|<\delta^3$, $M$  is given by the union of a two-sheeted graph over $\Sigma\setminus \bigcup_{i=1}^n B_{\sigma_i^{1/2}}(q_i)$
 with a collection of $n$ neck regions given by small perturbations of catenoids of waist radius $\sigma_i$ centered at the points $q_i$.
 
\end{thm}

The proof is the content of Section \ref{sect:Proofdoubling.precise} below, in which we check carefully that the assumptions give rise to a family of linearized doubling solutions to which the results of \cite{kapouleasMcGrath2023generalDoubling} can be applied. 

The precise choice of constants in the statement of Theorem \ref{thm:doubling.precise} is somewhat arbitrary, but sufficient for the construction we carry out in the following sections. The assumption \eqref{2eigen.bd} can be removed at the expense of the embeddedness of the resulting minimal surface, yielding in general immersed minimal doublings. Since we are primarily concerned here with the embedded case, we do not pursue this further.

\subsection{The projectivized energy $\hat\cH$}\label{sect:hatH.def}

Motivated by Theorem \ref{thm:doubling.precise}, we are now interested in developing the variational theory for the functionals $\mathcal{H}$ on $X_n$. As a first step, we observe next that the variational theory for $\mathcal{H}$ is equivalent to that of a positive, projectivized functional $\hat\cH$.

\begin{lem}\label{lem:tMax}
    Fix any $(\bp,\btau)\in X_n$. The restriction of $\cH$ to the ray $$\{\bp\}\x\{t\btau:t\in (0,\infty)\}$$ has exactly one critical point at a global maximum, where  $t$  is given by
$$t_{\max} = e^{2\pi \cH(\bp,\btau)/|\btau|_{\ell^2}^2-1/2}.$$
    In particular, $\cH(\bp,t_{\max} \btau) = \frac 1{4\pi}|t_{\max}\btau|_{\ell^2}^2$.
\end{lem}
\begin{proof}
By a direct computation,
\begin{equation}\label{eq:scaleTau}
    \cH(\bp,t\btau) =t^2\cH(\bp,\btau)-t^2\log(t)\cdot\frac { |\btau|_{\ell^2}^2}{2\pi}
\end{equation}
and
$$\frac d{dt}\cH(\bp,t\btau)=2t(\cH(\bp,\btau)-\frac{|\btau|_{\ell^2}^2}{4\pi}-\frac{|\btau|_{\ell^2}^2}{2\pi}\log(t)).$$
Hence, $\cH(\bp,t\btau)$ achieves a positive maximum at a unique critical point $t_{\max}$ given by
$$\log(t_{\max})=\frac{2\pi \cH(\bp,\btau)}{|\btau|_{\ell^2}^2}-\frac{1}{2}.$$
Putting $t=t_{\max}$ into (\ref{eq:scaleTau}),  we also have $$\cH(\bp,t_{\max} \btau) = \frac 1{4\pi} e^{4\pi\cH(\bp,\btau)/|\btau|_{\ell^2}^2-1}|\btau|_{\ell^2}^2 = \frac 1{4\pi}|t_{\max}\btau|_{\ell^2}^2$$
\end{proof}

As a corollary, we see that if $(\bp,\btau)\in X_n$ is a critical point of $\mathcal{H}$, then we must have $1=t_{max}$ above; i.e., 
\begin{equation}\label{crit.h.val}
    \mathcal{H}(\bp,\btau)=\frac{1}{4\pi}|\btau|_{\ell^2}^2,
\end{equation}
and $(\bp, \btau)$ maximizes $\mathcal{H}$ along the ray $\{\bp\}\times \{t\btau : t\in (0,\infty)\}$. Setting
$$\hat{X}_n:=\{(\bp,\bomega)\in X_n\mid |\bomega|_{\ell^2}^2=1\},$$
this motivates the following definition:
\begin{definition}\label{hatH.def}
Define  the  
 \textit{projectivized energy} $\hat\cH: \hat{X}_n\to (0,\infty)$ by
$$\hat\cH(\bp,\bomega):=\max_{t\in (0,\infty)}\cH(\bp,t\bomega)=\mathcal{H}(\bp,e^{2\pi \cH(\bp,\bomega)-1/2}\bomega).$$
\end{definition}
Equivalently, by an application of \eqref{eq:scaleTau}, we can write
\begin{equation}\label{hath.gibbs}
\hat\cH(\bp,\bomega)=\frac{1}{4\pi e}e^{4\pi \cH(\bp,\bomega)},
\end{equation}
so that $\hat\cH$ resembles a kind of Gibbs measure on $\hat{X}_n$. It follows from the discussion above that $(\bp,\btau)\in X_n$ is a critical point for $\cH$ if and only if $(\bp,\btau)$ satisfies \eqref{crit.h.val} and $(\bp,\btau/|\btau|_{\ell^2})\in \hat{X}_n$ is a critical point for $\hat\cH$, so that the critical point theory for the functionals $\mathcal{H}$ can be reformulated in terms of the positive energy functional $\hat\cH$. In Section \ref{sect:minH} below, we take the first step toward developing the variational theory for $\hat\cH$, showing that the minimization problem for $\hat\cH$ is well-posed on $\hat{X}_n$, and gives rise to critical points of $\cH$ whose behavior in the large $n$ limit is consistent with the conditions of Theorem \ref{thm:doubling.precise}.

\section{Proof of Theorem \ref{thm:doubling.precise}}\label{sect:Proofdoubling.precise}

We now come to the proof of Theorem \ref{thm:doubling.precise}, checking carefully that the hypotheses ensure the existence of a family of linearized doubling solutions satisfying all of the criteria in \cite[Sections 3-5]{kapouleasMcGrath2023generalDoubling} necessary to deduce the existence of a minimal doubling as in \cite[Theorem 5.7]{kapouleasMcGrath2023generalDoubling}.

\subsection{From nondegenerate critical points to families of configurations}

To begin, let $(\bp,\btau)$ be as in the statement of Theorem \ref{thm:doubling.precise}, and note that a standard Vitali covering argument forces
$$\delta=\delta(\bp)<C(\Sigma,g)n^{-1/2}\leq \frac{1}{2}\InjRad(N,g)$$
for $n\geq C(g,\Sigma)$ sufficiently large.

Next, consider the neighborhood $U\subset X_n$ of $(\bp,\btau)$ given by
$$U:=B_{\delta}(p_1)\times \cdots\times B_{\delta}(p_n)\times (\tau_1/2,2\tau_1)\times \cdots\times (\tau_n/2,2\tau_n),$$
so that for every $(\bq,\bsigma)\in U$, we have $d(q_i,p_i)<\delta$, 
$$\min_{i\neq j}d(q_i,q_j)\geq 2\delta,$$
and
$$\frac{1}{2n}\leq \frac{\sigma_i}{\bar{\btau}}\leq 2n.$$
We will use exponential coordinates centered at the points $p_i$ to fix a coordinate system on each $B_{\delta}(p_i)$, giving us an identification 
\begin{equation}\label{z.def}
Z_{\bq}:\mathcal{P}_{\bq}\to\mathcal{P}_{\bp}\cong (\mathbb{R}^2)^n\times \mathbb{R}^n=:\mathcal{P}
\end{equation}
that preserves $\ell^p$ norms up to a universal constant depending only on $\Sigma$.

Now, translating \cite[Definition 3.10]{kapouleasMcGrath2023generalDoubling} to our notation, we define the mismatch $\mathcal{M}:U\to \mathcal{P}$ by
$$\mathcal{M}(\bq,\bsigma)=(\mathcal{M}_{q_i}(\bq,\bsigma),\mathcal{M}_{\sigma_i}(\bq,\bsigma)):=-\pi(\frac{1}{\sigma_i}\nabla^{q_i}\mathcal{H}(\bq,\bsigma),\nabla^{\sigma_i}\mathcal{H}(\bq,\bsigma)),$$
so that $\mathcal{M}(\bp,\btau)=0$, and Proposition \ref{prop:criticalPointH} gives
\begin{eqnarray*}
    \mathcal{M}_{q_i}(\bq,\bsigma)&=&\left(\sum_{j\neq i}2\sigma_j(\nabla \phi_{q_j})(q_i)+2\sigma_i \nabla R_{q_j}(q_i)\right)\\
    &=&-2\pi\nabla(\phi+\frac{\sigma_i}{2\pi}\log d_{q_i})(q_i)=\nabla(-2\pi \phi-\sigma_i\log d_{q_i})
\end{eqnarray*}
and
\begin{eqnarray*}
    \mathcal{M}_{\sigma_i}(\bq,\bsigma)&=&-\pi\sum_{j\neq i}2\sigma_j G(q_i,q_j)+2\sigma_iR(q_i,q_i)-\frac{\sigma_i}{2\pi}\log(\sigma_i/2)\\
    &=&-2\pi(\phi+\frac{\sigma_i}{2\pi}\log d_{q_i})(q_i)+\sigma_i\log(\sigma_i/2)\\
    &=&(-2\pi \phi-\sigma_i\log(2d_{q_i}/\sigma_i))(q_i)
\end{eqnarray*}
for $\phi=\phi_{\bq,\bsigma}$ solving $J\phi=\sum_i \sigma_i \delta_{q_i}$, where we have set $d_{q_i}(x):=d(q_i,x).$

By an easy application of the estimates in Lemma \ref{green.lem} and the assumptions of the theorem, we can then prove $C^2$ bounds of the following form for $\mathcal{M}$ on $U$.

\begin{lem}
    There is a constant $C(g,\Sigma)$ such that on $U$,
$$\bar{\btau}^{-1}|\nabla^{q_j}\nabla^{q_k}\mathcal{M}|+|\nabla^{q_k}\nabla^{\sigma_j}\mathcal{M}|+\bar{\btau}|\nabla^{\sigma_k}\nabla^{\sigma_j}\mathcal{M}|\leq \frac{Cn}{\delta^3}$$
with respect to the fixed exponential coordinate system. 
\end{lem}
\begin{proof}
Writing
$$\mathcal{M}_{q_i}(\bq,\bsigma)=-\pi\left(2\sum_{j\neq i}\sigma_j\nabla^{q_j}G(q_i,q_j)+\sigma_i\nabla R_D(q_i)\right)$$
and
$$\mathcal{M}_{\sigma_i}(\bq,\bsigma)=-2\pi\left(\sum_{j\neq i}\sigma_jG(q_i,q_j)+\sigma_i R_D(q_i)-\frac{\sigma_i}{2\pi}\log(\sigma_i/2)\right),$$
a straightforward computation gives
$$|\nabla^{q_j}\nabla^{q_k}\mathcal{M}_{q_i}(\bq,\bsigma)|\leq 2\pi \sum_{\ell\neq i}\sigma_\ell|\nabla^3G(q_i,q_\ell)|+\pi\sigma_i|\nabla^3R_D(q_i)|,$$
$$|\nabla^{q_j}\nabla^{q_k}\mathcal{M}_{\sigma_i}(\bq,\bsigma)|\leq 2\pi\sum_{\ell\neq i}\sigma_\ell|\nabla^2G(q_i,q_\ell)|+\pi\sigma_i|\nabla^2R_D(q_i)|,$$
which together with Lemma \ref{green.lem}, Remark \ref{robin.rk}, and the estimates $\sigma_i\leq 2\tau_i$ and $d(q_i,q_j)\geq 2\delta$, give an estimate of the desired form
$$|\nabla^{q_j}\nabla^{q_k}\mathcal{M}|(\bq,\bsigma)\leq C(g,\Sigma)n\bar{\btau}\delta^{-3}.$$
Similarly, direct computation and the same collection of estimates gives a bound of the form
$$|\nabla^{q_j}\nabla^{\sigma_k}\mathcal{M}(\bq,\bsigma)|\leq C(g,\Sigma)\delta^{-2},$$
and since
$$\nabla^{\sigma_k}\nabla^{\sigma_j}\mathcal{M}_{q_i}(\bq,\bsigma)=0$$
and
$$\nabla^{\sigma_k}\nabla^{\sigma_j}\mathcal{M}_{\sigma_i}(\bq,\bsigma)=\frac{1}{\sigma_i}$$
when $i=j=k$ and vanishes otherwise, we see that
$$|\nabla^{\sigma_k}\nabla^{\sigma_j}\mathcal{M}(\bq,\bsigma)|=\frac{1}{\sigma_i}\leq \frac{2n}{\tau_i}\leq \frac{2n^2}{\bar{\btau}}.$$
Recalling that $n<C(g,\Sigma)\delta^{-2}$, the desired estimate follows.
\end{proof}

As an immediate corollary, we deduce that
$$|\nabla^{p_i}\mathcal{M}(\bq,\bsigma)-\nabla^{p_i}\mathcal{M}(\bp,\btau)|\leq \frac{Cn}{\delta^3}(|\bsigma-\btau|+\bar{\btau}d(\bp,\bq))$$
and
$$\bar{\btau}|\nabla^{\tau_i}\mathcal{M}(\bq,\bsigma)-\nabla^{\tau_i}\mathcal{M}(\bp,\btau)|\leq \frac{Cn}{\delta^3}(|\bsigma-\btau|+\bar{\btau}d(\bp,\bq)),$$
for all $(\bq,\bsigma)\in U$.
And by another application of the fundamental theorem of calculus, we see that for any $(\bq,\bsigma),(\bq',\bsigma')\in U$, computing once again in the fixed exponential coordinates,
\begin{eqnarray*}
    |\mathcal{M}(\bq,\bsigma)-\mathcal{M}(\bq',\bsigma')-\sum_{i=1}^n\nabla_{q_i-q_i'}^{p_i}\mathcal{M}(\bp,\btau)+\nabla^{\tau_i}_{\sigma_i-\sigma_i'}\mathcal{M}(\bp,\btau)|\\
    \leq \frac{Cn}{\delta^3}(d(\bq,\bq')+\bar{\btau}^{-1}|\bsigma-\bsigma'|)[|\bsigma-\btau|+|\bsigma'-\btau|+\bar{\btau}(d(\bp,\bq)+d(\bp,\bq'))].
\end{eqnarray*}

Now, recalling that the composition
$$T(\bp,\btau):=S_{\btau}\circ D^2\mathcal{H}(\bp,\btau)\circ S_{\btau}$$
is symmetric for the $\ell^2$ inner product on $\mathcal{P}$, taking
$$0<\mu\leq\min \{|\lambda|\mid \lambda \in Spec(T(\bp,\btau))\},$$
it follows that
$$|T(\bp,\btau)[\bv,\bs]|_{\ell^2}\geq \mu |[\bv,\bs]|_{\ell^2},$$
or equivalently,
\begin{eqnarray*}
    |S_{\btau}\circ D^2F(\bp,\btau)[\bv,\bs]|_{\ell^2}&\geq & \mu |S_{\btau}^{-1}[\bv,\bs]|_{\ell^2}\\
    &\geq &\frac{\mu}{2}(\frac{\bar{\btau}}{n}|\bv|_{\ell^{\infty}}+|\bs|_{\ell^{\infty}}).
\end{eqnarray*}
On the other hand, since $(\bp,\btau)$ is a critical point for $\mathcal{H}$, we see that 
$$S_{\btau}\circ D^2\mathcal{H}(\bp,\btau)=-\frac{1}{\pi}\nabla \mathcal{M}(\bp,\btau),$$
so that 
\begin{equation}
    |\nabla \mathcal{M}(\bp,\btau)[\bv,\bs]|_{\ell^{\infty}}\geq \frac{\pi\mu}{2n^{3/2}} (\bar{\btau}|\bv|_{\ell^{\infty}}+|\bs|_{\ell^{\infty}}).
\end{equation}

Combining this with the computations above, we see that for all $(\bq,\bsigma),(\bq',\bsigma')\in U$ satisfying
$$\frac{Cn}{\delta^3\bar{\btau}}(|\bsigma-\btau|+|\bsigma'-\btau'|+\bar{\btau}(d(\bq,\bp)+d(\bq',\bp))<\frac{\pi\mu}{4n^{3/2}},$$
we have
$$|\mathcal{M}(\bq,\bsigma)-\mathcal{M}(\bq',\bsigma')|\geq \frac{\pi\mu}{4n^{3/2}}(\bar{\btau}d(\bq,\bq')+|\bsigma-\bsigma'|).$$
In other words, defining a metric $d_{\btau}$ on $U$ by
$$d_{\btau}((\bq,\bsigma),(\bq',\bsigma')):=|\bsigma-\bsigma'|+\bar{\btau}d(\bq,\bq'),$$
we see that $\mathcal{M}$ is bi-Lipschitz on the ball $B^{d_{\btau}}_{c\mu\delta^3\bar{\btau}n^{-5/2}}(\bp,\btau)$ for a suitable small constant $c(g,\Sigma)>0$, with
$$Lip(\mathcal{M}^{-1}; \mathcal{M}(B^{d_{\btau}}_{c\mu\delta^3\bar{\btau}n^{-5/2}}(\bp,\btau)))\leq \frac{2n^{3/2}}{\mu}$$
with respect to the $\ell^{\infty}$ metric on $\mathcal{P}$. By elementary arguments, it follows that the image $\mathcal{M}(B^{d_{\btau}}_{c\mu\delta^3\bar{\btau}n^{-5/2}}(\bp,\btau))$ must contain the $\ell^{\infty}$ ball $B_{c'\mu^2\delta^3\bar{\btau}n^{-4}}(0)\subset \mathcal{P}$, where $c'(g,\Sigma)=c(g,\Sigma)/4$.

Now, by assumption \ref{quant.nondeg} of Theorem \ref{thm:doubling.precise}, we can take $\mu=\bar{\btau}^{1/1000}$ in the arguments above, to arrive at the following conclusion.

\begin{prop}\label{fam.prop}
Under the hypotheses of Theorem \ref{thm:doubling.precise}, for $n\geq C(\Sigma,g)$, there exists $c(g,\Sigma)>0$ and a neighborhood $V\subset X_n$ of $(\bp,\btau)$ such that
$$|\bsigma-\btau|+\bar{\btau}d(\bp,\bq)\leq \bar{\btau}^{1+1/1000}\delta^3n^{-5/2}\text{ for all }(\bq,\bsigma)\in V$$
and $\mathcal{M}$ maps $V$ homeomorphically onto the ball
$$B_{c\delta^3n^{-4}\bar{\btau}^{1+1/500}}(0)\subset \mathcal{P}.$$
\end{prop}

Taking $B_{\mathcal{P}}:=B_{c\delta^3\bar{\btau}^{1+1/500}n^{-4}}(0)\subset \mathcal{P}$, our goal now is to show that $\mathcal{M}^{-1}: B_{\mathcal{P}}\to V$ parametrizes a family of LD solutions satisfying Assumption 5.2 of \cite[Section 5]{kapouleasMcGrath2023generalDoubling}, to which the fixed point argument in the proof of \cite[Theorem 5.7]{kapouleasMcGrath2023generalDoubling} can be applied to produce a minimal doubling of $\Sigma$ with $n$ catenoidal necks.

\subsection{Ingredients for the doubling construction}

Let $V$ be as in Proposition \ref{fam.prop}, and for any $(\bq,\bsigma)\in V$, letting $\phi=\phi_{\bq,\bsigma}$ solve
$$J\phi=\sum_{i=1}^n\sigma_i\delta_{q_i},$$
where $J$ is the Jacobi operator of $\Sigma\subset N$, we see that $\varphi=-2\pi \phi$ defines a linearized doubling in the sense of \cite[Definition 3.6]{kapouleasMcGrath2023generalDoubling} with $L=\{q_1,\ldots,q_n\}$ and $\bsigma: L\to (0,\infty)$ given by $\bsigma(q_i)=\sigma_i$.

For each $i\neq j$, note that
$$d(q_i,q_j)\geq d(p_i,p_j)-\frac{\delta^3 \bar{\btau}^{1/1000}}{n^{5/2}}\geq \delta\cdot (20-\delta^2\bar{\btau}^{1/1000} n^{-5/2}),$$
where $\delta=\delta(\bp)$. In particular, since assumption \ref{del.low} of Theorem \ref{thm:doubling.precise} implies
$$\bar{\btau}^{1/1000}<\delta^{100}\leq (C(\Sigma,g)n^{-1/2})^{100}=C'(\Sigma,g)n^{-50},$$
we see that for $n\geq C(g,\Sigma)$ sufficiently large,
$$d(q_i,q_j)>18\delta$$
for any $i\neq j$, and \cite[Convention 3.8]{kapouleasMcGrath2023generalDoubling} holds with $\delta_{q_i}=\delta(\bp)$.

Next, we check that the family of linearized doublings associated to $V\subset X_n$ satisfy the assumptions of \cite[Convention 3.15]{kapouleasMcGrath2023generalDoubling}, with 
$$\alpha:=\frac{6}{500}=\frac{3}{250}.$$
To verify items (i)-(iii) of \cite[Convention 3.15]{kapouleasMcGrath2023generalDoubling}, it remains to check that 
$$9\sigma_i^{\alpha}<\sigma_i^{\alpha/100}<\delta$$
and
$$\sigma_i\leq \sigma_j^{1-\alpha/100}$$
for all $1\leq i,j\leq n$. For $n\geq C(g,\Sigma)$ sufficiently large, both follow in a straightforward way from the estimates
$$\frac{1}{2n}\leq \frac{\sigma_i}{\bar{\btau}}\leq 2n$$
and the bounds $\delta>\bar{\btau}^{1/100,000}$ and $\bar{\btau}<n^{-100,000}$ given by the hypotheses of Theorem \ref{thm:doubling.precise}. 

To check items (iv)-(vi) of \cite[Convention 3.15]{kapouleasMcGrath2023generalDoubling}, consider the solution $\phi$ of
$$J\phi=\sum_{i=1}^n\sigma_i\delta_{q_i}$$
for $(\bq,\bsigma)\in V$, so that
$$\phi(x)=\sum_{i=1}^n\sigma_i G(q_i,x).$$
As a consequence of Lemma \ref{green.lem}, observe that
$$|\nabla^k\phi|(x)\leq C_k(g,\Sigma)\sum_{i=1}^n\frac{|\sigma_i|}{d(q_i,x)^k}$$
for each $k\in \mathbb{N}$, while 
$$|\phi(x)|\leq C(g,\Sigma) \sum_{i=1}^n|\sigma_i|(|\log d(q_i,x)|+1).$$
On $\partial D_{\delta}(q_i)$, it follows that 
$$|\nabla^k\phi|(x)\leq C_k(g,\Sigma)n^2\bar{\btau}\delta^{-k},$$
so 
\begin{eqnarray*}
    \delta^{-2}\|\phi\|_{C^3(\partial D_{\delta}(q_i))}&\leq & C(\Sigma,g)\bar{\btau}n^2\delta^{-5}\\
    &\leq & C \bar{\btau} \bar{\btau}^{-2/100,000}\bar{\btau}^{-5/100,000}\\
    &\leq &C \bar{\btau}^{1-1/10,000}\leq \sigma_i^{1-1/750}=\sigma_i^{1-\alpha/9}
\end{eqnarray*}
for $n\geq C(g,\Sigma)$ sufficiently large, using repeatedly the fact that $n^{-1}\sigma_i\leq \bar{\btau}\leq n\sigma_i$ and $\min\{n,\delta\}>\bar{\btau}^{1/100,000}$; this confirms item (iv) of \cite[Convention 3.15]{kapouleasMcGrath2023generalDoubling}. Similarly, to check \cite[Convention 3.15(v)]{kapouleasMcGrath2023generalDoubling}, we combine the hypotheses of Theorem \ref{thm:doubling.precise} with the estimate $\bar{\btau}\leq n\sigma_i$ in a straightforward way to confirm that
$$\|\phi\|_{C^4(\Sigma\setminus \bigcup_{i=1}^n B_{\sigma_i^{\alpha}}(q_i))}\leq Cn^{4\alpha}\bar{\btau}^{-4\alpha}\bar{\btau}\leq \sigma_i^{8/9}$$
for any $1\leq i\leq n$, giving \cite[Convention 3.15(v)]{kapouleasMcGrath2023generalDoubling}.

The final item (vi) of \cite[Convention 3.15]{kapouleasMcGrath2023generalDoubling}, giving lower bounds on the height of the doubling away from the necks, is more subtle, and this is where the assumption \eqref{2eigen.bd} of Theorem \ref{thm:doubling.precise} comes into play. For $(\bq,\bsigma)\in V$ and $\phi$ as above, let $(\bv, \bs)=\mathcal{M}(\bq,\bsigma)$, so that the function
\begin{eqnarray*}
    \hat{\phi}_i(x)&:=&\phi(x)+\frac{\sigma_i}{2\pi}\log (2d(q_i,x)/\sigma_i)\\
    &=&\sum_{j\neq i}\sigma_jG(q_j,x)+\sigma_iR(q_i,x)-\frac{\sigma_i}{2\pi}\log(\sigma_i/2)
\end{eqnarray*}
satisfies
$$|\hat{\phi}_i(q_i)|=\pi|s_i|\leq c\delta^3\bar{\btau}^{1+1/500}n^{-4}$$
and
$$|d\hat{\phi_i}(q_i)|=\pi|v_i|\leq c\delta^3\bar{\btau}^{1+1/500}n^{-4}.$$
Now, since $d(q_i,q_j)\geq 18\delta$ whenever $i\neq j$, by definition of $\hat{\phi}_i$ and an application of Lemma \ref{green.lem}, we have an estimate of the form
\begin{equation}\label{hatphi.hold}
[d\hat{\phi}_i]_{C^{0,1/2}(B_{\delta}(q_i))}\leq C\left(\sum_{j\neq i}\frac{\sigma_j}{\delta^{3/2}}+\sigma_i\right)\leq C\delta^{-3/2}n\bar{\btau},
\end{equation}
which together with the preceding estimates implies that
$$|d\hat{\phi}_i(x)|\leq c\delta^3\bar{\btau}^{1+1/500}n^{-4}+C\delta^{-3/2}n\bar{\btau} d(x,q_i)^{1/2}$$
and
$$|\hat{\phi}_i(x)|\leq c\delta^3\bar{\btau}^{1+1/500}n^{-4}+C\delta^{-3/2}n\bar{\btau} d(x,q_i)^{3/2}$$
for all $x\in B_{\delta}(q_i)$. In particular, it follows from these estimates and the hypotheses of Theorem \ref{thm:doubling.precise} that
\begin{equation}\label{phi.floor}
    \phi\leq -\frac{\sigma_i}{2\pi}\log(2)+c\delta^3\bar{\btau}^{1+1/500}n^{-4}+C\delta^{-3/2}n\bar{\btau}\sigma_i^{3/2}\leq -\frac{\bar{\btau}}{8\pi n}\text{ on }\partial B_{\sigma_i}(q_i)
\end{equation}
and
$$\phi \geq -\frac{\sigma_i}{2\pi}\log(2/3)-\delta^3\bar{\btau}^{1+1/500}n^{-4}-C\delta^{-3/2}n\bar{\btau}\sigma_i^{3/2}>0\text{ on }\partial B_{\sigma_i/3}(q_i) $$
for $n\geq C(g,\Sigma)$ sufficiently large. 

Thus, the nodal set $\{\phi=0\}$ has a connected component contained in each of the annuli $B_{\sigma_i}(q_i)\setminus B_{\sigma_i/3}(q_i)$, and together these curves bound a domain $\Omega\subset \Sigma\setminus \bigcup_{i=1}^n B_{\sigma_i/3}(q_i)$ such that $J\phi=0$ on $\Omega$ while $\phi\in W_0^{1,2}(\Omega)$. In other words, we see that $\phi$ is an eigenfunction for $J$ on $\Omega$ with Dirichlet boundary condition and eigenvalue $0$, and since assumption \eqref{2eigen.bd} of Theorem \ref{thm:doubling.precise} implies that
$$\lambda_2(\Omega)> \lambda_2(\Sigma\setminus \bigcup_{i=1}^n B_{\sigma_i/3}(q_i))>0,$$
it follows that $\phi$ must be the \emph{first} Dirichlet eigenfunction on $\Omega$. This, together with \eqref{phi.floor} already suffices to conclude that $\phi<0$ in the complement of the disks $B_{\sigma_i}(q_i)$, but to confirm \cite[Convention 3.15(vi)]{kapouleasMcGrath2023generalDoubling}, we need to make this more quantitative.

\begin{lem}\label{ht.bd}
    For $\phi$ as above and $n\geq C(g,\Sigma)$ sufficiently large, we have
    $$\phi\leq -\frac{\bar{\btau}}{n^2}\text{ on }\Sigma\setminus \bigcup_{i=1}^nB_{\sigma_i}(q_i).$$
\end{lem}
\begin{proof}
    Set $\Omega':=\Sigma\setminus \bigcup_{i=1}^nB_{\sigma_i}(q_i),$ and let $\chi_1$ be a positive first eigenfunction for $J$ on the full surface $\Sigma$, normalized so that
    $$1\leq \chi_1\leq C(\Sigma,g).$$
    For each $t\in \mathbb{R}$, consider the quantity
    $$f(t):=\max_{\Omega'}(\phi+t\chi_1).$$
    Evidently $f(t)$ is an increasing function of $t$ with $f(0)<0$ and $\lim_{t\to\infty}f(t)=+\infty$, so there is some $t_0>0$ for which $f(t_0)=0$. 

    Now, if $\phi+t_0\chi_1$ achieves its maximum
    $$0=f(t_0)=(\phi+t_0\chi_1)(x_0)$$
    at some interior point $x_0\in\Omega'$, then we must have
    $$0\leq d^*d (\phi+t_0\chi_1)(x_0)=J(\phi+t_0\chi_1)(x_0)=t_0\lambda_1(J)\chi_1(x_0),$$
    but since $t_0>0$, $\chi_1(x_0)>0$ and $\lambda_1(J)<0$, this is impossible. 

    Thus, there must be some $x_0\in \partial\Omega'$ at which
    $$(\phi+t_0\chi_1)(x_0)=0.$$
    But then, since $\phi|_{\partial\Omega'}\leq -\frac{\bar{\btau}}{8\pi n}$ by \eqref{phi.floor}, and $\chi_1(x_0)\leq C$, it follows that
    $$\frac{\bar{\btau}}{8\pi n}\leq Ct_0.$$
    Finally, by definition of $t_0$ and the fact that $\chi_1\geq 1$ we deduce that
    $$\phi\leq -t_0\chi_1\leq -C^{-1}\frac{\bar{\btau}}{8\pi n},$$
    from which the desired estimate follows for $n\geq 8\pi C$.  
\end{proof}
Writing $\varphi=-2\pi \phi$, it then follows that
$$\varphi\geq 2\pi \frac{\bar{\btau}}{n^2}\geq \sigma_i^{1-1/250}$$
on $\Sigma\setminus \bigcup_{i=1}^n B_{\sigma_i}(q_i)$, and in particular on the smaller set $\Sigma\setminus \bigcup_{i=1}^n B_{\sigma_i^{\alpha}}(q_i)$, for every $1\leq i\leq n$, confirming \cite[Convention 3.15(vi)]{kapouleasMcGrath2023generalDoubling}. Summarizing, we now have the following.

\begin{prop}
    For $n\geq C(g,\Sigma)$ sufficiently large, the linearized doublings associated to $(\bq,\bsigma)\in V$ satisfy all the conditions (i)-(vi) of \cite[Convention 3.15]{kapouleasMcGrath2023generalDoubling} with $\alpha=\frac{3}{250}$.
\end{prop}

Next, we observe that the construction described in \cite[Remark 3.12]{kapouleasMcGrath2023generalDoubling} associates to each $(\bq,\bsigma)\in V$ an ``obstruction space" 
$$\widehat{\mathcal{K}}[\bq]=\bigoplus_{i=1}^n\widehat{\mathcal{K}}[q_i]\subset C^{\infty}(\Sigma)$$
satisfying the conditions of \cite[Assumption 3.11]{kapouleasMcGrath2023generalDoubling}. Namely, this gives us a family of spaces $\widehat{\mathcal{K}}[q_i]\subset C_c^{\infty}(B_{4\delta}(q_i))$ such that
$$J(\widehat{\mathcal{K}}[q_i])\subset C_c^{\infty}(B_{4\delta}(q_i)\setminus B_{\delta/4}(q_i)),$$
for which the first-order Taylor expansion at each $q_i$ gives an isomorphism
$$\mathcal{E}_{\bq}: \widehat{\mathcal{K}}[\bq]\to \mathcal{P}_{\bq}$$
whose inverse satisfies a uniform bound of the form $\|\mathcal{E}_{\bq}\|^{-1}\leq C\delta^{-2-\beta}$ with respect to the $C^{2,\beta}$ norm on $\widehat{\mathcal{K}}[\bq]$.

Now, given $(\bq,\bsigma)\in V$ and $\bkappa\in B_{\mathcal{P}}$, since 
$$|\bkappa|\leq c\delta^3\bar{\btau}^{1+1/500}n^{-4}<\delta^3n^{-2}\frac{\bar{\btau}}{2n}(\bar{\btau}/2n)^{\alpha/6}<\sigma_i^{1+\alpha/6}$$
for each $1\leq i\leq n$, we see that the conditions of \cite[Definition 3.17]{kapouleasMcGrath2023generalDoubling} are satisfied. Then, to every pair $\bzeta,\bkappa\in B_{\mathcal{P}}$, setting $(\bq,\bsigma)=\mathcal{M}^{-1}(\bzeta)\in V$, the construction of \cite[Definition 3.17]{kapouleasMcGrath2023generalDoubling} associates an initial doubling surface $M[\bzeta,\bkappa]$. Near each point $q_i$, the surface $M[\bzeta,\bkappa]$ coincides with the catenoid $K[q_i,\sigma_i,\kappa_i]$ centered at $q_i$ with waist of radius $\sigma_i$ and height and tilt relative to $T_{q_i}\Sigma$ determined by $\kappa_i=(v_i,s_i)\in T_{q_i}\Sigma\times \mathbb{R}$ as in \cite[Section 2]{kapouleasMcGrath2023generalDoubling}. Sufficiently far from $\{q_1,\ldots,q_n\}$, $M[\bzeta,\bkappa]$ is given by the two-sheeted graph associated with the functions
$$\mp2\pi \phi_{\bq,\bsigma} \mp \mathcal{E}_{\bq}^{-1}(\bzeta)+\mathcal{E}_{\bq}^{-1}(\bkappa),$$
and the catenoidal and graphical regions are then glued together as two-sheeted graphs over the annuli $B_{4\sigma_i^{\alpha}}(q_i)\setminus B_{\sigma_i^{\alpha}}(q_i)$ using a simple linear interpolation by cutoff functions. As discussed in \cite[Section 3]{kapouleasMcGrath2023generalDoubling}, these surfaces $M[\bzeta,\bkappa]$ then define a family of embedded surfaces in $N$ of genus $2g_{\Sigma}+n-1$ (or nonorientable genus $2g_{\Sigma}+2n-2$, if $\Sigma$ is nonorientable).

\subsection{Completing the Kapouleas-McGrath fixed point argument}

We claim now that the family of initial doubling surfaces $M[\bzeta,0]$ parametrized by $\bzeta\in B_{\mathcal{P}}\subset \mathcal{P}=\mathcal{P}_{\bp}$ as above satisfy the conditions of \cite[Assumption 5.2]{kapouleasMcGrath2023generalDoubling}. Items (iii)-(vi) are immediate from the discussion in the preceding sections, and (vii) is similarly straightforward, with $Z_{\bzeta}$ given by $Z_{\bzeta}:=Z_{\bq}: \mathcal{P}_{\bq}\to \mathcal{P}_{\bp}$
as in \eqref{z.def}, where $(\bq,\bsigma)=\mathcal{M}^{-1}(\bzeta)$. Crucially, note that \cite[Assumption 5.2(e)]{kapouleasMcGrath2023generalDoubling} holds automatically in this setting, since $\bzeta-Z_{\bzeta}(\mathcal{M}(\bq,\bsigma))=0$ by definition for $(\bq,\bsigma)=\mathcal{M}^{-1}(\bzeta)$.

Likewise, items (c) and (d) of \cite[Assumption 5.2]{kapouleasMcGrath2023generalDoubling} follow readily from the estimates of the previous sections, so all that remains is to identify a family of diffeomorphisms $\mathcal{F}_{\bzeta}:\Sigma\to \Sigma$ satisfying items (i), (ii), (a), and (b). This is also straightforward: setting $(\bq,\bsigma)=\mathcal{M}^{-1}(\bzeta)$, recall that the disks $B_{9\delta}(p_i)$ are disjoint and 
$$d(p_i,q_i)<\delta^3\bar{\btau}^{1/1000}<\delta^{103},$$
so we can easily use exponential maps to produce a diffeomorphism
$$F_i: B_{5\delta}(p_i)\to B_{5\delta}(q_i)$$
satisfying $F_i(p_i)=q_i$, $\|F_i\|_{C^{10}}(B_{3\delta}(p_i))\leq C(\Sigma,g)$, and 
$$|F_i(x)-x|\leq C(\Sigma,g)\delta^{103}\text{ for all }x\in B_{5\delta}(q_i).$$
By an elementary interpolation argument, we deduce moreover that
$$\|F_i-Id\|_{C^4(B_{5\delta}(p_i))}\leq C\delta^{10}.$$
We then define $\mathcal{F}_{\bzeta}:\Sigma\to \Sigma$ by setting
$$\mathcal{F}_{\bzeta}(x):=x\text{ for }x\in \Sigma\setminus \bigcup_{i=1}^nB_{9\delta}(p_i),$$
and, with respect to exponential coordinates centered at $p_i$,
$$\mathcal{F}_{\bzeta}(x):=x+\chi_i(x)(F_i(x)-x)\text{ for }x\in B_{9\delta}(p_i),$$
where $\chi_i\in C_c^{\infty}(B_{5\delta}(p_i))$ is a cutoff supported in $B_{5\delta}(p_i)$ such that $\chi_i\equiv 1$ on $B_{3\delta}(p_i)$ and 
$$\|\chi_i\|_{C^4}\leq C\delta^{-4},$$
so that
$$\|\chi_i\cdot (F_i-Id)\|_{C^4}\leq \|\chi_i\|_{C^4}\|F_i-Id\|_{C^4}\leq C\delta^6.$$
It is then straightforward to check that the diffeomorphisms $(\mathcal{F}_{\bzeta})_{\bzeta\in B_{\mathcal{P}}}$ satisfy all the conditions (i), (ii), (a), and (b) of \cite[Assumption 5.2]{kapouleasMcGrath2023generalDoubling}.

At this point, we could almost conclude by quoting \cite[Theorem 5.7]{kapouleasMcGrath2023generalDoubling} directly; however, we find it useful to recall some details of the fixed point argument used in its proof, in part because our setup and notation differ slightly from that of \cite{kapouleasMcGrath2023generalDoubling}, but more importantly, to emphasize the role of the quantitative nondegeneracy hypothesis \eqref{quant.nondeg} of Theorem \ref{thm:doubling.precise} in the argument.

First, given $(\bzeta,\bkappa)\in B_{\mathcal{P}}\times B_{\mathcal{P}}$ and $\beta\in (0,1)$, we recall the weighted H\"older norms $\|\cdot\|_{k,\beta,\gamma,\gamma';\Sigma}$ defined in \cite[Definition 4.2]{kapouleasMcGrath2023generalDoubling}, and introduce the notation
$$\|u\|_{0,\beta;M[\bzeta,\bkappa]}':=\|u\|_{0,\beta,-1/2,-3/2;M[\bzeta,\bkappa]}$$
for $u\in C^{0,\beta}(M[\bzeta,\bkappa])$ and
$$\|u\|_{2,\beta; M[\bzeta,\bkappa]}':=\|u\|_{2,\beta,3/2,1/2;M[\bzeta,\bkappa]}$$
for $u\in C^{2,\beta}(M[\bzeta,\bkappa])$. Note that \cite[Lemma 5.5]{kapouleasMcGrath2023generalDoubling} supplies for $(\bzeta,\bkappa)\in B_{\mathcal{P}}\times B_{\mathcal{P}}$ a continuous family of diffeomorphisms
$$\mathcal{F}_{\bzeta,\bkappa}: M[0,0]\to M[\bzeta,\bkappa]$$
such that
\begin{equation}
    \frac{1}{4}\|u\|_{2,\beta; M[\bzeta,\bkappa]}'\leq \|u\circ \mathcal{F}_{\bzeta,\bkappa}\|_{2,\beta; M[0,0]}'\leq 4\|u\|_{2,\beta; M[\bzeta,\bkappa]}'
\end{equation}
and
\begin{equation}
    \frac{1}{4}\|E\|_{0,\beta; M[\bzeta,\bkappa]}'\leq \|E\circ \mathcal{F}_{\bzeta,\bkappa}\|_{0,\beta; M[0,0]}'\leq 4\|E\|_{0,\beta; M[\bzeta,\bkappa]}'.
\end{equation}

Next, recalling that $\bzeta$ is, by construction, the mismatch of the linearized doubling $-2\pi \phi_{\bq,\bsigma}$ associated to $M[\bzeta,\bkappa]$, observe that \cite[Lemma 4.6]{kapouleasMcGrath2023generalDoubling} in the present setting gives the mean curvature estimate
\begin{equation}\label{hm.exp}
    \|H_{M[\bzeta,\bkappa]}-w_{\bzeta,\bkappa}^+\circ P_+-w_{\bzeta,\bkappa}^-\circ P_-\|_{0,\beta}'\leq (n\bar{\btau})^{1+\alpha/3},
\end{equation}
where, in our slightly different notation,
\begin{equation}\label{w.def}
w_{\bzeta,\bkappa}^{\pm}:=J(\mathcal{E}_{\bq}^{-1}(\bkappa\mp\bzeta))\in C_c^{\infty}(\Sigma\setminus \bigcup_{i=1}^nB_{\delta/4}(q_i))
\end{equation}
and $P_{\pm}$ denotes the nearest-point projection from the graph of $\mp 2\pi \phi_{\bq,\bsigma}+\mathcal{E}_{\bq}^{-1}(\bkappa\mp\bzeta)$ over $\Sigma\setminus \bigcup_{i=1}^n B_{\delta/4}(q_i)$ onto $\Sigma\setminus \bigcup_{i=1}^n B_{\delta/4}(q_i)$.

Recalling the notation $\mathcal{K}[\bq]=J_{\Sigma}(\widehat{\mathcal{K}}[\bq])$, we note next that \cite[Corollary 4.21]{kapouleasMcGrath2023generalDoubling}, the key technical result of \cite[Section 4]{kapouleasMcGrath2023generalDoubling}, supplies in our setting a continuous family of maps
$$\mathcal{R}'_{\bzeta,\bkappa}: C^{0,\beta}(M[\bzeta,\bkappa])\to C^{2,\beta}(M[\bzeta,\bkappa])\times \mathcal{K}[\bq]\times \mathcal{K}[\bq]$$
of the form 
$$\mathcal{R}'(E)=(u,w_E^+,w_E^-),$$
where, denoting by $J_{M[\bzeta,\bkappa]}$ the Jacobi operator of $M[\bzeta,\bkappa]$,
\begin{equation}\label{r.def}
    J_{M[\bzeta,\bkappa]}u=E+w_E^+\circ P_++w_E^-\circ P_-
\end{equation}
and
\begin{equation}\label{r.schaud}
    \|u\|_{2,\beta}'+\|w_E^{\pm}\|_{C^{0,\beta}(\Sigma)}\leq C_{\beta}(\Sigma,g)\delta^{-4-2\beta}\|E\|_{0,\beta}'.
\end{equation}

Finally, we note that \cite[Lemma 5.1]{kapouleasMcGrath2023generalDoubling} applied to our setting implies that, for any $\bzeta,\bkappa\in B_{\mathcal{P}}$ and $u\in C^{2,\beta}(M[\bzeta,\bkappa])$, if $u\in C^{2,\beta}(M[\bzeta,\bkappa])$ satisfies
$$\|u\|_{2,\beta;M[\bzeta,\bkappa]}'\leq \bar{\btau}^{1+\alpha/4},$$
then the normal graph
$$M_u[\bzeta,\bkappa]=Graph_{M[\bzeta,\bkappa]}^{N,g}(u)$$
of $u$ over $M[\bzeta,\bkappa]$ is a well-defined embedded surface whose mean curvature $H_u$, identified in the obvious way with a function on $M[\bzeta,\bkappa]$, satisfies
\begin{equation}\label{hu.exp}
    \|H_u-H_{M[\bzeta,\bkappa]}-J_{M[\bzeta,\bkappa]}u\|_{0,\beta}'\leq C(\Sigma,g)n\bar{\btau}^{-\alpha/2}(\|u\|_{2,\beta; M[\bzeta,\bkappa]}')^2.
\end{equation}

We are now in a position to state the following.

\begin{thm}[cf. Theorem 5.7 of \cite{kapouleasMcGrath2023generalDoubling}]
Under the assumptions of Theorem \ref{thm:doubling.precise}, for $n\geq \tilde{n}(g,\Sigma)$ sufficiently large, there exists $(\bzeta,\bkappa)\in B_{\mathcal{P}}\times B_{\mathcal{P}}$ and $u\in C^{\infty}(M[\bzeta,\bkappa])$ satisfying
$$\|u\|_{2,1/2; M[\bzeta,\bkappa]}\leq \bar{\btau}^{1+3/1000}$$
such that the normal graph $\breve{M}:=M_u[\bzeta,\bkappa]$ is an embedded minimal surface of area $|\breve{M}|$ satisfying
$$||\breve{M}|-(2|\Sigma|-\pi \sum_{i=1}^n\tau_i^2)|\leq \frac{1}{n}\sum_{i=1}^n\tau_i^2.$$
\end{thm}
\begin{proof}
We recall the fixed point argument from the proof of \cite[Theorem 5.7]{kapouleasMcGrath2023generalDoubling}, keeping careful track of estimates to emphasize the role played by the quantitative nondegeneracy assumption \eqref{quant.nondeg}, which ultimately allows us to ensure that the map whose fixed point yields the minimal doubling maps the set $\{u\in C^{2,1/3}(M[0])\mid \|u\|_{2,1/2}'\leq \frac{1}{4}\bar{\btau}^{1+\alpha/4}\}\times B_{\cP}\times B_{\cP}$ into itself.

First, given $\bzeta,\bkappa\in B_{\mathcal{P}}$, recalling \eqref{hm.exp} and \eqref{r.def}, define
$$(\psi_{\bzeta,\bkappa},v^+_{\bzeta,\bkappa},v_{\bzeta,\bkappa}^-)
    :=-\mathcal{R}_{\bzeta,\bkappa}'(H_{M[\bzeta,\bkappa]}-w_{\bzeta,\bkappa}^+\circ P_+-w_{\bzeta,\bkappa}^-\circ P_-),$$
and observe that \eqref{hm.exp} and \eqref{r.schaud} together give an estimate of the form
\begin{equation}\label{psi.est}
    \|\psi\|_{2,\beta}'+\|v^{\pm}\|_{C^{0,\beta}}\leq C_{\beta}(\Sigma,g)\delta^{-4-2\beta}(n\bar{\btau})^{1+\alpha/3}.
\end{equation}

Now, given $u\in C^{2,\beta}(M[\bzeta,\bkappa])$ with $\|u\|_{2,\beta; M[\bzeta,\bkappa]}'\leq \bar{\btau}^{1+\alpha/4}$, define 
$$(\xi_{u,\bzeta,\bkappa}, h^+_{u,\bzeta,\bkappa},h^-_{u,\bzeta,\bkappa}):=-\mathcal{R}_{\bzeta,\bkappa}'(H_{u+\psi}-H_{M[\bzeta,\bkappa]}-J_{M[\bzeta,\bkappa]}(u+\psi)),$$
so that \eqref{hu.exp} and \eqref{r.schaud} together yield the estimate
\begin{equation}\label{xi.est}
    \|\xi\|_{2,\beta}'+\|h^{\pm}\|_{C^{0,\beta}}\leq C_{\beta}(\Sigma,g)\delta^{-4-2\beta}n\bar{\btau}^{-\alpha/2}(\|u+\psi_{\bzeta,\bkappa}\|_{2,\beta}')^2.
\end{equation}

Taking the sum of the preceding definitions and temporarily suppressing subscripts, note that
$$(\psi+\xi,v^++h^+,v^-+h^-)=\mathcal{R}'(J_M(u+\psi)+w^+\circ P_++w^-\circ P_--H_{u+\psi}),$$
so by the defining property \eqref{r.def} of $\mathcal{R}'$, we have
\begin{eqnarray*}
    J_M(\psi+\xi)&=&J_M(u+\psi)+w^+\circ P_++w^-\circ P_--H_{u+\psi}\\
    &&+(h^++v^+)\circ P_++(h^-+v^-)\circ P_-,
\end{eqnarray*}
and in particular,
$$H_{u+\psi}=J_M(u-\xi)+(w^++h^++v^+)\circ P_++(w^-+h^-+v^-)\circ P_-.$$
Thus, if we can find $\bzeta,\bkappa\in B_{\mathcal{P}}$ and $u\in C^{2,\beta}(M[\bzeta,\bkappa])$ such that
$$\xi_{u,\bzeta,\bkappa}=u$$
and
$$h_{u,\bzeta,\bkappa}^{\pm}+v_{\bzeta,\bkappa}^{\pm}+w_{\bzeta,\bkappa}^{\pm}=0,$$
then the surface $M_{\tilde{u}}[\bzeta,\bkappa]$ with $\tilde{u}=u+\psi_{\bzeta,\bkappa}$ would be minimal. 

Recalling \eqref{w.def}, note that this trio of conditions can be expressed equivalently as
$$\xi_{u,\bzeta,\bkappa}=u,$$
$$\mathcal{E}_{\bq}(J_{\Sigma}^{-1}(h_{u,\bzeta,\bkappa}^{\pm}+v_{\bzeta,\bkappa}^{\pm}))+\bkappa\mp \bzeta=0.$$
Rearranging slightly and writing 
$$\bmu^{\pm}(u,\bzeta,\bkappa):=\mathcal{E}_{\bq}(J_{\Sigma}^{-1}(h_{u,\bzeta,\bkappa}^{\pm}+v_{\bzeta,\bkappa}^{\pm})),$$
we see that $M_{u+\psi_{\bzeta,\bkappa}}[\bzeta,\bkappa]$ is minimal provided
$$\xi_{u,\bzeta,\bkappa}=u,$$
$$\bkappa=-\frac{1}{2}(\bmu^++\bmu^-),$$
and
$$\bzeta=\frac{1}{2}(\bmu^+-\bmu^-),$$
i.e., if $(u,\bzeta,\bkappa)$ is a fixed point of the map
$$\tilde{\mathcal{J}}: \{u\in C^{2,\beta}(M[\bzeta,\bkappa])\mid \|u\|_{2,\beta}'\leq \bar{\btau}^{1+\alpha/4}\}\times (B_{\mathcal{P}})^2\to C^{2,\beta}(M[\bzeta,\bkappa])\times (\mathcal{P})^2$$
given by 
$$\tilde{\mathcal{J}}(u,\bzeta,\bkappa):=(\xi_{u,\bzeta,\bkappa},\frac{\bmu^+(u,\bzeta,\bkappa)-\bmu^-(u,\bzeta,\bkappa)}{2},-\frac{\bmu^+(u,\bzeta,\bkappa)+\bmu^-(u,\bzeta,\bkappa)}{2}).$$
Using the diffeomorphisms $\mathcal{F}_{\bzeta,\bkappa}: M[0]\to M[\bzeta,\bkappa]$ to identify functions on $M[\bzeta,\bkappa]$ with functions on $M[0]$ while preserving the $\|\cdot\|_{0,\beta}'$ and $\|\cdot\|_{2,\beta}'$ norms up to a factor of $4$, we can then identify the family of maps $\tilde{\mathcal{J}}_{\bzeta,\bkappa}$ with a single continuous map
$$\mathcal{J}:\{u\in C^{2,\beta}(M[0])\mid \|u\|_{2,\beta}'\leq \frac{1}{4}\bar{\btau}^{1+\alpha/4}\}\times (B_{\mathcal{P}})^2\to C^{2,\beta}(M[0])\times (\mathcal{P})^2.$$

Fixing $\beta=\frac{1}{3}$ above, we can extend $\mathcal{J}$ to a continuous self-map 
$$\mathcal{J}: Y\to Y$$
of the Banach space
$$Y:=C^{2,1/3}(M[0])\times (\mathcal{P})^2,$$
of which
$$K:=\{u\in C^{2,1/3}(M[0])\mid \|u\|_{2,1/2}'\leq \frac{1}{4}\bar{\btau}^{1+\alpha/4}\}\times B_{\mathcal{P}}\times B_{\mathcal{P}}$$
is a compact, convex subset. Thus, if we can show that $\mathcal{J}(K)\subset K$, then the Schauder fixed-point theorem would supply a fixed point for $\mathcal{J}$ in $K$.

To this end, note that when $\|u\|_{2,\beta}'\leq \bar{\btau}^{1+\alpha/4}$, \eqref{psi.est} and \eqref{xi.est} together imply that 
$$\|v_{\bzeta,\bkappa}^{\pm}+h_{u,\bzeta,\bkappa}^{\pm}\|_{C^{0,1/2}}\leq C\delta^{-5}[(n\bar{\btau})^{1+\alpha/3}+n\bar{\btau}^{-\alpha/2}(\bar{\btau}^{2+\alpha/2}+\delta^{-10}(n\bar{\btau})^{2+2\alpha/3})]$$
and
$$\|\xi_{u,\bzeta,\bkappa}\|_{2,1/2}'\leq C\delta^{-5}n\bar{\btau}^{-\alpha/2}(\bar{\btau}^{2+\alpha/2}+\delta^{-10}(n\bar{\btau})^{2+2\alpha/3}).$$
Recalling that we have $\alpha=3/250$, and the hypotheses of Theorem \ref{thm:doubling.precise} give
$$\max\{n,\delta^{-1}\}<\bar{\btau}^{-1/100,000}=\bar{\btau}^{-\alpha/1200},$$
it follows from the preceding estimates and a little arithmetic that 
$$\|\xi_{u,\bzeta,\bkappa}\|_{2,1/2}'\leq C\bar{\btau}^{2-\alpha/200}$$
and, recalling the definition of $\bmu^{\pm}$ and invoking Schauder estimates for $J$, 
\begin{eqnarray*}
    |\bmu^{\pm}(u,\bzeta,\bkappa)|_{\ell^{\infty}}&\leq &\|J_{\Sigma}^{-1}(h_{u,\bzeta,\bkappa}^{\pm}+v_{\bzeta,\bkappa}^{\pm})\|_{C^{2,1/2}}\\
    &\leq &C \|h_{u,\bzeta,\bkappa}^{\pm}+v_{\bzeta,\bkappa}^{\pm}\|_{C^{0,1/2}}\\
    &\leq & C\bar{\btau}^{1+\alpha/3-2\alpha/300}.
\end{eqnarray*}
For $n\geq \nu(g,\Sigma)$ sufficiently large, it clearly follows that
$$\|\xi_{u,\bzeta,\bkappa}\circ \mathcal{F}_{\bzeta,\bkappa}\|_{2,\beta}'\leq \frac{1}{4}\bar{\btau}^{1+\alpha/4},$$
and
$$|\bmu^{\pm}|<\bar{\btau}^{1+\alpha/4}.$$
Finally, since $B_{\mathcal{P}}$ is defined to be the ball in $\mathcal{P}$ of radius
$$c\delta^3n^{-4}\bar{\btau}^{1+1/500}=c\delta^3n^{-4}\bar{\btau}^{1+\alpha/6},$$
we check that
$$c\delta^3n^{-4}\bar{\btau}^{1+\alpha/6}\geq c\bar{\btau}^{1+\alpha/6-7\alpha/1200}>\bar{\btau}^{1+\alpha/4}$$
for $n\geq \nu(g,\Sigma)$ sufficiently large as well, we deduce that indeed $\mathcal{J}(K)\subset K$, as desired. Hence, $\mathcal{J}$ has a fixed point in $K$, giving the desired minimal doubling $\breve{M}$.

The area estimate from \cite[Theorem 5.7]{kapouleasMcGrath2023generalDoubling} then gives
$$||\breve{M}|-(2|\Sigma|-\pi \sum_{i=1}^n\sigma_i^2)|\leq C\sum_{i=1}^n\sigma_i^2\sigma_i^{1/2}|\log\sigma_i|)$$
for some $\bsigma$ with 
$$|\sigma_i-\tau_i|\leq \bar{\btau}\delta^3\mu n^{-5/2}<\frac{\tau_i}{n^2},$$
so the stated area estimate for $\breve{M}$ in terms of $\btau$ follows immediately for $n\geq \tilde{n}(g,\Sigma)$ sufficiently large.
    
\end{proof}

The conclusions of Theorem \ref{thm:doubling.precise} now follow immediately, where we use the identity \eqref{crit.h.val} for critical points of $\cH$ to restate the area estimate in terms of $\cH(\bp,\btau)$.

\section{Minimizers for $\hat\cH$ and mean field limits}\label{sect:minH}

As in Section \ref{sect:KMtoEnergyFunctional}, fix a closed surface $(\Sigma,g)$ and an invertible, indefinite Schr\"odinger operator $J=d^*d-V$ on $\Sigma$. In this section, we prove that, for sufficiently large $n$, the minimization problem for the projectivized interaction energy $\hat\cH$ is well-posed, and yields critical points for $\cH$ satisfying estimates (1) and (2) of Theorem \ref{thm:doubling.precise}. Along the way, we also obtain a refined description of the asymptotic behavior of the minimizers as $n\to\infty$, modeled on the well-studied mean field analysis for classical Coulomb energies as described in, e.g. \cite[Section 2]{SerfatyLectures}.

To begin, we first introduce a natural relaxation for the interaction energy $\cE(\mu):=\int_{\Sigma\times \Sigma}G(x,y)d\mu(x)d\mu(y)$ for Radon measures, as follows. For each $\delta>0$, consider the Lipschitz regularization of $G(x,y)$ given by
$$G_{\delta}(x,y):=-\frac{1}{2\pi}\log(\max\{d(x,y),\delta\})+R(x,y),$$
satisfying $G_\delta(x,y)=G(x,y)$ when $d(x,y)\geq\delta$ and $$G_\delta(x,y)=-\frac{1}{2\pi}\log(\delta)+R(x,y)$$ when $d(x,y)<\delta$. Then for any Radon measure $\mu$ on $\Sigma$, we define the relaxed energy
$$\cE_{\delta}(\mu):=\int_{\Sigma\times \Sigma}G_\delta(x,y)d\mu(x)d\mu(y)\in \mathbb{R}.$$
We observe next that the functional $\cH$ on $X_n$ can be recast in terms of these energies for $\delta$ sufficiently small; throughout this section, we write $|\btau|:=|\btau|_{\ell^2}$ for simplicity.

\begin{lem}\label{ed.char}
    If $(\bp,\btau)\in X_n$ satisfies $\delta<\min_{i\neq j}d(p_i,p_j)$, then 
$$\cH(\bp,\btau)= \mathcal{E}_{\delta}(\sum_{i=1}^n \tau_i\delta_{p_i})+\sum_{i=1}^n\frac{\tau_i^2}{4\pi}\left (1-2 \log \frac {\tau_i}2\right)+\frac{|\btau|^2}{2\pi}\log\delta.$$
\end{lem}
\begin{proof}
  This is a straightforward computation; by definition of $G_{\delta}$ and $\cH$, we have 
  \begin{align*}
        \mathcal{E}_{\delta}(\sum_{i=1}^n \tau_i \delta_{p_i})&=\sum_{i,j=1}^n\tau_i\tau_jG_{\delta}(p_i,p_j)\\
        &=\sum_{i=1}^n\sum_{j\neq i}\tau_i\tau_j G(p_i,p_j)+\sum_{i=1}^n\tau_i^2G_{\delta}(p_i,p_i)\\
        &=\sum_{i=1}^n\sum_{j\neq i}\tau_i\tau_j G(p_i,p_j)+\sum_{i=1}^n\tau_i^2(-\frac{1}{2\pi}\log(\delta)+R(p_i,p_i))\\
        &=\cH(\bp,\btau)-\sum_{i=1}^n\frac{\tau_i^2}{4\pi}\left (1-2 \log \frac {\tau_i}2\right)-\frac{|\btau|^2}{2\pi}\log\delta.
    \end{align*}
\end{proof}

Next, noting that $G_\delta(x,y)\leq G_{\delta'}(x,y)\leq G(x,y)$ whenever $0<\delta'<\delta$, we see that the monotone convergence theorem gives the following characterization of the interaction energy $\cE$.
\begin{definition}
For any Radon measure $\mu$ on $\Sigma$, we define the energy
$$\cE(\mu):=\int G(x,y)d\mu(x)d\mu(y)=\sup_{\delta>0}\mathcal{E}_{\delta}(\mu)=\lim_{\delta\to 0}\mathcal{E}_{\delta}(\mu)\in (-\infty,+\infty].$$
\end{definition}



Since the functions $G_{\delta}(x,y)$ are Lipschitz on $\Sigma\times \Sigma$, it is straightforward to check that the relaxed energies $\cE_{\delta}$ are continuous with respect to the Lipschitz-dual norm $\|\cdot\|_{Lip(\Sigma)^*}$, and by a simple application of the Arzela-Ascoli theorem, it follows that $\cE_{\delta}$ is continuous with respect to weak-$*$ convergence on the space
$$\cP:=\{\mu\in C^0(\Sigma)^*\mid \mu\geq 0,\text{ }\mu(\Sigma)=1\}$$
of probability measures on $\Sigma$. Writing $\cE(\mu)=\sup_{\delta>0}\cE_{\delta}(\mu)$, it then follows that $\cE$ is lower semi-continuous with respect to weak-* convergence on $\Sigma$, and an application of the direct method shows that the infimum 
\begin{equation}\label{i.def}
\cI=\cI(g,V):=\inf_{\mu\in \cP}\cE(\mu)
\end{equation}
is realized by a probability measure $\mu\in \cP$, whose properties we describe further in Section \ref{sect:meanFieldLimit}. 

By definition of the Green's function $G(x,y)$, note that for any $u\in C^2(\Sigma)$ such that $Ju\geq 0$, we have
\begin{equation}\label{e.char}
\cE(Ju\; dA_g)=\iint Ju(x)Ju(y)G(x,y)dA_g(x)dA_g(y)=\int u JudA_g.
\end{equation}
More generally, it is a straightforward exercise, using Lemma \ref{green.lem} and Fubini's theorem, to check that for any $\mu\in \cP$ with $\cE(\mu)<\infty$, the potential $\phi_{\mu}(x):=\int G(x,y)d\mu(y)$ belongs to $W^{1,2}(\Sigma)$, and \eqref{e.char} holds in a weak sense with $u=\phi_{\mu}$.

Since we are assuming that $J$ is indefinite, i.e. $\lambda_1(J)<0$, we can take $u_1$ to be the first eigenfunction for $J$ such that $u_1<0$ and $\|u_1\|_{L^2(dA_g)}=1$. Then $Ju_1=\lambda_1u_1>0$,
and we see that
$$\cE(Ju_1\;dA_g)=\int_{\Sigma}u_1Ju_1 dA_g=\lambda_1(J),$$
while
$$\left(\int Ju_1dA_g\right)^2=\lambda_1(J)^2\left(\int u_1dA_g\right)^2\leq \lambda_1(J)^2\area(\Sigma,g),$$
and, consequently,
\begin{equation}\label{i.neg}
\cI\leq \frac{\cE(Ju_1dA_g)}{\left(\int Ju_1dA_g\right)^2}\leq\frac{1}{\lambda_1(J)\area(\Sigma,g)}<0.
\end{equation}

Turning now to the $\hat\cH$ minimization problem, recall from \eqref{hath.gibbs} that for any $(\bp,\bomega)\in  \hat X_n:=\{(\bp,\bomega)\in X_n:|\bomega|=1\}$, we have
$$\hat\cH(\bp,\bomega)=\frac{1}{4\pi e}e^{4\pi\cH(\bp,\bomega) }.$$
As a result, we see that
\begin{equation}\label{eq:infHatH}
\inf_{\hat{X}_n}\hat\cH=\frac 1{4\pi e}\exp\left(4\pi \inf_{(\bp,\bomega)\in  \hat X_n}\cH(\bp,\bomega)\right).
\end{equation}
In particular, it follows that $(\bp,\bomega)$ is a minimizer of $\hat \cH$ on $\hat X_n$ if and only if it is a minimizer of $\cH$ on $\hat X_n$, in which case the discussion in Section \ref{sect:hatH.def} shows that
$$(\bp,e^{ {2\pi \cH(\bp,\bomega)}  -1/2}\bomega)$$
is a {\it critical point} of $\cH$ on $X_n$ with energy
$$\cH(\bp,e^{ {2\pi \cH(\bp,\bomega)}  -1/2}\bomega)=\inf_{\hat{X}_n}\hat\cH.$$

The main theorem of this section confirms the existence of $\cH$-minimizing pairs in $\hat{X}_n$ for sufficiently large $n$, and establishes some key estimates for these minimizers that will be used in the next section to verify the hypotheses of Theorem \ref{thm:doubling.precise} for the associated critical points of $\cH$. 

\begin{thm}\label{thm:existGoodMinimizers}
Let $\Sigma$ be a closed surface equipped with a Riemannian metric $g$ and a function $V\in C^{\infty}(\Sigma)$ such that the Schrödinger operator $J=d^*_gd-V$ is invertible with $\lambda_1(J)<0$. Then
there exist constants $\overline n(g,V), \overline C(g,V),\overline\Lambda(g,V)>0$ such that for every $n\geq  \overline  n$, we have:
\begin{enumerate}
\item 
$n\cI - \overline C\log n<\inf_{\hat X_n} \cH< n\cI+\overline C n^{3/4}.$
\item There exists a minimizer for $\cH$ (equivalently, for $\hat \cH$) on $\hat X_n$.
\item\label{item:existGoodMinimizerPTau}  Every minimizer $(\bp,\bomega)$ for $\cH$ on $\hat X_n$ satisfies $\frac 1{ \overline\Lambda\sqrt n}<\omega_i<\frac{\overline\Lambda}{\sqrt n}$ for each $i$,  and $d(p_i,p_j)> n^{-3}$ for each $i\ne j$.
\end{enumerate} 
\end{thm}

Note that by \eqref{eq:infHatH}, the above estimate for $\inf_{\hat X_n} \cH$ implies
\begin{equation}\label{hath.asymp}
e^{4\pi n\cI-C(g,V)\log n}<\inf_{\hat{X}_n} \hat\cH < e^{4\pi n\cI+C(g,V) n^{3/4}},
\end{equation}
and if $(\bp,\btau)=(\bp, e^{2\pi \cH(\bp,\bomega)-1/2},\bomega)$ is the critical point of $\cH$ associated to a $\hat\cH$-minimizer $(\bp,\bomega)$, we see in particular that the weights $\tau_i$ decay exponentially fast as $n\to\infty$. As a byproduct of our analysis, we are also able to prove the following, describing the limiting distribution of the points $p_1,\ldots,p_n$ as $n\to\infty$.

\begin{thm}\label{thm:convMeasure}
In the setting of Theorem \ref{thm:existGoodMinimizers}, let $\{n_1,n_2,...\}\subset \N$ be an increasing sequence and, for each $k=1,2,...$, suppose $(\bp^k,\bomega^k)$ is a minimizer for $\cH$ on $\hat X_{n_k}$. Define the probability measure  $$\nu^k:=\frac 1 {n_k}\sum_{i=1}^{n_k}\delta_{p^k_i}$$ on $\Sigma$. Then $\nu^k$   subsequentially weak$^*$-converges  to a  probability measure   realizing the infimum $\cI:=\inf_{\mu\in\cP}\cE(\mu)$. 
\end{thm}

Finally, in the case where $J$ has index one, we prove some preliminary lower bounds for the Hessian $D^2\cH(\bp,\btau)$ of a critical point $(\bp,\btau)\in X_n$ for $\cH$ for which $(\bp,\btau/|\btau|_{\ell^2})$ is $\cH$-minimizing in $\hat X_n.$

\begin{prop}\label{prop:wt.nondeg}
In addition to the hypotheses of Theorem \ref{thm:existGoodMinimizers}, suppose that $\lambda_2(J)>0$, and let $(\bp,\btau)\in X_n$ be a critical point of $\cH$ such that
$$\cH(\bp,\btau/|\btau|_{\ell^2})=\inf_{\hat X_n}\cH.$$
Then there are constants $n_0(g,V)\in \mathbb{N}$ and $c_0(g,V)>0$ such that if $n\geq n_0$, then for any 
$$(\bv,\bs)\in \bigoplus_{i=1}^nT_{p_i}\Sigma\oplus \mathbb{R}^n= T_{(\bp,\btau)}X_n,$$
satisfying $\langle \bs,\btau\rangle_{\ell^2}=0$, we have
$$\langle S_{\btau}D^2\cH(\bp,\btau)S_{\btau}(\bv,\bs),(\bv,\bs)\rangle_{\ell^2}\geq \max\{0,c_0 n|\bs|^2-n^5|\bv||\bs|\}.$$
\end{prop}

The remainder of \S \ref{sect:minH} is devoted to the proof of these three results. In \S \ref{sect:upperBoundInf}, we prove the upper bound for $\inf\hat\cH$. In \S \ref{sect:PropertyMinimizers}, we establish the existence of minimizer for $\hat\cH$ and study its properties. In \S \ref{sect:lowerBoundInfHatH}, we prove a lower bound for $\inf\hat\cH$, and complete the proof of Theorem \ref{thm:existGoodMinimizers}. Finally in \S \ref{sect:meanFieldLimit}, we prove Theorem \ref{thm:convMeasure}, and in \S \ref{sect:wt.nondeg}, we prove Proposition \ref{prop:wt.nondeg}.

Throughout the remainder of \S \ref{sect:minH}, we will be in the setting of Theorem \ref{thm:existGoodMinimizers}, so that $\Sigma$ is a closed surface equipped with a metric $g$ and function $V\in C^{\infty}(\Sigma)$ such that the Schrödinger operator $J=d^*_gd-V$ is nondegenerate with $\lambda_1(J)<0$. As usual, the value of absolute constants $C$ may change from line to line, but will always depend on the parameters $(g,V)$, unless specified otherwise.

\subsection{The upper bound for $\inf_{\hat X_n}\cH$.}\label{sect:upperBoundInf}

To obtain asymptotically sharp upper bounds for $\inf_{\hat{X_n}}\cH$ in terms of $\cI$, we argue similarly to Steps 2 and 3 of \cite[Section 2.4]{SerfatyLectures}, beginning with the following discrete approximation lemma for smooth measures.

\begin{lem}\label{disc.app}
For any $0\leq\psi \in L^{\infty}(\Sigma)$ such that $\mu:=\psi dA_g$ is a probability measure, there are constants $\hat c(g,\|\psi\|_{L^{\infty}})>0$ and $\hat C(g,V,\|\psi\|_{L^\infty})>0$ such that for any $n\in \mathbb{N}$, we can find $n$ distinct points $p_1,\ldots,p_n\in \Sigma$ such that
$$\min_{i\neq j}d(p_i,p_j)> \hat cn^{-1/2},$$
and for any $\delta\in (0,1/2)$ the measure $\nu:=\frac 1n \sum_i\delta_{p_i}$ satisfies
$$\left| \cE_\delta(\mu) - \cE_\delta(\nu) \right| <\hat  C (|\log \delta| n^{-1/2} +  n^{-1/4}).$$
\end{lem} 

\begin{proof}   

To begin, fix $n\in \mathbb{N}$, and for any $1\leq k\leq n$, let $q_1,\ldots,q_k$ be a collection of $k$ distinct points maximizing the separation $\min_{i\neq j}d(q_i,q_j)$, and set $r_k:=\min_{i\neq j}d(q_i,q_j)$. Without loss of generality, we may assume that the disks $\{B_{10r_k}(q_i)\}$ cover $\Sigma$; otherwise, we could replace one of the points $q_i$ such that $d(q_i,q_j)=r_k$ for some $j\neq i$ with a new point $q_i'$ such that $d(q_i',q_j)\geq 10r_k$ for every $j\neq i$, and if the balls of radius $10r_k$ for this new collection still fail to cover $\Sigma$, we can repeat this process again, eventually contradicting the maximality of the separation $r_k$. As a consequence, we have an estimate of the form
$$\area(\Sigma,g)\leq C(g)k r_k^2,$$
while the disjointness of the disks $B_{r_k/2}(q_i)$ implies an estimate of the form
$$k r_k^2\leq C(g).$$
Next, for each $i=1,...,k$, let $$Q_i:=\{x\mid d(x,q_i)=\min_j d(x,q_j)\},$$ be the  Voronoi cell associated to $q_i$, and set
$$\alpha_i:=\int_{Q_i}\psi.$$
Noting that $\sum_{i=1}^n \alpha_i=\|\psi\|_{L^1}=1$, we can then select some $m_i\in \mathbb{N}\cup\{0\}$ for each $i=1,...,k$ such that $\sum_i m_i=n$ and $|m_i/n-\alpha_i|\leq 2/n$ (cf. \cite[Section 2.4, Claim 1]{SerfatyLectures}). 

Next, inside each disk $B_{r_k/2}(q_i)\subset Q_i$, we can select $m_i$ further points $q_{i,1},\ldots,q_{i,m_i}\in Q_i$ satisfying satisfying $d(q_{i,j},q_{i,j'})> c_3\cdot(km_i)^{-1/2}$ whenever $j\neq j'$ and $d(q_{i,j},\partial Q_i)> c_3k^{-1/2}$ for some $c_3=c_3(\Sigma,g)>0$. The collection of all such points $q_{i,j}$ will then provide the desired set $\{p_1,...,p_n\}$; we check now that they satisfy the desired properties.

We first confirm the separation estimate $d(p_i,p_j)>\hat{c}n^{-1/2}$ for $i\ne j$. By construction, if $p_i,p_j$ belong to distinct Voronoi cells, then $d(p_i,p_j)>c_3 k^{-1/2}>c_3 n^{-1/2}$, so it remains to consider the case where $p_i,p_j\in Q_i$ for a common cell $Q_i$. In this case, we have
$$d(p_i,p_j)>c_3 \cdot (km_i)^{-1/2},$$
where we recall that
$$m_i\leq n\alpha_i+2.$$
On the other hand, by definition of $\alpha_i$ and the fact that $Q_i\subset B_{10r_k}(q_i)$, we have an estimate of the form
$$\alpha_i\leq C(g)r_k^2\|\psi\|_{L^{\infty}}\leq C(g,\|\psi\|_{L^{\infty}})\cdot k^{-1},$$
so that
$$km_i\leq k(n\alpha_i+2)\leq [C(g,\|\psi\|_{L^{\infty}})+2]n,$$
giving an estimate of the desired form
$$d(p_i,p_j)>c_3 \cdot (km_i)^{-1/2}>\hat{c} n^{-1/2}.$$

Next, setting
$$\nu:=\sum_{i=1}^n\frac{1}{n}\delta_{p_{i}}=\sum_{i=1}^k\sum_{j=1}^{m_i}\frac{1}{n}\delta_{q_{i,j}},$$
it remains to establish the desired bound for $|\mathcal{E}_{\delta}(\mu)-\mathcal{E}_{\delta}(\nu)|$.
To this end, we first estimate the integrals 
$$\int_{_{Q_i\x Q_j}}G_\delta(x,y) (d\mu d\mu-d\nu d\nu)$$
for every $1\leq i,j\leq n,$ treating separately the cases where the centers $q_i,q_j$ of the cells $Q_i,Q_j$ satisfy $d(q_i,q_j)>30r_k$ and $d(q_i,q_j)\leq 30r_k$.

\subsection*{Case 1: Assume $d(q_i,q_j)>30r_k$.} \hspace{40mm}

In light of the containment $Q_i\subset B_{10r_k}(q_i)$, the separation assumption $d(q_i,q_j)>30r_k$ and a simple application of the triangle inequality yields in this case the estimates
\begin{equation}\label{tri.est}
    \frac{1}{3}d(q_i,q_j)<d(x,y)<\frac{5}{3}d(q_i,q_j).
\end{equation}
Now, write
$$\left|\int_{Q_i \times Q_j} G_\delta(x, y) \, (d\mu \, d\mu - d\nu \, d\nu) \right| \notag \leq I_1+I_2,$$
where
$$I_1=\left| \int_{Q_i \times Q_j} G_\delta(q_i, q_j) \, (d\mu \, d\mu - d\nu \, d\nu) \right|$$
and
$$I_2=\left| \int_{Q_i \times Q_j} (G_\delta(x, y) - G_\delta(q_i, q_j)) \, (d\mu \, d\mu - d\nu \, d\nu) \right|.$$

To estimate $I_1$, note first that
$$I_1=|G_{\delta}(q_i,q_j)||\alpha_i\alpha_j-\frac{m_im_j}{n^2}|,$$
which together with the bounds $|\alpha_i-m_i/n|\leq 2/n$ and $\alpha_i\leq C(g,\|\psi\|_{L^{\infty}})/k$ leads to an estimate of the form
$$I_1\leq |G_{\delta}(q_i,q_j)|\cdot (\alpha_j|\alpha_i-\frac{m_i}{n}|+\frac{m_i}{n}|\alpha_j-\frac{m_j}{n}|)\leq |G_{\delta}(q_i,q_j)|\cdot\frac{C(g,\|\psi\|_{L^{\infty}})}{nk}.$$
Combining this with the observation that $\area(Q_i)\geq \area(B_{r_k/2}(q_i))\geq \frac{c(g)}{k}$ for a suitable $c(g)>0$, we deduce that
$$I_1\leq \frac{k^2}{c}\cdot \frac{C}{nk}\int_{Q_i\times Q_j}|G_{\delta}(q_i,q_j)|dA_gdA_g.$$
Finally, since $|G_{\delta}(q_i,q_j)|\leq \frac{C}{d(q_i,q_j)}$ for some constant $C(g,V)$, an application of \eqref{tri.est} gives an estimate of the form
$$I_1\leq C \frac{k}{n}\int_{Q_i\times Q_j}\frac{1}{d(x,y)}dA_gdA_g$$
for a constant $C=C(g,\|\psi\|_{L^{\infty}},V)$.

As for $I_2$, it follows from Lemma \ref{green.lem} and \eqref{tri.est} that there is a constant $C(g,V)$ such that
\begin{align}
\| G_\delta \|_{Lip(Q_i\x Q_j)} \leq \frac{C}{d(q_i,q_j)}. \notag
\end{align}
It then follows that
$$|G_{\delta}(x,y)-G_{\delta}(q_i,q_j)|\leq \frac{C}{d(q_i,q_j)}\cdot(d(x,q_i)+d(y,q_j))\leq \frac{20r_k C}{d(q_i,q_j)},$$
for any $(x,y)\in Q_i\times Q_j$, and consequently
$$I_2\leq \frac{C'(g,V)}{\sqrt{k}}\cdot \frac{1}{d(q_i,q_j)}(\mu(Q_i)\mu(Q_j)+\nu(Q_i)\nu(Q_j)).$$
Since $\mu(Q_i)=\alpha_i\leq C(g,\|\psi\|_{L^{\infty}})\cdot k^{-1}$ and $\nu(Q_i)=\frac{m_i}{n}\leq \alpha_i+\frac{2}{n}$, it then follows that
$$I_2\leq \frac{C(g,V,\|\psi\|_{L^{\infty}})}{\sqrt{k}\cdot d(q_i,q_j)}\cdot \frac{1}{k^2}\leq \frac{C(g,V,\|\psi\|_{L^{\infty}})}{\sqrt{k}\cdot d(q_i,q_j)}\cdot \area(Q_i)\area(Q_j),$$
and appealing once more to \eqref{tri.est}, we obtain an estimate of the form
$$I_2\leq \frac{C}{\sqrt{k}}\int_{Q_i\times Q_j} \frac{1}{d(x,y)}dA_gdA_g.$$

Combining this with the preceding estimate for $I_1$, we find that
\begin{equation}
    |\int_{Q_i\times Q_j}G_{\delta}(x,y)(d\mu d\mu-d\nu d\nu)|\leq C\left(\frac{k}{n}+\frac{1}{k^{1/2}}\right)\int_{Q_i\times Q_j}\frac{1}{d(x,y)}dA_gdA_g
\end{equation}
for some constant $C(g,V,\|\psi\|_{L^{\infty}}$, when $d(q_i,q_j)>30r_k.$

\subsection*{Case 2: Assume $d(q_i,q_j)\leq 30r_k$.} \hspace{40mm}

In this case, we use the crude bound $|G_{\delta}(x,y)|\leq C(g,V)|\log \delta|$ together with the estimates
$$\nu(Q_i)+\mu(Q_i)\leq \frac{C}{k}\leq C'(g,\|\psi\|_{L^{\infty}})\area(Q_i)$$
to see that
\begin{align*}
\left| \int_{Q_i \times Q_j} G_\delta(x, y) \, (d\mu \, d\mu - d\nu \, d\nu) \right| \notag 
&\leq C  \int_{Q_i \times Q_j} |\log \delta |\, (d\mu \, d\mu + d\nu \, d\nu) \notag \\
&\lesssim  C'|\log \delta| vol(Q_i\times Q_j) \notag.
\end{align*}
Note moreover that the (disjoint) union of all such pairs is contained in the set
$$S_k=\{(x,y)\in \Sigma\times \Sigma\mid d(x,y)\leq 50r_k\},$$
whose volume is easily seen to satisfy a bound of the form
$$vol(S_k)\leq C(g) k^{-1}.$$

\medskip

Finally, to estimate $|\cE_{\delta}(\mu)-\cE_{\delta}(\nu)|$, we sum over these cases to obtain an estimate of the form
\begin{eqnarray*}
    |\cE_{\delta}(\mu)-\cE_{\delta}(\nu)|&\leq &\sum_{i=1}^n\sum_{j=1}^n\left|\int_{Q_i\times Q_j}G_{\delta}(x,y)(d\mu d\mu- d\nu d\nu)\right|\\
    &\leq &C\left(\frac{k}{n}+k^{-1/2}\right)\int_{\Sigma\times \Sigma}\frac{1}{d(x,y)}dA_gdA_g+C|\log \delta|\cdot \frac{1}{k}.
\end{eqnarray*}
Since $\int_{\Sigma\times \Sigma}\frac{1}{d(x,y)}dA_gdA_g<\infty$, taking $k=\sqrt n$ then gives us the desired estimate. \end{proof}

With this approximation lemma in hand, we can now prove a fairly sharp upper bound for the infimum $\inf_{\hat X_n}\cH$.

\begin{prop}\label{prop:betaNUpperBound}
    For each $n$ we have
    $$\inf_{\hat X_n}\cH<\cI n+ C n^{3/4}$$
    for a constant $C=C(g,V)$.
\end{prop}
\begin{proof}
By Proposition \ref{prop:eq.char} below, we know that $\cI$ is realized by an equilibrium measure $\cI=\cE(\mu)$ of the form $\mu=\psi dA_g$, where $\psi\in L^{\infty}(\Sigma)$ is given by the restriction of $-\cI V$ to a subset $K\subset \{V\geq 0\}$. By applying Lemma \ref{disc.app} with $\delta=1/n$, we then have a point  
$$(\bp,\btau):=(p_1,...,p_n, 1/n,...,1/n)\in X_n$$ 
satisfying estimates of the form
$$|\cE_{\delta}(\mu)-\cE_{\delta}(\nu)|<C(g,V)n^{-1/4}$$
and
$$\min_{i\neq j}d(p_i,p_j)>\hat{c} n^{-1/2}>\frac{1}{n}$$
for $n$ sufficiently large. 

In particular, it follows that
$$\cE_{\delta}(\nu)\leq \cI+C n^{-1/4},$$
and since $d(p_i,p_j)>\frac{1}{n}=\delta$ for $i\neq j$ and $\tau_i=\frac{1}{n}$, it follows from Lemma \ref{ed.char} that
\begin{align*}
\cH(\bp,\btau)&= \mathcal{E}_{\delta}(\sum_{i=1}^n\frac{1}{n}\delta_{p_i})+\sum_{i=1}^n\frac{1}{4\pi n^2}\left (1+2 \log 2n\right)-\frac{1}{2\pi n^2}\log n\\
&\leq \cI+C' n^{-1/4}.
\end{align*} 
Finally, applying (\ref{eq:scaleTau}) with $t=\sqrt n$, we have
\begin{align*}
    \cH(\bp,1/\sqrt n,...,1/\sqrt n)&=n \cH(\bp,\btau)-\frac {\log\sqrt n}{2\pi n}\\ 
    &< n\cI+Cn^{3/4},
\end{align*}
and since $(\bp,\frac{1}{\sqrt n},...,\frac{1}{\sqrt n})\in \hat X_n$, we arrive at a bound of the desired form.
\ref{disc.app}.

\end{proof}

\subsection{Existence and properties of minimizers for $\hat \cH$}\label{sect:PropertyMinimizers}

Recall that, by the relation \eqref{eq:infHatH}, the minimizers of $\hat\cH$ on $\hat X_n:=X_n\cap\{|\bomega|=1\}$ are also the minimizers of $\cH$ on $\hat X_n$, realizing the infimum
$$\alpha_n:=\inf_{(\bp,\bomega)\in\hat X_n}\H(\bp,\bomega).$$

Moreover, recall from \eqref{i.neg} that $\cI:=\inf_{\mu\in\cP}\cE(\mu)<0$, since $\lambda_1(J)<0$. As a consequence, it follows from Proposition \ref{prop:betaNUpperBound} that $\alpha_n<0$ for $n\geq C(g,V)$ sufficiently large. For $n$ sufficiently large, the following proposition establishes precompactness of minimizing sequences for $\cH$ in $\hat X_n$, along with some additional quantitative information.

\begin{prop}\label{prop:piTauiAlmostMinimum} Assume $\lambda_1(J)<0$.
    There are constants $\hat c(g,V)>0$ and $C(g,V)<\infty$ such that for every $n\geq C$, if $(\bp,\bomega)\in \hat X_n$ satisfies
    $$\alpha_n\leq\cH(\bp,\bomega)\leq (1-\hat c^2/n)\alpha_n,$$
    then the following hold:
    \begin{enumerate}
        \item For each $i$, $\omega_i> \frac{\hat c}{2\sqrt n}$. 
        \item $\min_{i\ne j} d(p_i,p_j)>\hat \gamma_n$ for some $\hat\gamma_n(g,V)>0$.
    \end{enumerate}
\end{prop}
\begin{remark}
    While the lower bound for $\omega_i$ is sharp up to a constant multiple, note that the lower bound for $d(p_i,p_j)$ given here is simply a rough preliminary bound, used only to show that $\cH$-minimizing sequences remain in the interior of $\hat X_n$. 
\end{remark}

Before proving this proposition, we record two important consequences.

\begin{prop}\label{prop:existenceMinimizer} Under the assumptions of Theorem \ref{thm:existGoodMinimizers}, for $n\geq C(g,V)$, the infimum $\inf _{\hat X_n}\cH$  is achieved by some minimizer  $(\bp,\bomega)\in \hat X_n$. 
\end{prop}
\begin{proof}[Proof (assuming Proposition \ref{prop:piTauiAlmostMinimum}).]
Let us take a minimizing sequence of $\cH$ in $\hat X_n$,  $\{(\bp^{k},\bomega^{k})\}_k$. Passing to a subsequence, it is always possible to find a limit $(\bp,\bomega)\in \Sigma^n\times ([0,\infty)^n\cap S^{n-1})$ in the compact closure of $\hat{X}_n$. For $n\geq C(g,V)$ sufficiently large, it follows from Proposition \ref{prop:piTauiAlmostMinimum} that $$\liminf_{k\to\infty}(\min_i\omega^{k})>0,\quad\textrm{ and }\quad \liminf_{k\to\infty}(\min_{i\ne j}d(p_i^{k},p_j^{k}))>0,$$ 
so that the limit $(\bp,\bomega)$ must have $\omega_i>0$ for each $1\leq i\leq n$ and $d(p_i,p_j)>0$ for $i\neq j$. Thus, $(\bp,\bomega)\in \hat{X}_n$ provides the desired minimizer.
\end{proof}

As another important consequence of Proposition \ref{prop:piTauiAlmostMinimum}, we see that the weights $\omega_i$ of any $\cH$-minimizer $(\bp,\bomega)\in \hat X_n$ must be proportional.

\begin{prop}\label{prop:tauIComparable} Under the same assumptions, there exists some $\hat \Lambda(g,V)>0$ such that for every $n\geq C(g,V)$ sufficiently large, every minimizer $(\bp,\bomega)$ of $\cH$ on $\hat X_n$ satisfies $\frac 1{\hat \Lambda \sqrt n}<\omega_i<\frac {\hat \Lambda}{\sqrt n}$ for each $i$.
\end{prop}
\begin{proof}[Proof (assuming Proposition \ref{prop:piTauiAlmostMinimum}).]
Let $(\bp,\bomega)$ be a minimizer of $  \cH$  on $\hat X_n$. Note that item (1) of Proposition \ref{prop:piTauiAlmostMinimum} already supplies a lower bound of the desired form for $\omega_i$, so to prove Proposition \ref{prop:tauIComparable}, it suffices to prove an estimate of the form
$$\bar\bomega>c\omega_i$$ 
for each $1\leq i\leq n$ for the average $\bar\bomega:=\frac{1}{n}\sum_{i=1}^n\omega_i$, for some constant $c(g,J)>0$.

Since $(\bp,\bomega)$ is a critical point for $\hat\cH$ in $\hat X_n$, it follows from the discussion in Section \ref{sect:hatH.def} that the point
$$(\bp,\btau):=(\bp,e^{ {2\pi \cH(\bp,\bomega)} -1/2}\bomega)$$
is a {\it critical point} for $\cH$ in $X_n$, satisfying 
$$\frac 1{4\pi}|\btau|^2=\frac{1}{4\pi e}e^{4\pi\cH(\bp,\bomega)}=\frac{1}{4\pi e}e^{4\pi\alpha_n}.$$
In particular, it follows from Proposition \ref{prop:betaNUpperBound} that we have an estimate of the form 
\begin{equation}\label{log.tausq}
    \log(|\btau|^2/4)< -c n
\end{equation}
for $n$ sufficiently large.

Now, by the criticality condition with respect to $\tau_i$ in Proposition \ref{prop:criticalPointH} and Lemma \ref{green.lem}, we have 
\begin{align*}
    \frac{\tau_i}{\pi}\log\frac{\tau_i}2&=2\sum_{j\neq i}\tau_jG(p_j,p_i)+2\tau_i R(p_i,p_i)
    \geq  -C \sum_{j=1}^n\tau_j
\end{align*}
for a suitable constant $C(g,V)$. And since  $$\log(\tau_i/2)\leq \frac{1}{2}\log(|\btau|^2/4)< -c n$$ 
by \eqref{log.tausq},
it follows immediately that 
$$ \sum_{j=1}^n\tau_j> c n\tau_i,$$
or, equivalently, $\bar\btau>c\tau_i$. Finally, since $\btau=e^{ {2\pi \cH(\bp,\bomega)} -1/2}\bomega$, we have that $\frac{\tau_i}{\bar \btau}=\frac{\omega_i}{\bar\bomega}$, so the desired estimate $\bar\bomega>c\omega_i$ follows.    
\end{proof}

We turn now to the proof of Proposition \ref{prop:piTauiAlmostMinimum}, beginning with the following lemma.

\begin{lem}\label{minphi.bd} There exists some constant $c_0(g,V)>0$ such that for every $k$, for any $(\bp,\btau)\in X_k$, if $\phi$ is the solution of $J\phi=\sum_{i=1}^k\tau_i\delta_{p_i}$, then
    $$\min \phi < -c_0\sum_{i=1}^k\tau_i.$$ 
\end{lem}
\begin{proof}[Proof of Lemma \ref{minphi.bd}]

Since $\lambda_1(J)<0$, we can fix a first eigenfunction $\psi_1>0$ for $J$ satisfying $J\psi_1=\lambda_1\psi_1$ and $\min \psi_1=1$. Integrating $J\psi_1=\lambda_1\psi_1$ against $\phi$ then gives
$$\lambda_1\int \phi \psi_1=\int \phi J\psi_1=\int \psi_1 J\phi=\sum_{i=1}^k\tau_i\psi_1(p_i)\geq\sum_{i=1}^k\tau_i.$$
Hence,
$$\min(\phi\psi_1)\leq \frac{-\sum_i \tau_i}{|\lambda_1|},$$
and using $\min(\phi\psi_1)\leq \min\phi\cdot\max\psi_1$ we have
$$\min\phi\leq \frac{-\sum_i \tau_i}{|\lambda_1|\max\psi_1},$$
giving the desired estimate with, e.g., $c_0=\frac{1}{2|\lambda_1|\max\psi_1}$.
\end{proof}

We can now prove Proposition \ref{prop:piTauiAlmostMinimum}.

\begin{proof}[Proof of Proposition \ref{prop:piTauiAlmostMinimum}]
Recall that $\alpha_n<0$ for $n\geq C(g,V)$ sufficiently large.

{\it In Step 1 to 4 below, we are going to find some $0<\hat c(g,J)<1/10$ such that for every $n$ sufficiently large,  the inequality $\cH(\bp,\bomega)\leq (1-\hat c ^2/ n)\alpha_n$ implies $\min_i\omega_i>\frac {\hat c }{2\sqrt n}$}, proving item (1) in Proposition \ref{prop:piTauiAlmostMinimum}. Then in Step 5 we will prove item (2).

To begin, set $\sigma:=\hat c /\sqrt n$ for a constant $\hat c$ that we will fix later, and consider a point $(\bp,\bomega)\in \hat X_n$ satisfying $\cH(\bp,\bomega)\leq (1-\sigma^2)\alpha_n$.

\subsection*{Step 1}
Up to relabeling, suppose $\omega_n=\min_i\omega_i,$ which clearly implies $\omega_i\leq 1/{\sqrt{n}}$. Let $\tilde\bp:=(p_1,...,p_{n-1})$ and $\tilde\bomega:= (\omega_1,...,\omega_{n-1})$, and let $\hat\phi$ solve
$$J\tilde{\phi}=\sum_{i=1}^{n-1}\omega_i\delta_{p_i};$$
in other words, $$\tilde\phi(x)=\sum^{n-1}_{i=1}\omega_i G(p_i,x).$$
Then directly from the definition of $\cH$, we see that
\begin{align}\label{ch.exp.bd}
\nonumber \cH(\bp,\bomega)&=\cH(\tilde{\bp},\tilde{\bomega})+2\omega_n \tilde\phi(p_n)+\omega_n^2R(p_n,p_n)+\frac{\omega_n^2}{4\pi}\left (1-2 \log \frac {\omega_n}2\right)\\
&> \cH(\tilde{\bp},\tilde{\bomega})+2\omega_n\tilde{\phi}(p_n)
\end{align}
for $n$ large enough that $1-2\log \frac{\omega_n}{2}\geq 1-2\log \frac{1}{2\sqrt{n}}>\|R\|_{C^0}$. As a consequence, using that $\omega_n\leq 1/\sqrt n$ alongside the estimate
$$\cH(\bp,\bomega)>\cH(\tilde\bp,\tilde\bomega)+2\omega_n\tilde\phi(p_n)$$
and applying Proposition \ref{prop:betaNUpperBound} to see that
$$\cH(\bp,\bomega)\leq (1-\sigma^2)\alpha_n\leq\frac{1}{2}n\cI$$
for $n$ sufficiently large, we see that
$$\cH(\tilde \bp,\tilde\bomega)<\frac{1}{2}n\cI-\frac{2}{\sqrt{n}}\tilde{\phi}(p_n).$$
In particular, since the Green's function $G$ satisfies a universal \emph{lower} bound $G(x,y)\geq -C(g,V)$, a simple application of Cauchy-Schwarz gives
$$\tilde{\phi}(p_n)\geq -C(g,V)\sum_{i=1}^{n-1}\omega_i\geq -C \sqrt{n},$$
which together with the preceding bound implies
$$\cH(\tilde\bp,\tilde\bomega)<\frac{1}{2}n\cI-2C<0$$
for $n$ sufficiently large.

\subsection*{Step 2}

Now, let $q$ be a minimum point of the function $\tilde\phi$ on $\Sigma$, and note that $\tilde\phi(q)<0$ by Lemma \ref{minphi.bd}. Now, for $\sigma\in (0,1)$, consider the point 
$$(\tilde\bp,q,\sqrt{\frac{1-\sigma^2}{1-\omega_n^2}}\tilde\bomega,\sigma)\in \hat X_n.$$
Using (\ref{eq:scaleTau}), we compute
\begin{align*}
&\cH(\tilde\bp,q,\sqrt{\frac{1-\sigma^2}{1-\omega_n^2}}\tilde\bomega,\sigma)\\
&=\cH(\tilde\bp,\sqrt{\frac{1-\sigma^2}{1-\omega_n^2}}\tilde\bomega)+2\sigma\tilde\phi(q)+\sigma^2R(q,q)+\frac{\sigma^2}{4\pi}(1-2\log\frac{\sigma}2)\\
&\leq\frac{1-\sigma^2}{1-\omega_n^2}\cH(\tilde\bp,\tilde\bomega)-\frac{1-\sigma^2}{1-\omega_n^2}\log\sqrt{\frac{1-\sigma^2}{1-\omega_n^2}}\cdot \frac{|\tilde \bomega|^2}{2\pi}+2\sigma\tilde \phi(q)+C\sigma^2|\log\sigma|\\
&\leq\frac{1-\sigma^2}{1-\omega_n^2}\cH(\tilde\bp,\tilde\bomega)-\frac {1-\sigma^2}{4\pi(1-\omega_n^2)}\log\frac{1-\sigma^2}{1-\omega_n^2}+2\sigma\tilde\phi(q)+C\sigma^2|\log\sigma|\\
\end{align*}
Then, recalling that $\cH(\tilde\bp,\tilde\bomega)<0$, noting that $0<1-\sigma^2<\frac{1-\sigma^2}{1-\omega_n^2},$  and observing that $|\log \frac{1-\sigma^2}{1-\omega_n^2}|\leq C\sigma^2$, since $\omega_n<1/2$, we see that 
\begin{equation}\label{ch.pert.bd}
\cH(\tilde\bp,q,\sqrt{\frac{1-\sigma^2}{1-\omega_n^2}}\tilde\bomega,\sigma)< (1-\sigma^2)\cH(\tilde\bp,\tilde\bomega)+2\sigma\tilde\phi(q)+C\sigma^2|\log\sigma|.
\end{equation}

\subsection*{Step 3}
Next, note that
$$\cH(\bp,\bomega)\leq (1-\sigma^2)\alpha_n\leq  (1-\sigma^2)\cH(\tilde\bp,q,\sqrt{\frac{1-\sigma^2}{1-\omega_n^2}}\tilde\bomega,\sigma).$$
Then, applying \eqref{ch.exp.bd} and \eqref{ch.pert.bd}, we obtain
$$\cH(\tilde{\bp},\tilde{\bomega})+2\omega_n\tilde{\phi}(p_n)<  (1-\sigma^2)^2\cH(\hat{\bp},\hat{\bomega})+2(1-\sigma^2)\sigma \tilde{\phi}(q)+C\sigma^2|\log\sigma|,$$
assuming that $n\geq C(g,V)$ is sufficiently large. Then, subtracting $\cH(\tilde \bp,\tilde\bomega)$ and dividing by $2$ on both sides of the above, we find that
$$\omega_n\tilde{\phi}(q)\leq \omega_n\tilde{\phi}(p_n)< - \sigma^2\cH(\tilde{\bp},\tilde{\bomega})+(1-\sigma^2)\sigma \tilde{\phi}(q)+C\sigma^2|\log\sigma|.$$ 
Noting that
$$\cH(\tilde\bp,\tilde\bomega)\geq 2n\cI > -C_1(\omega_1+...+\omega_{n-1})^2$$
for some $C_1(g,V)>0$ whenever $n$ is large enough,
we then have
\begin{equation}\label{step3.bd}
(\sigma-\sigma^3-\omega_n)\tilde\phi(q)>- C_1\sigma^2(\omega_1+...+\omega_{n-1})^2-C\sigma^2|\log\sigma|.
\end{equation}

\subsection*{Step 4} 
{\it Now suppose that 
\begin{equation}\label{sigma.fix}
\omega_n\leq \sigma/2=\frac{\hat c }{2\sqrt n},
\end{equation}
for a constant $\hat c\in (0,1)$ to be chosen shortly}, and note that we then have 
$$\sigma-\sigma^3-\omega_n>\sigma/4$$
for $n$ sufficiently large. Using that $\tilde\phi(q)=\min\tilde\phi<0$ by Lemma \ref{minphi.bd}, an application of \eqref{step3.bd} gives
\begin{equation}\label{eq:hatPhiQLowerBound}
\tilde\phi(q)> -\frac 14 C_1\sigma(\omega_1+...+\omega_{n-1})^2-C\sigma^2|\log\sigma|.
\end{equation}
On the other hand, taking $k=n-1$ in Lemma \ref{minphi.bd}, we also see that 
\begin{equation}\label{eq:hatQUpperBound}
    \tilde\phi(q)\leq -c_0(\omega_1+...+\omega_{n-1}),
\end{equation}
for a suitable constant $c_0(g,V)\in (0,1)$, which together with \eqref{eq:hatPhiQLowerBound} implies that
$$c_0(\omega_1+\cdots+\omega_{n-1})\leq -\tilde{\phi}(q)<\frac{C_1}{4}\sigma(\omega_1+\cdots+\omega_{n-1})^2+C\sigma^2|\log \sigma|.$$
Now, since the Cauchy-Schwarz inequality gives
$$(\omega_1+\cdots+\omega_{n-1})\leq \sqrt{n},$$
taking 
$$\hat c=\frac{c_0}{C_1}$$ 
in \eqref{sigma.fix}, it follows that
$$c_0(\omega_1+\cdots+\omega_{n-1})<\frac{c_0}{4}(\omega_1+\cdots+\omega_{n-1})+C\sigma^2|\log\sigma|,$$
and rearranging gives
$$\frac{3c_0}{4}(\omega_1+\cdots+\omega_{n-1})<C\sigma^2|\log\sigma|\leq C'(g,V)/\sqrt{n}.$$
On the other hand, since 
$$(\omega_1+\cdots+\omega_{n-1})^2\geq \sum_{i=1}^{n-1}\omega_i^2=1-\omega_n^2\geq \frac{1}{2},$$
it then follows that
$$\frac{3c_0}{4\sqrt{2}}<C'(g,V)/\sqrt{n},$$
leading to an obvious contradiction for $n>\frac{32(C')^2}{9c_0^2}$. 

In other words, it follows that for $n\geq C(g,V)$ sufficiently large, we must have $\omega_n>\frac{\hat c}{2\sqrt n}$, completing the proof of item (1) of the Proposition.



\subsection*{Step 5}
Finally, we prove item (2) of Proposition \ref{prop:piTauiAlmostMinimum}. 
By definition of $\cH(\bp,\bomega)$ and the assumption that 
$$\cH(\bp,\bomega)\leq (1-\hat c ^2/ n)\alpha_n<0,$$
we see that
$$\sum^n_{i=1}\sum_{j\neq i}\omega_i\omega_jG(p_i,p_j)+\sum_{i=1}^n\omega_i^2R(p_i,p_i)+\sum_{i=1}^n\frac{\omega_i^2}{4\pi}\left (1-2 \log \frac {\omega_i}2\right)<0.$$
Using Lemma \ref{green.lem} together with obvious bounds for $\sum_{i=1}^n\omega_i^2\left(1-2\log\frac{\omega_i}{2}\right)$, it then follows that
$$\sum_{i=1}^n\sum_{j\neq i}\omega_i\omega_j|\log d(p_i,p_j)|\leq C_n(g,V)$$
for some constant $C_n(g,V)$ depending on $n$. On the other hand, since we have already proved item (1) of the Proposition, we know moreover that 
$$\omega_i\omega_j\geq \frac{\hat{c}^2}{4n},$$
and it follows that
$$|\log(\min_{i\neq j}d(p_i,p_j))|\leq C_n(g,V)\frac{4n}{\hat{c}^2},$$
from which the desired non-sharp estimate follows.



\end{proof}

Finally, we show that Proposition \ref{prop:tauIComparable} together with the criticality condition suffice to give improved lower bounds on the distance $\min_{i\neq j}d(p_i,p_j)$ for any $\hat\cH$-minimizing configuration. These bounds are non-sharp, but suffice for our purposes.

\begin{prop}\label{prop:p_iLowerBoundDistance} Under the assumptions of Theorem \ref{thm:existGoodMinimizers} or every $\epsilon>0$, there is a constant $\hat c_{\epsilon}(g,J)>0$ such that, for every $n$ large enough (depending on $g,J$),  every minimizer  $(\bp,\bomega)\in \hat X_n$ for $\hat\cH$ must satisfy 
    $$d(p_i,p_j)\geq \hat c_{\epsilon}n^{-2-\epsilon}$$
    for all $i\neq j$.
\end{prop}
\begin{proof}Let  $(\bp,\bomega)\in \hat X_n$ be a minimizer for $\cH$ on $\hat X_n$, and observe that $(\bp,\bomega)$ is then critical for $\cH$ among all variations of the form $\bq\mapsto (\bq,\bomega)$ fixing the $(0,\infty)^n$ factor.

\subsection*{Step 1}
    For each $i=1,...,n$, criticality of $(\bp,\bomega)$ for $\cH$ in the $p_i$ variable gives (as in the proof of Proposition \ref{prop:criticalPointH})
    \begin{equation}\label{cn.crit}
\sum_{j\neq i}\omega_j\nabla^yG(p_j,p_i)+\omega_i\nabla^yR(p_i,p_i)=0\in T_{p_i}\Sigma.
\end{equation}
Given a test function $\psi\in C^1(\Sigma)$, we can then pair \eqref{cn.crit} with the vector $\omega_i\nabla\psi(p_i)\in T_{p_i}\Sigma,$ and sum over $1\leq i\leq n$ to obtain
$$\sum_{i=1}^n\sum_{j\neq i}\omega_i\omega_j\langle \nabla \phi_{p_j}(p_i),\nabla \psi(p_i)\rangle=-\sum_{i=1}^n\frac{1}{2}\omega_i^2\langle\nabla R_D(p_i),\nabla \psi(p_i)\rangle,$$
where, recall, $\phi_{p_j}(x)=G(p_j,x)$ and $R_D(x)=R(x,x)$. In particular, since $(\bp,\bomega)\in \hat X_n$ and $R_D\in C^{\infty}(\Sigma)$, it follows that
$$|\sum_{j\neq i}\omega_i\omega_j\langle \nabla \phi_{p_j}(p_i),\nabla \psi(p_i)\rangle|\leq \|R_D\|_{C^1}\|\nabla \psi\|_{C^0}.$$
Next, recalling Lemma \ref{green.lem}, fix a radius $r_0(g)<\InjRad(\Sigma,g)$ such that $R$ is $C^1$ on the neighborhood $T_{r_0}=\{(x,y)\in \Sigma\times \Sigma\mid d(x,y)\leq r_0\}$ of the diagonal in $\Sigma\times \Sigma$. Then, let $S\subset \{(i,j) \mid i\neq j\}$ be the set of pairs for which $d(p_i,p_j)<r_0$.  For $(i,j)\notin S$, it follows from Lemma \ref{green.lem} that
$$|\nabla \phi_{p_j}(p_i)|\leq \frac{C(g,V)}{r_0}=C'(g,V),$$
which together with the preceding calculation and the obvious bound
$\sum_{i,j=1}^n\omega_i\omega_j\leq n$
leads us to an estimate of the form
$$|\sum_{(i,j)\in S}\omega_i\omega_j\langle \nabla \phi_{p_j}(p_i),\nabla \psi(p_i)\rangle|\leq C(g,V)n\|\nabla \psi\|_{C^0}.$$
Moreover, for $(i,j)\in S$, using that 
$$|\nabla \phi_{p_j}(p_i)+\frac{1}{2\pi}\nabla (\log d_{p_j})(p_i)|\leq \|R\|_{C^1(T_{r_0})},$$
we see finally that
\begin{equation}\label{cn.crit.2}
    |\sum_{(i,j)\in S}\omega_i\omega_j\frac{\langle \nabla d_{p_j}(p_i),\nabla \psi(p_i)\rangle}{d(p_i,p_j)}|\leq C(g,V)n\|\nabla \psi\|_{C^0}.
\end{equation}

\subsection*{Step 2} Now, for each pair $(i,j)\in S$, consider the unique unit speed geodesic $\gamma_{ij}$ on $\Sigma$ from $p_i$ to $p_j$, and apply the fundamental theorem of calculus to compute
\begin{align*}
    &\langle \nabla d_{p_i}(p_j),\nabla \psi(p_j)\rangle+\langle \nabla d_{p_j}(p_i),\nabla \psi (p_i)\rangle\\
    &= \frac{d}{dt}|_{t=d(p_i,p_j)}(\psi \circ \gamma_{ij})-\frac{d}{dt}|_{t=0}(\psi \circ \gamma_{ij})\\
    &=\int_0^{d(p_i,p_j)}(\psi \circ \gamma_{ij})''(t)dt\\
    &=\int_0^{d(p_i,p_j)}(\Hess \psi)_{\gamma_{ij}(t)}(\gamma_{ij}'(t),\gamma_{ij}'(t))dt.
\end{align*}
Now, fix a cutoff function $\chi\in C^{\infty}([0,r_0))$ such that
\begin{itemize}
    \item $\chi (t)= t$ on $[0,\frac 13r_0)$.
    \item $\chi(t)=0$ if $t>\frac 23r_0$.
    \item $\chi(t)\leq t$ for all $t$.
\end{itemize}
Given a point $q\in \Sigma$ and $\gamma\in (0,1)$, we then define $\psi\in C^1(\Sigma)$ by
\begin{equation}
    \psi:=\frac{1}{1+\gamma}\chi^{1+\gamma}\circ d_q,
\end{equation}
In a moment, we will check that the Hessian $\Hess\psi$ satisfies a lower bound of the form
\begin{equation}\label{hess.bd}
    \Hess \psi \geq \gamma d_q^{\gamma-1}-C(g)d_q^{\gamma}.
\end{equation}
With this estimate in place, the preceding computation gives
\begin{eqnarray*}
    \langle \nabla d_{p_i}(p_j),\nabla \psi(p_j)\rangle+\langle \nabla d_{p_j}(p_i),\nabla \psi (p_i)\rangle&\geq &\int_0^{d(p_i,p_j)}(\gamma d_q(\gamma_{ij}(t))^{\gamma-1}-Cd_q(\gamma_{ij}(t))^{\gamma}dt\\
    &\geq & d(p_i,p_j)\cdot [\gamma \min_t d(q,\gamma_{ij}(t))^{\gamma-1}-C'(g)].
\end{eqnarray*} 
In particular, applying this in \eqref{cn.crit.2}, we arrive at an estimate of the form
$$\sum_{(i,j)\in S}\omega_i\omega_j\min_t d(q,\gamma_{ij}(t))^{\gamma-1}\leq \frac{1}{\gamma} C(g,V)n,$$
which together with the lower bound $\omega_i>\frac{\hat{c}}{2\sqrt{n}}$ gives
$$\sum_{(i,j)\in S}\min_t d(q,\gamma_{ij}(t))^{\gamma-1}\leq \frac{1}{\gamma}C(g,V)n^2$$
for any $q\in \Sigma$. Fixing any $1\leq i\leq n$ and taking $q=p_i$, we see finally that
$$\sum_{j\neq i}d(p_i,p_j)^{\gamma-1}\leq \frac{1}{\gamma}C(g,V)n^2,$$
and therefore
$$d(p_i,p_j)\geq \left(\frac{1}{\gamma}C(g,V)n^2\right)^{-\frac{1}{1-\gamma}}=c_{\gamma}(g,V)n^{-2/(1-\gamma)},$$
for any $j\neq i$, confirming the desired estimate.

\subsection*{Step 3} It remains to verify the Hessian estimate \eqref{hess.bd} for the test function $\psi$, via a fairly straightforward application of the Hessian comparison theorem. Namely, a direct computation gives
$$\Hess\psi=((\chi^\gamma \chi')\circ d_q)\Hess d_q+((\gamma\chi^{\gamma-1}(\chi')^2+\chi^\gamma \chi'')\circ d_q)\nabla d_q\otimes \nabla d_q,$$
while, on the ball $B_{\frac 23r_0}(q)$, comparing $(\Sigma,g)$ with the space forms of curvature $\min K_g$ and $\max K_g$ in the Hessian comparison theorem \cite[Theorem 1.1]{schoenyau} gives an estimate of the form
$$|\Hess d_q-\frac{1}{d_q}(g-\nabla d_q\otimes\nabla  d_q)|\leq C(g)d_q.$$ 

Now, on the complement $\Sigma\setminus B_{r_0/3}(q)$ of the ball $B_{r_0/3}(q)$, the preceding computations suffice to establish a uniform bound of the form
$$|\Hess \psi|\leq C_0(g),$$
from which \eqref{hess.bd} follows by taking $C(g)$ sufficiently large, since $d_q^{\gamma}\geq (r_0/3)^{\gamma}\geq r_0/3$ on this set. On the ball $B_{r_0/3}(q)$, recall that $\chi(d_q)=d_q$, so that 
\begin{eqnarray*}
    \Hess \psi&=&d_q^{\gamma}\Hess d_q+\gamma d_q^{\gamma-1}\nabla d_q\otimes \nabla d_q\\
    &\geq &d_q^{\gamma}\cdot \frac{1}{d_q}(g-\nabla d_q\otimes \nabla d_q)-C(g)d_q^{1+\gamma}g\\
    &&+\gamma d_q^{\gamma-1}\nabla d_q\otimes d_q\\
    &\geq &\gamma d_q^{\gamma-1}g-C(g)d_q^{1+\gamma}g.
\end{eqnarray*}
This confirms \eqref{hess.bd}, completing the proof.

\end{proof}

\subsection{The lower bound for $\inf_{\hat X_n}\cH$}\label{sect:lowerBoundInfHatH}

Next, we complete the proof of Theorem \ref{thm:existGoodMinimizers}, by establishing a sharp lower bound for $\inf_{\hat X_n}\cH$.

\begin{prop}\label{prop:infHatHLowerBound}
    Under the assumptions of Theorem \ref{thm:existGoodMinimizers}, there exists some constant $\hat C(g,V)>0$ such that for each $n\geq \hat{C}$, we have
    $$\inf_{\hat X_n} \cH>   n\cI-\hat C\log n.$$
\end{prop}
\begin{proof}
Combining Propositions \ref{prop:existenceMinimizer}, \ref{prop:tauIComparable}, and \ref{prop:p_iLowerBoundDistance}, we see now that, for every $n$ sufficiently large, the infimum $\inf_{\hat X_n}\cH$ is realized by a minimizer $(\bp,\bomega)\in\hat X_n$ satisfying the following:
\begin{itemize}
    \item $\frac 1{\hat\Lambda\sqrt n} <\omega_i<\frac {\hat \Lambda}{ \sqrt n}$ for some $\hat\Lambda(g,J)>0$
    \item $d(p_i,p_j)>  n^{-2.5}$ for any $i\ne j$.
\end{itemize}
For later use, we set $\delta_n:=  n^{-3}$, and $\epsilon_n:=\delta_n/n=n^{-4}.$ 

Consider the point $(\bp,\bomega/|\bomega|_{\ell^1})\in X_n$, and the associated probability measure
$$\nu:=\frac 1{|\bomega|_{\ell^1}}\sum_i\omega_i\delta_{p_i},$$
and introduce another probability measure 
$$ \nu_{\epsilon_n}:=\frac 1{|\bomega|_{\ell^1}}\sum_i\omega_i\frac{\chi_{B_{\epsilon_n}(p_i)}}{|B_{\epsilon_n}(p_i)|},$$
giving an $L^{\infty}$ approximation of $\nu$, where $|B_{\epsilon_n}(p_i)|$ denotes the area of the ball. We will obtain the desired lower bound by relating $\inf_{\hat X_n}\cH=\cH(\bp,\bomega)$ to $\cE(\nu_{\epsilon_n})$, via comparisons with $\cE_{\delta_n}(\nu)$ and $\cE_{\delta_n}(\nu_{\epsilon_n})$.

\subsection*{Step 1: Comparing $\inf_{\hat X_n}\cH$ and $\cH(\bp,\frac{\bomega}{|\bomega|_{\ell^1}})$.} First, we claim that we can assume that $\cH(\bp,\frac{\bomega}{|\bomega|_{\ell^1}})<0$. Indeed, using  (\ref{eq:scaleTau}), we calculate 
$$\cH(\bp,\frac{\bomega}{|\bomega|_{\ell^1}})=\frac 1{|\bomega|^2_{\ell^1}}[\cH(\bp,\bomega)+\frac 1{2\pi}\log |\bomega|_{\ell^1}],$$ 
and since Cauchy-Schwarz gives $|\bomega|_{\ell^1}\leq\sqrt n$, while  $$\cH(\bp,\bomega)=\inf_{\hat X_n}\cH<\cI n+Cn^{3/4},$$
by Proposition \ref{prop:betaNUpperBound}, it follows that for $n$ sufficiently large,  $$\cH(\bp,\frac{\bomega}{|\bomega|_{\ell^1}})<\frac{1}{|\omega|_{\ell^1}^2}[n\cI+C n^{3/4}+\frac{1}{4\pi}\log n]\leq \frac{1}{|\omega|_{\ell^1}^2}\frac{n}{2}\cI< 0.$$

Using the negativity $\cH(\bp,\frac{\bomega}{|\bomega|_{\ell^1}})<0$, by another application of (\ref{eq:scaleTau}) and the Cauchy-Schwarz inequality $|\bomega|_{\ell^1}\leq \sqrt{n}$, we have
\begin{align*}
\inf_{\hat X_n}\cH=\cH(\bp,\bomega)&=|\bomega|^2_{\ell^1}\cH(\bp,\frac{\bomega}{|\bomega|_{\ell^1}})-\frac 1{2\pi}\log |\bomega|_{\ell^1}\\
&\geq n\cH(\bp,\frac{\bomega}{|\bomega|_{\ell^1}})-\frac 1{4\pi}\log n.
\end{align*}

\subsection*{Step 2: Comparing $\cH(\bp,\frac{\bomega}{|\bomega|_{\ell^1}})$ and $\cE_{\delta_n}(\nu)$} By Lemma \ref{ed.char}, we can write
\begin{align*}
\cH(\bp,\frac{\bomega}{|\bomega|_{\ell^1}})&= \mathcal{E}_{\delta_n}(\nu)+\sum_{i=1}^n\frac{\omega_i^2}{4\pi{|\bomega|_{\ell^1}^2}}\left (1-2 \log \frac {\omega_i}{2{|\bomega|_{\ell^1}}}\right)+\frac{1}{2\pi{|\bomega|_{\ell^1}^2}}\log\delta_n.
\end{align*}
Since $\frac{1}{\hat{\Lambda}}< \omega_i\sqrt{n}<\hat{\Lambda}$ and $\delta_n=n^{-3}$, it is easy to see that
$$|\sum_{i=1}^n\frac{\omega_i^2}{4\pi{|\bomega|_{\ell^1}^2}}\left (1-2 \log \frac {\omega_i}{2{|\bomega|_{\ell^1}}}\right)+\frac{1}{2\pi{|\bomega|_{\ell^1}^2}}\log\delta_n|\leq C(g,V)\frac{\log n}{n},$$
so that
$$|\cH(\bp,\frac{\bomega}{|\bomega|_{\ell^1}})- \mathcal{E}_{\delta_n}(\nu)|\leq C\frac{\log n}{n}.$$
for a suitable constant $C=C(g,V)<\infty$. 
\subsection*{Step 3: Comparing $\cE_{\delta_n}(\nu)$ and $\cE_{\delta_n}(\nu_{\epsilon_n})$}
By definition, since 
$$n\epsilon_n=\delta_n<\frac{1}{\sqrt{n}}\min_{i\ne j}d(p_i,p_j),$$
we see that
$$\cE_{\delta_n}(\nu)= \sum_{i=1}^n\sum_{j\neq i}\frac{\omega_i\omega_j}{|\bomega|_{\ell^1}^2}G(p_i,p_j)+\sum_i\frac {\omega_i^2}{|\bomega|_{\ell^1}^2}(-\frac 1{2\pi}\log\delta_n+R(p_i,p_i)),$$
while
\begin{eqnarray*}
    \cE_{\delta_n}(\nu_{\epsilon_n})&=& \sum_{i=1}^n\sum_{j\neq i}\frac{\omega_i\omega_j}{|\bomega|_{\ell^1}^2}\frac{1}{|B_{\epsilon_n}(p_i)||B_{\epsilon_n}(p_j)|}\int_{B_{\epsilon_n}(p_i)\times B_{\epsilon_n}(p_j)}G(x,y)dxdy\\
    &&+\sum_i\frac {\omega_i^2}{|\bomega|_{\ell^1}^2}\frac 1{|B_{\epsilon_n}(p_i)|^2}\iint_{B_{\epsilon_n}(p_i)\x B_{\epsilon_n}(p_i)}(-\frac 1{2\pi}\log\delta_n+R(x,y))dxdy\\
    &=&\cE_{\delta_n}(\nu)+\sum_{i=1}^n\sum_{j\neq i}\frac{\omega_i\omega_j}{|\bomega|_{\ell^1}^2}\frac{1}{|B_{\epsilon_n}(p_i)||B_{\epsilon_n}(p_j)|}\int_{B_{\epsilon_n}(p_i)\times B_{\epsilon_n}(p_j)}[G(x,y)-G(p_i,p_j)]\\
    &&+\sum_i\frac{\omega_i^2}{|\bomega|_{\ell^1}^2}\frac{1}{|B_{\epsilon_n}(p_i)|^2}\int_{B_{\epsilon_n}(p_i)\times B_{\epsilon_n}(p_i)}[R(x,y)-R(p_i,p_i)].
\end{eqnarray*}
Now, by an application of Lemma \ref{green.lem}, we see that for any $i\neq j$, and $(x,y)\in B_{\epsilon_n}(p_i)\times B_{\epsilon_n}(p_j)$,
$$|G(x,y)-G(p_i,p_j)|\leq C\frac{\epsilon_n}{\delta_n}=\frac{C}{n}$$
and for $(x,y)\in B_{\epsilon_n}(p_i)\times B_{\epsilon}(p_i)$,
$$|R(x,y)-R(p_i,p_i)|\leq C\epsilon_n$$
for some $C=C(g,V)<\infty$. Using this in the preceding computation, we then see that
\begin{eqnarray*}
|\cE_{\delta_n}(\nu_{\epsilon_n})-\cE_{\delta_n}(\nu)|&\leq & \sum_{i=1}^n\sum_{j\neq i}\frac{\omega_i\omega_j}{|\bomega|_{\ell^1}^2}\cdot \frac{C}{n}+\sum_{i=1}^n\frac{\omega_i^2}{|\bomega|_{\ell^1}^2}C\epsilon_n\leq \frac{C}{n}.
\end{eqnarray*}

\subsection*{Step 4: Comparing $\cE_{\delta_n}(\nu_{\epsilon_n})$ and $\cE(\nu_{\epsilon_n})$}
By definition,
$$\cE(\nu_{\epsilon_n})-\cE_{\delta_n}(\nu_{\epsilon_n})=\int[G(x,y)-G_{\delta_n}(x,y)]d\nu_{\epsilon_n}(x)d\nu_{\epsilon_n}(y).$$
Since $G(x,y)-G_{\delta_n}(x,y)=0$ when $d(x,y)>\delta_n$, and 
$$\delta_n+\epsilon_n<\frac{2}{\sqrt{n}}\min_{i\neq j}d(p_i,p_j),$$
it follows from a direct computation that
\begin{eqnarray*}
    \cE(\nu_{\epsilon_n})-\cE_{\delta_n}(\nu_{\epsilon_n})&=& \sum_{i=1}^n\frac{\omega_i^2}{|\bomega|_{\ell^1}^2}\frac{1}{|B_{\epsilon_n}(p_i)|^2}\int_{B_{\epsilon_n}(p_i)\times B_{\epsilon_n}(p_i)} \frac{1}{2\pi}|\log (d(x,y)/\delta_n)|dA_g dA_g\\
    &\leq &\sum_{i=1}^n\frac{\omega_i^2}{|\bomega|_{\ell^1}^2}\frac{1}{|B_{\epsilon_n}(p_i)|^2}\cdot C(g)\epsilon_n^4|\log(\epsilon_n)|\\
    &\leq & C(g,V)\frac{\log n}{n}.
\end{eqnarray*}

\subsection*{Step 5: Conclusion} Finally, combining the estimates obtained in Steps 1 to 4, we see that
\begin{eqnarray*}
    \cI&\leq &\cE(\nu_{\epsilon_n})\\
    &\leq &\cE_{\delta_n}(\nu_{\epsilon_n})+C\frac{\log n}{n}\\
    &\leq &\cE_{\delta_n}(\nu)+2C\frac{\log n}{n}\\
    &\leq &\cH(\bp,\frac{\bomega}{|\bomega|_{\ell^1}})+3C\frac{\log n}{n}\\
    &\leq &\frac{1}{n}\inf_{\hat X_n}\cH+4C\frac{\log n}{n},
\end{eqnarray*}
for some $C=C(g,V)<\infty$, giving us an estimate of the desired form.
\end{proof}

As a byproduct of the proof, we can also prove the following variance estimate, showing that the weights $\omega_i$ lie very close to the mean $\bar{\bomega}$ on average as $n\to\infty$.

\begin{lem}\label{lem:tauIBarTau}
Under the assumptions of Theorem \ref{thm:existGoodMinimizers}, for $n\geq C(g,V)$ sufficiently large, for every minimizer $(\bp,\bomega)$ of $\cH$ on $\hat X_n$, we have
$$\sum_i(\omega_i-\bar\bomega)^2\lesssim n^{-1/4}.$$
\end{lem}
\begin{proof}
Using Step 1 and  5 of the proof of the  Proposition \ref{prop:infHatHLowerBound}, we see that
    $$\frac 1{|\bomega|^2_{\ell^1}}[\cH(\bp,\bomega)+\frac 1{2\pi}\log |\bomega|_{\ell^1}]=\cH(\bp,\frac{\bomega}{|\bomega|_{\ell^1}})> \cI-C\frac {\log n}n,$$
    and so
$$\cH(\bp,\bomega)>|\bomega|^2_{\ell^1}(\cI-C\frac{\log n}n)\geq |\bomega|^2_{\ell^1}\cI-C\log n,$$
using  that $|\bomega|_{\ell^1}^2\leq n$ by Cauchy-Schwarz. Combining this with Proposition \ref{prop:betaNUpperBound}, we deduce that
$$n\cI+C_1(g,V)n^{3/4}>|\bomega|^2_{\ell^1}\cI-C(g,V)\log n$$
for suitable constants $C(g,V), C_1(g,V)$, giving us (since $\cI<0$) an estimate of the form 
\begin{equation}\label{l1sq.lbd}
    |\bomega|^2_{\ell^1}>n-C(g,V)n^{3/4}.
\end{equation}

On the other hand, since $|\bomega|_{\ell^2}=1$, a standard computation gives
$$\sum_i(\omega_i-\bar\bomega)^2=1-2\sum_i\omega_i\bar\bomega+n\bar\bomega^2=1-n\bar\bomega^2,$$
or equivalently,
$$n\sum_i(\omega_i-\bar\bomega)^2=n-n^2\bar\bomega^2=n-|\bomega|^2_{\ell^1}.$$
Combining this with \eqref{l1sq.lbd} then gives
$$n\sum_i (\omega_i-\bar{\bomega})^2<C(g,V) n^{3/4},$$
and dividing through by $n$ yields the desired estimate.
\end{proof}

Finally,  combining Proposition \ref{prop:existenceMinimizer}, \ref{prop:tauIComparable},  \ref{prop:p_iLowerBoundDistance}, and \ref{prop:infHatHLowerBound} completes the proof of Theorem \ref{thm:existGoodMinimizers}.

\subsection{Mean field limits}\label{sect:meanFieldLimit} In this section, we 
prove Theorem \ref{thm:convMeasure}, and obtain a finer description of the equilibrium measures $\mu$ realizing the infimum $\cI$.

\begin{proof}[Proof of Theorem \ref{thm:convMeasure}]
To begin, let $\{n_1,n_2,...\}\subset \N$ be an increasing sequence and, for each $k=1,2,...$, suppose $(\bp^k,\bomega^k)$ is a minimizer for $\cH$ on $\hat X_{n_k}$. 

Without loss of generality, we can assume that each $n_k$ is greater than $\overline n (g,J)$ from Theorem \ref{thm:existGoodMinimizers}, so that each $(\bp^k,\bomega^k)$ satisfies the conclusions of that theorem, including the estimate
\begin{equation}\label{ch.upper}
    \cH(\bp^k,\bomega^k)<n_k\cI+\bar{C}(g,V)n_k^{3/4}.
\end{equation}

Next, define the probability measures 
$$\tilde{\nu}^{k}:=\frac 1 {|\bomega^k|_{\ell^1}}\sum_i\omega^{k}_i\frac{\chi_{B_{n_k^{-4}}(p^k_i)}}{|B_{n_k^{-4}}(p^{k}_i)|},$$
and observe that the estimates from Step 2 to 4 in the proof of Proposition \ref{prop:infHatHLowerBound} yield a bound of the form
\begin{equation}\label{tild.nu.comp}
    |\cH(\bp^k,\frac{\bomega^k}{|\bomega^k|_{\ell^1}})-\cE(\tilde\nu^k)|\leq C(g,V)\frac{\log n_k}{n_k}.
\end{equation}

On the other hand, combining the estimate \eqref{ch.upper} with the computation
$$\cH(\bp^n,\frac{\bomega^n}{|\bomega^n|_{\ell^1}})=\frac 1{|\bomega^n|^2_{\ell^1}}[\cH(\bp^n,\bomega^n)+\frac 1{2\pi}\log |\bomega^n|_{\ell^1}],$$
we see that
$$ \cH(\bp^k,\frac{\bomega^k}{|\bomega^k|_{\ell^1}})\leq \frac{1}{|\bomega^k|_{\ell^1}^2}(n_k\cI+Cn_k^{3/4})\leq \cI+Cn_k^{-1/4}.$$
Putting this together with \eqref{tild.nu.comp}, we then see that
$$\cE(\tilde{\nu}^k)\leq \cI+2Cn_k^{-1/4},$$
so in particular, $\tilde{\nu}^k\in\cP$ is a minimizing sequence for $\cE$. 

Passing to a further subsequence, we may assume that $\tilde{\nu}^k$ converges in the weak-* sense to a limit probability measure $\mu\in \cP$. Since the relaxed functionals $\cE_{\delta}\leq \cE$ are continuous under weak-* convergence for each $\delta>0$, we then have
$$\cE_{\delta}(\mu)=\lim_{k\to\infty}\cE_{\delta}(\tilde{\nu}_k)\leq \cI,$$
and taking the supremum over $\delta>0$ gives $\cE(\mu)=\cI$; i.e., $\mu$ is an equilbrium measure for $\cE$.

To complete the proof of Theorem \ref{thm:convMeasure}, it remains to show that the probability measures  $$\nu^k:=\frac 1 {n_k}\sum_{i=1}^{n_k}\delta_{p^k_i}$$
converge to the same equilibrium measure $\mu$; i.e., that $\nu^k-\tilde{\nu}^k\rightharpoonup^*0$ as $k\to\infty$. To this end, fix $f\in Lip(\Sigma)$, and observe that
\begin{eqnarray*}
    |\int fd\tilde{\nu}^k-\sum_i\frac{\omega_i^k}{|\bomega^k|_{\ell^1}}f(p_i^k)|&=&\frac{1}{|\bomega^k|_{\ell^1}}\sum_i\omega_i^k\frac{1}{|B_{n_k^{-4}}(p_i^k)|}|\int_{B_{n_k^{-4}}(p_i^k)}(f-f(p_i^k))|\\
    &\leq &Lip(f)\sum_i \frac{\omega_i^k}{|\bomega^k|_{\ell^1}}n_k^{-4}\\
    &=&Lip(f)n_k^{-4}.
\end{eqnarray*}
Next, note that Lemma \ref{lem:tauIBarTau} gives an estimate of the form
$$\sum_i(\frac{\omega^k_i}{\overline{\bomega^k}}-1)^2\leq \frac{C}{(\overline{\bomega^k})^2} n_k^{-1/4}\leq C n_k^{3/4},$$
which together with an application of Cauchy-Schwarz yields
$$\sum_i|\frac {\omega_i^k}{\overline{\bomega^k}}-1|\leq C\sqrt{n_k^{3/4}\cdot n_k}=C n_k^{7/8}.$$
Combining this with the preceding computation, we see that
\begin{eqnarray*}
    |\int f d\tilde{\nu}^k-\int f d\nu^k|&\leq & Lip(f)n_k^{-4}+\sum_{i=1}^n|\frac{\omega_i^k}{|\bomega^k|_{\ell^1}}-\frac{1}{n_k}||f(p_i)|\\
    &=&Lip(f) n_k^{-4}+\frac{1}{n_k}\sum_{i=1}^n|\frac{\omega_i^k}{\overline{\bomega}^k}-1||f(p_i)|\\
    &\leq &Lip(f) n_k^{-4}+\|f\|_{C^0}\frac{1}{n_k}\cdot Cn_k^{7/8}\\
    &\leq &C\|f\|_{Lip(\Sigma)}n_k^{-1/8},
\end{eqnarray*}
so that $\tilde{\nu}^k-\nu^k\to 0$ in $Lip(\Sigma)^*$ as $k\to\infty$. Hence, the weak-* limits of $\nu^k$ and $\tilde{\nu}^k$ must coincide, and are given by an equilbrium measure $\mu$ for $\cE$.

\end{proof}

We close this subsection with a description of the equilibrium measures $\mu$ associated to any invertible, indefinite Schr\"odinger operator $J=d^*d-V$ on a closed surface $(\Sigma,g)$. For classical interaction energies modeled on the Green's function of the Laplacian, results characterizing equilibrium measures date back to Frostman's work in the 1930s, and nowadays a fairly robust existence and partial regularity theory is available, closely related to the study of obstacle problems (cf. \cite[Section 2]{SerfatyLectures} and references therein). Due to the idiosyncrasies of our setting--in particular, the \emph{indefiniteness} of the Schr\"odinger operator $J$--we opt here for a relatively self-contained treatment, which the reader may compare with the classical theory.

\begin{prop}\label{prop:eq.char}
    For any invertible Schr\"odinger operator $J=d^*d-V$ with $\lambda_1(J)<0$ as above, there exists an equilibrium measure $\mu$ for $\cE$ of the form
    $$\mu=-\cI V dA_g\lfloor_K$$
    for a compact set $K\subset \{V\geq 0\}$. Moreover, in the case where $V>0$ on $\Sigma$, we have that $K=\Sigma$ and $\cI=-1/\int_{\Sigma}VdA_g$ if and only if $J$ has index one.
\end{prop}
\begin{proof}
As we previously noted, using the characterization $\cE(\mu)=\sup_{\delta>0}\cE_{\delta}(\mu)$, it is easy to see that $\cE$ is lower semi-continuous with respect to weak-* convergence on $\cP$, which together with the Banach-Alaoglu theorem implies the existence of a minimizer $\mu\in \cP$ realizing
$$\cE(\mu)=\cI.$$
Denoting by $\phi=\phi_{\mu}\in W^{1,2}(\Sigma)$ the associated potential (so that $d\mu=J\phi_\mu dA_g$), we can then check in the usual way that $\phi$ satisfies a Frostman-like inequality: namely, for any other $\nu\in \cP$ and $t>0$, the minimality of $\mu$ implies that
$$0\leq\cE((1-t)\mu+t\nu)-\cE(\mu)=2t\int G(x,y)(d\mu d\nu-d\mu d\mu)+t^2\cE(\nu),$$
and taking $t\to 0$ yields
$$\cI\leq \int G(x,y)d\mu(x) d\nu(y)=\int \phi_{\mu}d\nu,$$
with equality when $\nu=\mu$. Thus, we deduce that
\begin{equation}\label{frosty}
    \phi_{\mu}\geq \cI \text{ on }\Sigma,\text{ while }\phi_{\mu}=\cI\text{ holds }\mu\text{-almost everywhere}.
\end{equation}
Moreover, decomposing $G(x,y)$ into its positive singular part $G_+(x,y)$ and Lipschitz-continuous negative part $G_-(x,y)$, it follows from the definition of $\phi_{\mu}$ and a simple application of Fatou's lemma that $\phi_{\mu}$ is lower semi-continuous on $\Sigma$. 

By \eqref{frosty} and the lower semi-continuity of $\phi=\phi_{\mu}$, we see that the set 
$$K:=\{\phi\equiv \cI\}$$ 
is a closed set containing the support $spt(\mu)$, so that $\phi$ is smooth with $J\phi=0$ on the open set $\{\phi>\cI\}.$ For a generic $\epsilon>0$, it then follows from Sard's theorem that 
$\Omega_{\epsilon}:=\{\phi>\cI+\epsilon\}$ is a smooth (or empty) domain on which $J\phi=0$. Thus, testing the weak equation $J\phi=\mu$ against an arbitrary $\psi\in C^{\infty}(\Sigma)$ and decomposing $\Sigma=\Omega_{\epsilon}\cup\Omega_{\epsilon}^c$, we see that
\begin{eqnarray*}
    \int \psi d\mu&=&\int_{\Sigma}\langle d\phi,d\psi\rangle-V\phi\psi\\
    &=&\int_{\Omega_{\epsilon}}div(\psi d\phi)+\psi J\phi\\
    &&+\int_{\Sigma\setminus\Omega_{\epsilon}}\langle d\phi,d\psi\rangle-V\phi\psi\\
    &=&-\int_{\partial\Omega_{\epsilon}}\psi |d\phi|+\int_{\Sigma\setminus\Omega_{\epsilon}}\langle d\phi,d\psi\rangle-V\phi\psi,
\end{eqnarray*}
where in the last line we have used the facts that $J\phi=0$ on $\Omega_{\epsilon}$ and $\frac{\partial \phi}{\partial\nu}=-|d\phi|$ on $\partial\{\phi>\cI+\epsilon\}$. Next, writing $\Sigma\setminus\Omega_{\epsilon}=\{\cI\leq \phi\leq \cI+\epsilon\}$, an application of the coarea formula gives
\begin{eqnarray*}
    |\int_{\Sigma\setminus\Omega_{\epsilon}}\langle d\phi,d\psi\rangle|&=&|\int_{\cI}^{\cI+\epsilon}\left(\int_{\partial \{\phi<t\}}\frac{\partial \psi}{\partial\nu}\right)dt|\\
    &\leq &|\int_{\cI}^{\cI+\epsilon}\left(\int_{\{\phi<t\}}div(d\psi)\right)dt|\\
    &\leq &\epsilon \|d^*d\psi\|_{L^1(\Sigma)}, 
\end{eqnarray*}
and it is easy to see that
$$\lim_{\epsilon\to 0}\int_{\Sigma\setminus\Omega_{\epsilon}}V\phi \psi=\cI\int_{K}V\psi$$
Thus, passing to the limit $\epsilon\to 0$ in the preceding computation for $\int \psi d\mu$, we see that
$$\int \psi d\mu=-\cI\int_{K}V\psi dA_g-\lim_{\epsilon \to 0}\int_{\partial\Omega_{\epsilon}}|d\phi|\psi.$$
In particular, for $\psi\geq 0$, it follows that
$$0\leq \int \psi d\mu\leq -\cI\int_KV\psi dA_g,$$
so that $0\leq \mu\leq -\cI V dA_g\lfloor_K$. This inequality is already sufficient to show that $\mu=J\phi\in L^{\infty}(\Sigma)$, and consequently
$$\phi\in \bigcap_{p\in [1,\infty)}W^{2,p}(\Sigma)\subset \bigcap_{\alpha\in (0,1)}C^{1,\alpha}(\Sigma).$$
Since $\phi\in W^{2,p}(\Sigma)$, it follows moreover that $Hess(\phi)$ vanishes $dA_g$-almost everywhere in $K$, justifying the computation
$$\mu=(J\phi)dA_g=(J\phi)dA_g\lfloor_K=-\cI V dA_g\lfloor_K,$$
the desired characterization of the equilibrium measure.

In the case where $V>0$ globally on $\Sigma$, we can diagonalize $J$ with respect to the inner product $\langle u,v\rangle_{L^2(VdA_g)}$ to obtain an $L^2(VdA_g)$-orthonormal basis for $W^{1,2}(\Sigma)$ of eigenfunctions
$$Ju_i=\lambda_i Vu_i,$$
and since $J(1)=-V$, it follows that we must have $\lambda_1=-1$ and 
$$u_1\equiv c_1:=\left(\int VdA_g\right)^{-1/2}.$$
Moreover, for the other eigenfunctions $u_i$, it follows from the $L^2(VdA_g)$-orthogonality that $\int u_i VdA_g=0$, and therefore
$$\int Ju_i dA_g=0\text{ for }i\geq 2.$$
In particular, writing any $u\in W^{2,2}(\Sigma)$ as 
$$u=\sum_{i=1}^{\infty}a_i u_i,$$
it follows that $Ju$ is a probability measure if and only if
\begin{equation}\label{cstrt}
-a_1c_1V+\sum_{i=2}^{\infty}a_i\lambda_i V u_i\geq 0\text{ and }-\int a_1c_1V dA_g=1,
\end{equation}
where the latter condition necessarily fixes $a_1=-c_1$.
Thus, we can realize $\cI$ in this case as the infimum
$$\cI=-c_1^2+\inf\{\sum_{i=2}^{\infty}\lambda_i a_i^2\mid c_1^2+\sum_{i=2}^{\infty}\lambda_ia_i\geq 0\}.$$
If $\lambda_2>0$, it follows immediately that
$$\cI=-c_1^2=-\frac{1}{\int VdA_g},$$
with equilibrium measure $\mu=-\cI V dA_g.$

On the other hand, if $J$ has index $k\geq 2$, then by choosing $a_2,\ldots,a_k$ sufficiently small, we can arrange that $\sum_{i=2}^k\lambda_i a_i^2<0$ while preserving the condition $c_1^2+\sum_{i=2}^k\lambda_ia_i\geq 0$, so that $\cI<-c_1^2$ in this case.

\end{proof}

\begin{remark}
Note that the preceding analysis implies that the necks of any minimal doubling constructed by $\hat{\cH}$-minimization over a surface of index $\geq 2$ \emph{never} equidistribute, concentrating instead in a proper subset $K\neq \Sigma$. Thus, for example, the periodic doublings of the Clifford torus constructed in \cite{KapYang} are not expected to arise from $\hat\cH$ minimization, and should instead correspond to higher-energy critical points for $\cH$.
\end{remark}

\subsection{Nondegeneracy in the weight parameters}\label{sect:wt.nondeg}

We come now to the proof of Proposition \ref{prop:wt.nondeg}, giving some useful a priori lower bounds on the Hessian $D^2\cH(\bp,\btau)$ of a critical point $(\bp,\btau)\in X_n$ for which $(\bp,\btau/|\btau|_{\ell^2})$ is $\cH$-minimizing in $\hat X_n$. 

Suppose once again that we're in the setting of Theorem \ref{thm:existGoodMinimizers}, and let $(\bp,\bomega)\in \hat X_n$ minimize $\cH$ on $\hat X_n$ for some $n\geq \bar{n}$, and let $(\bp,\btau)\in X_n$ be the associated critical point for $\cH$, so that
\begin{equation}\label{btau.char}
    \btau=e^{2\pi \inf_{\hat X_n}\cH-1/2}\bomega
\end{equation}
and
\begin{equation}\label{h.max.char}
    \cH(\bp,\btau)=\frac{1}{4\pi}|\btau|^2.
\end{equation}
Writing
$$\cP=T_{(\bp,\btau)}X_n=\bigoplus_{i=1}^nT_{p_i}\Sigma\oplus \mathbb{R}^n,$$
recall that the Hessian $D^2\cH(\bp,\btau)$ gives a self-adjoint map
$$L:=D^2\cH(\bp,\btau): \cP \to \cP$$
with respect to the natural $\ell^2$ inner product on $\cP$. Now, from the proof of Proposition \ref{prop:criticalPointH}, recall that
$$\nabla^{\sigma_i}\cH(\bq,\bsigma)=2\sum_{j\neq i}\sigma_jG(q_i,q_j)+2\sigma_iR(q_i,q_i)-\frac{\sigma_i}{\pi}\log(\tau_i/2),$$
and differentiating again and evaluating at the critical point $(\bp,\btau)$, we see that for any $\bs\in \mathbb{R}^n$ and $(\bv,\bt)\in \cP$,
\begin{eqnarray*}
    \langle L(0,\bs), (\bv,\bt)\rangle&=&\sum_{i=1}^n\sum_{j\neq i}2s_i\langle d\phi_{p_i}(p_j),\tau_j v_j\rangle\\
    &&+\sum_{i=1}^ns_i\left(\sum_{j\neq i}2\tau_j\langle d\phi_{p_j}(p_i),v_i\rangle+2\sum_{i=1}^n\tau_i\langle dR_D(p_i),v_i\rangle\right)\\
    &&+\sum_{i=1}^n\sum_{j\neq i}2G(p_i,p_j)s_it_j\\
    &&+\sum_{i=1}^n(2R_D(p_i)-\frac{1}{\pi}(\log(\tau_i/2)+1))s_it_i.
\end{eqnarray*}
Since $(\bp,\btau)$ is a critical point for $\cH$, note that Proposition \ref{prop:criticalPointH} gives
\begin{equation}\label{v.crit}
    \sum_{j\neq i}\tau_jd\phi_{p_j}(p_i)+\frac{\tau_i}{2}dR_D(p_i)=0,
\end{equation}
and applying this in the second line of the preceding computation, we get the simpler expression
\begin{eqnarray*}
\langle L(0,\bs), (\bv,\bt)\rangle&=&\sum_{i=1}^n\sum_{j\neq i}2s_i\langle d\phi_{p_i}(p_j),\tau_j v_j\rangle+\sum_{i=1}^ns_i\tau_i\langle d R_D(p_i),v_i\rangle\\
    &&+\sum_{i=1}^n\sum_{j\neq i}2G(p_i,p_j)s_it_j+\sum_{i=1}^n(2R_D(p_i)-\frac{1}{\pi}(\log(\tau_i/2)+1))s_it_i.
\end{eqnarray*}
Taking $\bt=0$ and recalling the definition of $S_{\btau}:\cP\to\cP$, it follows that
$$|\langle S_{\btau}LS_{\btau}(0,\bs),(\bv,0)\rangle|\leq |\sum_{i=1}^n\sum_{j\neq i}2s_i\langle d\phi_{p_i}(p_j),v_j\rangle|+|\sum_{i=1}^ns_i\langle dR_D(p_i),v_i\rangle|.$$
Applying Lemma \ref{green.lem} to estimate $d\phi_{p_i}(p_j)$ and $dR_D(p_i)$, and noting that $d(p_i,p_j)>n^{-3}$ by Theorem \ref{thm:existGoodMinimizers}, it follows that
\begin{eqnarray*}
    |\langle S_{\btau}LS_{\btau}(0,\bs),(\bv,0)\rangle|&\leq &\sum_{i=1}^n\sum_{j\neq i}\frac{C}{d(p_i,p_j)}|s_i||v_j|+C\sum_{i=1}^n|s_i||v_i|\\
    &\leq & Cn^3\sum_{i,j=1}^n|s_i||v_j|.
\end{eqnarray*}
A simple application of Cauchy-Schwarz then gives
\begin{equation}\label{sls.cross}
    |\langle S_{\btau}LS_{\btau}(0,\bs),(\bv,0)\rangle|\leq Cn^4|\bs||\btau|.
\end{equation}

Next, define an auxiliary function $F: X_n\to \mathbb{R}$ by $F(\bq,\bsigma):=\cH(\bq,\bsigma/|\bsigma|)$, and note that \eqref{eq:scaleTau} gives
$$F(\bq,\bsigma)=|\bsigma|^{-2}\cH(\bq,\bsigma)+\frac{1}{2\pi}\log |\bsigma|.$$
Since $(\bp,\btau)$ is a critical point for $\cH$, a straightforward computation gives
\begin{eqnarray*}
        D^2F(\bp,\btau)&=&|\btau|^{-2}D^2\cH(\bp,\btau)-\left(\frac{1}{2\pi|\btau|^2}-\frac{2\cH(\bp,\btau)}{|\btau|^4}\right)\Pi_{0\oplus \mathbb{R}^n}\\
        &&+\left(\frac{8\cH(\bp,\btau)}{|\btau|^2}-\frac{1}{\pi}\right)\frac{(0,\btau)\otimes (0,\btau)}{|\btau|^4},
    \end{eqnarray*}
    and combining this with \eqref{h.max.char}, it follows that
\begin{equation}\label{f.vs.h}
    D^2F(\bp,\btau)=|\btau|^{-2}D^2\cH(\bp,\btau)+\frac{1}{\pi|\btau|^2}\frac{(0,\btau)\otimes (0,\btau)}{|\btau|^2}.
\end{equation}

We now derive two important consequences of this computation.

\begin{lem}\label{lem:spanTau}
    At the critical point $(\bp,\btau)$, the vector $(0,\btau)\in \cP$ is an eigenvector for $L=D^2\cH(\bp,\btau)$ with
    $$S_{\btau}LS_{\btau}(0,\btau)=L(0,\btau)=-\frac{1}{\pi}(0,\btau).$$
\end{lem}
\begin{proof}
Since the auxiliary function $F$ defined above is invariant under rescaling the $\bsigma$ parameter, we see that 
$$S_{\btau}D^2FS_{\btau}(0,\btau)=S_{\btau}D^2F(\bp,\btau)(0,\btau)=D^2F(\bp,\btau)=0,$$
which together with \eqref{f.vs.h} implies that
\begin{eqnarray*}
    S_{\btau}D^2\cH(\bp,\btau)S_{\btau}(0,\btau)&=&D^2\cH(\bp,\btau)(0,\btau)\\
    &=&|\btau|^2D^2F(\bp,\btau)(0,\btau)-\frac{1}{\pi}(0,\btau)\\
    &=&-\frac{1}{\pi}(0,\btau),
\end{eqnarray*}
as claimed.
\end{proof}

Next, we use the $\cH$-minimizing property of $(\bp,\bomega)$ in $\hat X_n$ to obtain nonnegativity of $S_{\btau}LS_{\btau}$ in directions orthogonal to $(0,\btau)$.

\begin{lem}\label{hess.nonneg}
    For $(\bp,\bomega)$ and $(\bp,\btau)$ as above, the Hessian $L=D^2\cH(\bp,\btau)$ is positive semidefinite on the hyperplane $\langle (0,\btau)\rangle^{\perp}\subset \cP$ orthogonal to $(0,\btau)$. In particular, $S_{\btau}LS_{\btau}\geq 0$ on $\langle (0,\btau)\rangle^{\perp}$ as well.
\end{lem}
\begin{proof}
    Since $(\bp,\btau/|\btau|)$ minimizes $\cH$ on $\hat X_n$, we see that $(\bp,\btau)$ minimizes the auxiliary function $F$ on $X_n$, and it follows from \eqref{f.vs.h} that for any $(\bv,\bs)\in \langle (0,\btau)\rangle^{\perp}$,
    $$\langle L(\bv,\bs),(\bv,\bs)\rangle=|\btau|^2\langle D^2F(\bp,b\tau)(\bv,\bs),(\bv,\bs)\rangle \geq 0,$$
    as claimed. Since the map $S_{\btau}$ is obviously positive definite, the nonnegativity of $S_{\btau}LS_{\btau}$ on $\langle (0,\btau)\rangle^{\perp}$ follows immediately as well.    
\end{proof}

With these preliminary results out of the way, we now complete the proof of Proposition \ref{prop:wt.nondeg} by obtaining suitable lower bounds on terms of the form $\langle L(0,\bs),(0,\bs)\rangle$, via another mean field-type analysis.

\begin{proof}[Proof of Proposition \ref{prop:wt.nondeg}]

To prove the proposition, we need to establish--under the additional assumption that $J$ has index one--an estimate of the form
$$\langle S_{\btau}LS_{\btau}(\bv,\bs),(\bv,\bs)\rangle \geq \max\{0,c_0 n|\bs|^2-n^5|\bv||\bs|\}$$
for all $(\bv,\bs)\in \langle (0,\btau)\rangle^{\perp}$ and $n\geq n_0(g,V)$ is sufficiently large. 

Note that Lemma \ref{hess.nonneg} already shows that $S_{\btau}LS_{\btau}$ is nonnegative definite in directions orthogonal to $(0,\btau)$, and for any $(\bv,\bs)\in \langle (0,\btau)\rangle^{\perp}$, it also follows that
\begin{eqnarray*}
    \langle S_{\btau}LS_{\btau}(\bv,\bs),(\bv,\bs)\rangle&=&\langle S_{\btau}LS_{\btau}(\bv,0),(\bv,0)\rangle+2\langle S_{\btau}LS_{\btau}(0,\bs),(\bv,0)\rangle\\
    &&+\langle S_{\btau}LS_{\btau}(0,\bs),(0,\bs)\rangle\\
    &\geq &2\langle S_{\btau}LS_{\btau}(0,\bs),(\bv,0)\rangle+\langle S_{\btau}LS_{\btau}(0,\bs),(0,\bs)\rangle\\
    &\geq &Cn^4|\bs||\btau|+\langle S_{\btau}LS_{\btau}(0,\bs),(0,\bs)\rangle,
\end{eqnarray*}
where in the last line we use \eqref{sls.cross}. Denoting by $Q: \mathbb{R}^n\times \mathbb{R}^n\to \mathbb{R}$ the quadratic form
$$Q(\bs,\bs):=\langle S_{\btau}LS_{\btau}(0,\bs),(0,\bs)\rangle,$$
we then see that the desired estimate will follow once we have a bound of the form
$$\inf\{\frac{Q(\bs,\bs)}{|\bs|^2}\mid 0\neq \bs\in \langle \btau\rangle^{\perp}\}\geq c_0 n$$
for some $c_0(g,V)>0$. In particular, since $\btau$ must span the first eigenspace for $Q$ by Lemma \ref{lem:spanTau}, this is equivalent to proving a lower bound
\begin{equation}\label{lam.2.q}
    \lambda_2(Q)\geq c_0n
\end{equation}
for the second eigenvalue $\lambda_2(Q)$. 

To prove \eqref{lam.2.q}, recall that $d(p_i,p_j)>n^{-3}$ for $n$ sufficiently large, and--as in the proof of Theorem \ref{thm: asymptotics}--consider the $L^{\infty}$ functions
$$\varphi_i:=\frac{\chi_{B_{n^{-4}}(p_i)}}{|B_{n^{-4}}(p_i)|}.$$
Since the functions $\varphi_i$ have disjoint supports, the map $T: \mathbb{R}^n\to L^{\infty}(\Sigma)$ given by $T(\bs):=\sum_{i=1}^ns_i\varphi_i$ is clearly injective. 

Now, given $\bs\in \mathbb{R}^n$, recall that
$$Q(\bs,\bs)=\sum_{i=1}^n\sum_{j\neq i}2G(p_i,p_j)s_is_j+\sum_{i=1}^n(2R_D(p_i)-\frac{1}{\pi}(\log(\tau_i/2)+1))s_i^2.$$
Moreover, by \eqref{btau.char} and the estimates of Theorem \ref{thm:existGoodMinimizers}, we see that 
\begin{eqnarray*}
    \log(\tau_i/2)&=&2\pi\inf_{\hat X_n}\cH-1/2+\log(\omega_i)+\log(1/2)\\
    &\leq &2\pi n\cI+\overline{C}n^{3/4},
\end{eqnarray*}
and since $R_D\in C^0(\Sigma)$, it follows that
$$Q(\bs,\bs)\geq \sum_{i=1}^n\sum_{j\neq i}2G(p_i,p_j)s_is_j-n\cI|\bs|^2$$
for $n\geq n_0(g,V)$ sufficiently large.

As a consequence, for any $\bs\in \mathbb{R}^n$, setting 
$$u:=J^{-1}(T(\bs))=\sum_i s_i\int G(x,y)\varphi_i(y)dA_g(y)\in W^{2,2}(\Sigma),$$
we compute
\begin{eqnarray*}
    Q(\bs,\bs)-2\int u Ju dA_g&=&Q(\bs,\bs)-2\sum_{i,j=1}^ns_is_j\int_{\Sigma\times \Sigma} G(x,y)\varphi_i(x)\varphi_j(y)dA_g dA_g\\
    &\geq &2\sum_{i=1}^n\sum_{j\neq i}s_is_j\int_{\Sigma\times \Sigma} (G(p_i,p_j)-G(x,y))\varphi_i(x)\varphi_i(y)dA_gdA_g\\
    &&-n\cI |\bs|^2-\sum_{i=1}^ns_i^2\int_{\Sigma\times \Sigma}G(x,y)\varphi_i(x)\varphi_i(y)dA_gdA_g.
\end{eqnarray*}
Arguing much as in Step 3 of the proof of Theorem \ref{thm: asymptotics}, we note that since $\varphi_i$ vanishes outside of $B_{n^{-4}}(p_i)$ and $d(p_i,p_j)>n^{-3}$ for $i\neq j$, Lemma \ref{green.lem} gives an estimate of the form
$$\int_{\Sigma\times \Sigma}|G(p_i,p_j)-G(x,y)|\varphi_i(x)\varphi_j(y)\leq \frac{C}{n}$$
for every pair $i\neq j$. Likewise, a straightforward computation using Lemma \ref{green.lem}, similar to Step 4 of the proof of Theorem \ref{thm: asymptotics}, gives an estimate of the form
$$\int_{\Sigma\times \Sigma}G(x,y)\varphi_i(x)\varphi_i(y)dA_g dA_g\leq C \log n.$$
Using these in the estimate above, we find that
\begin{eqnarray*}
    Q(\bs,\bs)-2\int u Ju dA_g&\geq &-n\cI|\bs|^2-\frac{C}{n}\sum_{i,j=1}^ns_is_j-C|\bs|^2\log n\\
    &\geq &n|\cI||\bs|^2-C|\bs|^2-C(\log n)|\bs|^2\\
    &\geq &\frac{n}{2}|\cI||\bs|^2
\end{eqnarray*}
for $n\geq n_0(g,V)$ sufficiently large. Setting $c_0(g,V):=\frac{1}{2}|\cI(g,V)|>0$, we can now prove \eqref{lam.2.q} as follows. For any two-dimensional subspace $W\subset \mathbb{R}^n$, note that $J^{-1}T(W)\subset W^{2,2}(\Sigma)$ is also two-dimensional, by the injectivity of $J^{-1}T$. Thus, since $\lambda_2(J)>0$, there must be some $\bs\in W$ such that $u=J^{-1}T\bs$ satisfies
$$\int u Ju dA_g \geq 0,$$
and it follows from the preceding estimates that
$$\frac{Q(\bs,\bs)}{|\bs|^2}\geq c_0 n+2|\bs|^{-2}\int u Ju dA_g\geq c_0 n.$$
The bound \eqref{lam.2.q} then follows immediately from the min-max characterization of $\lambda_2(Q)$, completing the proof.

\end{proof}

\section{Proof of main theorems}\label{sect:MainProof}

In this section, we complete the proofs of Theorems \ref{thm:main_doubling} and \ref{thm: asymptotics}, and Corollary \ref{thm:maincor}, by proving genericity for the set of metrics such that the $\hat{\cH}$-minimizing critical points associated to the Jacobi operator of any index one minimal surface satisfy the hypotheses of Theorem \ref{thm:doubling.precise} along some subseqence $n_k\to\infty$. 

\subsection{Set-up}\label{sect:proofMainSetUp}
Let $N$ be a fixed closed, smooth 3-manifold. For each $k\in\{4,5,...\}$, let $\Gamma_k$ be the space of $C^k$ Riemannian metrics on $N$, and $\cM^k$  the set of all pairs $(g,\Sigma)$, where $g\in \Gamma_k$ and $\Sigma$ is any embedded minimal surface in $(N,g)$, endowed with the natural topology such that $(g_j,\Sigma_j)\to (g,\Sigma)$ in $\cM^k$ if and only if $\|g_j-g\|_{C^k(N,g)}\to 0$ and $\Sigma_j=F_j(\Sigma)$ for a sequence of maps $F_j:\Sigma\to N$ converging in $C^{k-1}(\Sigma)$ to the inclusion $\iota: \Sigma\to N$. Likewise, define $\Gamma_{\infty}:=\bigcap_{k\geq 4}\Gamma_k$ to be the space of smooth metrics and $\cM^{\infty}=\bigcap_{k\geq 4}\cM^k$ the space of smooth pairs $(g,\Sigma)$, endowed with the topology of smooth convergence. For $k\in \{4,5,\ldots\}\cup\{\infty\}$, let $\pi:\cM^k\to\Gamma_k$ denote the projection map given by $\pi(g,\Sigma):=g$. 

By White's results in \cite{Whi91,white2017bumpy}, for each integer $k\geq 4$, $\cM^k$ has the structure of a separable $C^{k-2}$ Banach manifold, and the map $\pi$ is a $C^{k-2}$ Fredholm map with Fredholm index 0. Moreover, $(g,\Sigma)$ is a critical point of $\pi$ if and only if $\Sigma$   has some non-trivial Jacobi field in $(N,g).$ Thus, for every $k\geq 4$, and  any {\it non-degenerate},  embedded minimal surface $\Sigma$ in $(N,g)$, there exist some open neighborhoods $\cU^k(g,\Sigma)\subset\cM^k$ around $(g,\Sigma)$ and $\cV^k(g,\Sigma)\subset\Gamma_k$ around $g$, such that $$\pi|_{\cU^k(g,\Sigma)}:\cU^k(g,\Sigma)\to \cV^k(g,\Sigma)$$ is a homeomorphism.
Note the obvious inclusions $\cM^4\supset\cM^5\supset ...\supset \cM^\infty$ and $\Gamma_4\supset\Gamma_5\supset ...\supset \Gamma_\infty$.

Now, for each $g\in \Gamma_\infty$, and each  non-degenerate, embedded minimal surface in $(N,g)$, we define the sets
$$\cU(g,\Sigma):=\cU^4(g,\Sigma)\cap \cM^\infty\subset\cM^\infty$$
and 
$$\cV(g,\Sigma):=\cV^4(g,\Sigma)\cap \Gamma_\infty\subset\Gamma_\infty,$$
and record the following properties.
\begin{itemize}
\item $\pi|_{\cU(g,\Sigma)}:\cU(g,\Sigma)\to \cV(g,\Sigma)$ is a homeomorphism in the $C^4$-topology.
\item $\cU(g,\Sigma)$ and $\cV(g,\Sigma)$ are open in the $C^{\infty}$ topology on $\cM^{\infty}$ and $\Gamma_{\infty}$.
\item For each $h\in \cV(g,\Sigma)$, denote by $\Sigma_h$ the unique embedded minimal surface in $(N,h)$ for which $(h,\Sigma_h)\in \cU(g,\Sigma)$. By assuming $\cU^4(g,\Sigma)$ and $\cV^4(g,\Sigma)$ to be sufficiently small, we can further ensure that $\Sigma_h$ remains non-degenerate, with the same Morse index as $\Sigma$, for each $h\in \cV(g,\Sigma)$.
\end{itemize}

Now, given any {\it two-sided}, embedded minimal surface $\Sigma$ in $(N,g)$, we may consider its Jacobi operator $J=J_{\Sigma}=d^*_gd-V$, where $V=|A|^2+\Ric(\nu,\nu)$ and $\nu$ is a unit normal vector field on $\Sigma$. With respect to the metric $g|_\Sigma$ and the Schr\"odinger operator $J$ on $\Sigma$, we may define the interaction energies  $\cH_{(g,\Sigma)}:X_n\to (0,\infty)$ as in Definition \ref{h.def} and $\hat \cH_{(g,\Sigma)}:\hat X_n\to (0,\infty)$, as in Definition \ref{hatH.def}. Note, the definition of these two functions depend not only on the Riemannian surface $(\Sigma,g|_\Sigma)$, but also the embedding of $\Sigma$ into $(N,g)$, through the potential $V=|A|^2+\Ric(\nu,\nu)$. Below, we sometimes omit the subscripts in  $\cH_{(g,\Sigma)}$ and $\hat \cH_{(g,\Sigma)}$ where there is no cause for confusion.

As a consequence of the preceding discussion, note that the subset
$$\cM_1^{\infty}:=\{(g,\Sigma)\subset \cM^{\infty}\mid \Sigma\text{ two-sided and }\lambda_1(J_{\Sigma})<0<\lambda_2(J_{\Sigma})\}$$
of $\cM^{\infty}$ consisting of pairs $(g,\Sigma)$ where $\Sigma$ is a two-sided, nondegenerate, embedded minimal surface with Morse index one is open in $\cM^{\infty}$.

We now introduce a few more definitions. For simplicity, we reserve the notation $|\cdot|$ for the $\ell^2$-norm throughout \S \ref{sect:MainProof}. 
\begin{definition}\label{def:propertyD}
For any positive integer $n$, we say that $(g,\Sigma)\in \cM_1^{\infty}$ satisfies {\it Property $\Db(g,\Sigma,n)$}  if the following hold:
\begin{enumerate}
    \item Any minimizing sequence for $\cH_{(g,\Sigma)}$ in $\hat X_n$ is precompact in $\hat X_n$. (Hence, there exists at least one critical point $(\bp,\btau)\in X_n$ of energy $\cH(\bp,\btau)=\inf_{\hat{X}_n}\hat \cH$.)
    \item {\it Every} critical point $(\bp,\btau)\in X_n$ of $\cH_{(g,\Sigma)}$ such that $\cH(\bp,\btau)=\inf_{\hat X_n}\hat \cH$ satisfies items (1) - (4) of Theorem \ref{thm:doubling.precise}.
\end{enumerate}
We then set
$$\cS_n:=\{(g,\Sigma)\in \cM_1^{\infty}\mid D(g,\Sigma,n)\text{ holds}\}.$$
\end{definition}

\begin{remark}\label{ind1.rk}
    By definition of $\cM_1^{\infty}$ and domain monotonicity of Dirichlet eigenvalues for the Jacobi operator $J$, note that for any $(g,\Sigma)\in \cM_1^{\infty}$, we have
    $$\lambda_2(J;\Omega)\geq \lambda_2(J;\Sigma)>0$$
    for any domain $\Omega\subset \Sigma$. Thus, item (4) of Theorem \ref{thm:doubling.precise} holds automatically for any critical point of $\cH_{(g,\Sigma)}$ satisfying items (1)-(3), so we can omit condition (4) from the definition of $\Db(g,\Sigma,n)$ and $\cS_n$ with no loss of generality.
\end{remark}

The main result of this section can then be expressed as follows.

\begin{thm}\label{thm:good.generic}
Let $\cG\subset \Gamma_{\infty}$ be the set of metrics such that, for every two-sided, index one minimal surface $\Sigma$ in $(N,g)$, we have 
$$(g,\Sigma)\in \bigcap_{k\in \mathbb{N}}\bigcup_{n\geq k}\cS_n.$$
Then $\cG$ is comeager in $\Gamma_{\infty}$.
\end{thm}

Before discussing the proof, let us note how Theorems \ref{thm:main_doubling} and \ref{thm: asymptotics}, and Corollary \ref{thm:maincor} from the introduction follow.

\begin{proof}[Proof of Theorems \ref{thm:main_doubling} and \ref{thm: asymptotics}]
Let $g\in \cG$, and let $\Sigma$ be any two-sided, index one minimal surface in $(N,g)$. Since $(g,\Sigma)\in \bigcap_{k\in \mathbb{N}}\bigcup_{n\geq k}\cS_n,$ it follows that there exists a sequence of integers $n_k\to\infty$ such that $D(g,\Sigma,n_k)$ holds for every $k$. Thus, for $k$ large enough that $n_k\geq \tilde{n}(g,\Sigma)$, there exists at least one critical point $(\bp^k,\btau^k)\in X_{n_k}$ realizing $\cH(\bp^k,\btau^k)=\inf_{\hat X_{n_k}}\hat \cH$, which then satisfies the hypotheses of Theorem \ref{thm:doubling.precise}, since $D(g,\Sigma,n_k)$ holds. 

Thus, for $k$ sufficiently large, applying Theorem \ref{thm:doubling.precise} to this critical point $(\bp^k,\btau^k)$ yields a minimal doubling $M_k$ of $\Sigma$ with $n_k$ necks centered near the points $p_1,\ldots,p_{n_k}$, with an area estimate of the form
$$|\area(M_k)-(2\area(\Sigma)-\frac{1}{4}\cH(\bp^k,\btau^k))|\leq \frac{1}{n_k}\cH(\bp^k,\btau^k),$$
 and Theorem \ref{thm:main_doubling} follows immediately. Moreover, the convergence of the measures $\frac{1}{n_k}\sum_{i=1}^{n_k}\delta_{p_i}$ to an equilibrium measure $\mu$ then follows from Theorem \ref{thm:convMeasure}, while the preceding area estimate together with \eqref{hath.asymp} gives
 \begin{eqnarray*}
\lim_{k\to\infty}\frac{\log (2\area(\Sigma)-\area(M_k))}{n_k}&=&\lim_{k\to\infty}\frac{\log \cH(\bp^k,\btau^k)}{n_k}\\
&=&4\pi \cI,
 \end{eqnarray*}
completing the proof of Theorem \ref{thm: asymptotics}.
\end{proof}

The proof of Corollary \ref{thm:maincor} is then straightforward.

\begin{proof}[Proof of Corollary \ref{thm:maincor}]

By the proof of Theorem \ref{thm:good.generic}, the set $\cG$ is contained in the the set of bumpy metrics on $N$. For any $g\in \cG$, since $g$ is bumpy, the results of \cite{MN16,Zho20,MN21} (see in particular \cite[Theorem 8.3]{MN21}) imply that there exists some non-degenerate, two-sided, embedded minimal surface $\Sigma$ with Morse index one and area bounded above by the first Almgren-Pitts width $\omega_1(N,g)$ of $(N,g)$. Then, applying Theorem \ref{thm:main_doubling} to $(N,g)$ and $\Sigma$, we obtain a sequence of minimal surface doublings converging to $2\Sigma$ as varifolds, with genus going to infinity and area governed by Theorem \ref{thm: asymptotics}. This proves Corollary \ref{thm:maincor}.
\end{proof}

The rest of the section is devoted to the proof of Theorem \ref{thm:good.generic}. A key technical ingredient is a perturbation argument showing that $\bigcup_{n\geq k}\cS_n$ is dense in $\cM_1^{\infty}$. More precisely, we show the following.

\begin{thm}\label{thm:perturb}
    For every $(g,\Sigma)\in \cM_1^{\infty}$, there is a constant $n_0(g,\Sigma)$ such that for every $n\geq n_0$, there is a conformal metric $g_n\in [g]$ such that $(g_n,\Sigma)\in \cS_n$, and $g_n\to g$ smoothly as $n\to\infty$.
\end{thm}

In the next subsection, we explain how to prove Theorem \ref{thm:good.generic} assuming Theorem \ref{thm:perturb}, before turning to the proof of the perturbation result.

\subsection{Proof of Theorem \ref{thm:good.generic}}
As our first step in the proof of Theorem \ref{thm:good.generic}, we will check that, for every $k\in \mathbb{N}$, the set
$$\bigcup_{n\geq k}\cS_n$$
is open and dense in $\cM_1^{\infty}$. Assuming Theorem \ref{thm:perturb}, whose proof we postpone to a later subsection, the denseness of $\bigcup_{n\geq k}\cS_n$ follows automatically, so it remains to check the following.

\begin{lem}\label{sn.open}
    For each $n\in \mathbb{N}$, the set $\cS_n$ is open in $\cM_1^{\infty}$.
\end{lem}
\begin{proof}
Let $(g_0,\Sigma_0)\in \cS_n$, and let $(g_j,\Sigma_j)\in \cM_1^{\infty}$ be a sequence converging smoothly to $(g_0,\Sigma_0)$. To prove the lemma, it suffices to show that $(g_j,\Sigma_j)\in \cS_n$ for $j$ sufficiently large. Since the convergence $\Sigma_j\to \Sigma_0$ is smooth, note that we can find a sequence of diffeomorphisms $\Phi_j:N\to N$ converging smoothly to the identity $I: N\to N$ such that $\Phi_j(\Sigma_0)=\Sigma_j$. Thus, up to replacing the metrics $g_j$ with the isometric ones $\tilde{g}_j=\Phi_j^*g_j$, we can assume without loss of generality that $\Sigma_j=\Sigma_0=:\Sigma$ for every $j$.

First, since $g_j$ converges smoothly to $g_0$, it follows from the analysis of Appendix \ref{green.app} (see Remark \ref{smth.green.cvg}) that the Green's functions $G_{(g_j,\Sigma)}$ and diagonal Robin's functions $R_D^j$ for the Jacobi operators $J_{\Sigma_j}$ satisfy
\begin{equation}\label{c0.green.cvg}
    \lim_{j\to\infty}\|G_{(g_j,\Sigma)}-G_{(g_0,\Sigma)}\|_{C^0(\Sigma\times \Sigma)}=0,
\end{equation}
while
\begin{equation}\label{eq:smth.green.cvg}
    \lim_{j\to\infty}\|G_{(g_j,\Sigma)}-G_{(g_0,\Sigma)}\|_{C^l(K)}+\|R_D^j-R_D^0\|_{C^l(\Sigma)}=0
\end{equation}
for any $l\in \mathbb{N}$ and any compact subset $K\subset \Sigma\times \Sigma\setminus \{(x,x)\mid x\in \Sigma\}$. In particular, it follows from the definition of $\cH$ that $\cH_{(g_j,\Sigma)}-\cH_{(g_0,\Sigma)}\to 0$ uniformly on $\hat{X}_n$, and $\cH_{(g_j,\Sigma)}\to \cH_{(g_0,\Sigma)}$ in $C^{\infty}_{\loc}(X_n).$ 

We now argue that $(g_j,\Sigma)$ satisfies item (1) of Definition \ref{def:propertyD} for large $j$. If this were not true, then after passing to a subsequence, we could find for each $j$ a point $$(\bp^j,\bomega^j)\in \hat{X}_n$$
satisfying 
$$\cH_{(g_j,\Sigma)}(\bp^j,\bomega^j)\leq \inf_{\hat{X}_n}\cH_{(g_j,\Sigma)}+\frac{1}{j}$$
and
\begin{equation}\label{bdry.cvg}
\dist((\bp^j,\bomega^j),\partial \hat{X}_n)<1/j.
\end{equation}
On the other hand, it follows from the \emph{uniform} convergence $\cH_{(g_j,\Sigma)}-\cH_{(g_0,\Sigma)}\to 0$ on $\hat{X}_n$ that
$$\lim_{j\to\infty}(\cH_{(g_j,\Sigma)}(\bp^j,\btau^j)-\cH_{(g_0,\Sigma)}(\bp^j,\bomega^j))=0$$
and
$$\lim_{j\to\infty}\inf_{\hat{X}_n}\cH_{(g_j,\Sigma)}=\inf_{\hat{X}_n}\cH_{(g_0,\Sigma)},$$
which implies that $(\bp^j,\bomega^j)\in \hat{X}_n$ must be a minimizing sequence for $\cH_{(g_0,\Sigma)}$. Since $D(g_0,\Sigma,n)$ holds, item (1) of Definition \ref{def:propertyD} then implies that $(\bp^j,\bomega^j)$ must be precompact in $\hat{X}_n$, contradicting \eqref{bdry.cvg}. Thus, $(g_j,\Sigma_j)$ must satisfy item (1) of Definition \ref{def:propertyD} for $j$ sufficiently large.

It remains to check item (2) of Property $D(g_j,\Sigma,n)$ for $j$ sufficiently large. Indeed, if this were not the case, then after passing to a subsequence, we could find a sequence $(\bp^j,\btau^j)\in X_n$ of least-energy critical points for $\cH_{(g_j,\Sigma)}$ for which one of the conditions (1)-(3) of Theorem \ref{thm:doubling.precise} fails. But since $(g_0,\Sigma)\in \cS_n$, it follows from item (1) of $D(g_0,\Sigma,n)$ as above that $(\bp^j,\btau^j/|\btau^j|)$ must converge subsequentially to a minimizer for $\cH_{(g_0,\Sigma)}$ in $\hat{X}_n$, and therefore $(\bp^j,\btau^j)$ converges to a least energy critical point $(\bp,\btau)$ for $\cH_{(g_0,\Sigma)}$. Moreover, since $D(g_0,\Sigma,n)$ holds, we know that $(\bp,\btau)$ satisfies items (1)-(3) of Theorem \ref{thm:doubling.precise}. Conditions (1) and (2) of Theorem \ref{thm:doubling.precise} are manifestly open, so these must also hold for $(\bp^j,\btau^j)$ for $j$ sufficiently large. Likewise, it follows from the smooth convergence $\cH_{(g_j,\Sigma)}\to \cH_{(g_0,\Sigma)}$ in a neighborhood of $(\bp,\btau)$ that $(\bp^j,\btau^j)$ must satisfy item (3) for $j$ sufficiently large as well, giving the desired contradiction.

\end{proof}

Combining this with Theorem \ref{thm:perturb}, we now see that $\bigcup_{n\geq k}\cS_n$ is indeed open and dense in $\cM_1^{\infty}$ for any $k\in \mathbb{N}$.

Next, for each $\epsilon>0$, consider the open subset of $\cM_1^{\infty}$ given by
\begin{eqnarray*}
    \cM_{1,\epsilon}^{\infty}:=\{(g,\Sigma)\in \cM_1^{\infty} &\mid \lambda_1(J_{\Sigma})+\epsilon<0<\lambda_2(J_{\Sigma})-\epsilon,\text{ and }\\
    &\area(\Sigma)+\|A_{\Sigma}\|_{C^0(\Sigma,g)}<\frac{1}{\epsilon}\}.
\end{eqnarray*}
Note that we can then write $\cM_1^{\infty}$ as the union
$$\cM_1^{\infty}=\bigcup_{\epsilon\in \mathbb{Q}_+}\cM_{1,\epsilon}^{\infty}$$
of $\cM_{1,\epsilon}^{\infty}$ over the positive rationals $\epsilon \in \mathbb{Q}_+$. The utility of partitioning $\cM_1^{\infty}$ in this way stems from the following finiteness result.

\begin{lem}\label{meps.finite}
For each $g\in \Gamma_{\infty}$, the cardinality 
$$m(g,\epsilon):=\#(\cM_{1,\epsilon}^{\infty}\cap \pi^{-1}(g))$$
is finite. 
\end{lem}
\begin{proof}
Fix $g\in \Gamma_{\infty}$, and suppose, for a contradiction, that $(g,\Sigma_j)\in \cM_{1,\epsilon}^{\infty}$, with $\Sigma_j\subset N$ an infinite sequence of distinct surfaces. Since there is a uniform bound
$$\area(\Sigma_j)+\|A_{\Sigma_j}\|_{C^0}<\frac{1}{\epsilon},$$
on area and second fundamental forms, it follows that, after passing to a subsequence, $\Sigma_j$ must converge smoothly to a limit minimal surface $\Sigma\subset (N,g)$ satisfying
$$\lambda_1(J_{\Sigma})+\epsilon\leq 0\leq \lambda_2(J_{\Sigma})-\epsilon.$$
On the other hand, as a smooth limit of distinct minimal surfaces in $(N,g)$, $\Sigma$ must be degenerate, contradicting the inequalities $\lambda_1(J)<0<\lambda_2(J)$. Thus, this situation is impossible, and it follows that $m(g,\epsilon)<\infty.$

\end{proof}

Now, for each $\epsilon>0$ and $k\in \mathbb{N}$, denote by $\cG_{k,\epsilon}\subset \Gamma_{\infty}$ the set of metrics
$$\cG_{k,\epsilon}:=\{g\in\Gamma_{\infty}\mid \pi^{-1}(g)\cap \cM_{1,\epsilon}^{\infty}\subset \bigcup_{n\geq k}\cS_n\}.$$
That is, $g\in \cG_{k,\epsilon}$ if, for every two-sided minimal surface $\Sigma$ in $(N,g)$ with $\lambda_1(J_{\Sigma})+\epsilon< 0<\lambda_2(J_{\Sigma})-\epsilon$ and $area(\Sigma)+\|A_{\Sigma}\|_{C^0}<\frac{1}{\epsilon}$, property $\Db(g,\Sigma,n_k)$ holds for some $n_k\geq k$.
We then prove the following.

\begin{prop}\label{geps.open.dense}
For each $\epsilon>0$ and $k\in \mathbb{N}$, $\cG_{k,2\epsilon}$ is open and dense in $\Gamma_{\infty}$.   
\end{prop}
\begin{proof}
Fix an arbitrary $g\in \Gamma_{\infty}$, and let $m=m(g,\epsilon)$ as in Lemma \ref{meps.finite}. We then enumerate by $\Sigma_1,\ldots,\Sigma_m$ the surfaces for which $(g,\Sigma_i)\in \cM_{1,\epsilon}^{\infty}$. Now, for each $1\leq i\leq m$, let $\cU(g,\Sigma_i)\subset \cM^{\infty}$ and $\cV(g,\Sigma_i)\subset \Gamma_{\infty}$ be as in Section \ref{sect:proofMainSetUp}, so that
$$\cV_{\epsilon}(g):=\bigcap_{i=1}^m\cV(g,\Sigma_i)$$
gives an open neighborhood of $g$ in $\Gamma_{\infty}$: See Figure \ref{fig:V_epsilon}. Setting
$$\cU_{i,k}:=\cU(g,\Sigma_i)\cap \bigcup_{n\geq k}\cS_n$$
and recalling that $\bigcup_{n\geq k}\cS_n$ is open and dense in $\cM_1^{\infty}$, it follows that $\cU_{i,k}$ is open and dense in $\cU(g,\Sigma_i)$, and therefore 
$$\cV_{\epsilon,k}(g):=\pi(\cU_{1,k})\cap \cdots \cap \pi(\cU_{m,k})$$
must be open and dense in $\cV_{\epsilon}(g)$. Recall that for every $h\in \cV(g,\Sigma_i)$, the set $\pi^{-1}(h)\cap \cU(g,\Sigma_i)$ has a unique element $(h,\Sigma_{i,h})$, and if $h\in \cV_{\epsilon,k}(g)$, it follows that $\Sigma_{i,h}\in \cU_{i,k}$. In particular, we see that
$$\pi^{-1}(\cV_{\epsilon,k}(g))\cap \cU(g,\Sigma_i)\subset \cU_{i,k}$$
for every $1\leq i\leq m$.

   \begin{figure}
        \centering
        \makebox[\textwidth][c]{\includegraphics[width=4.5in]{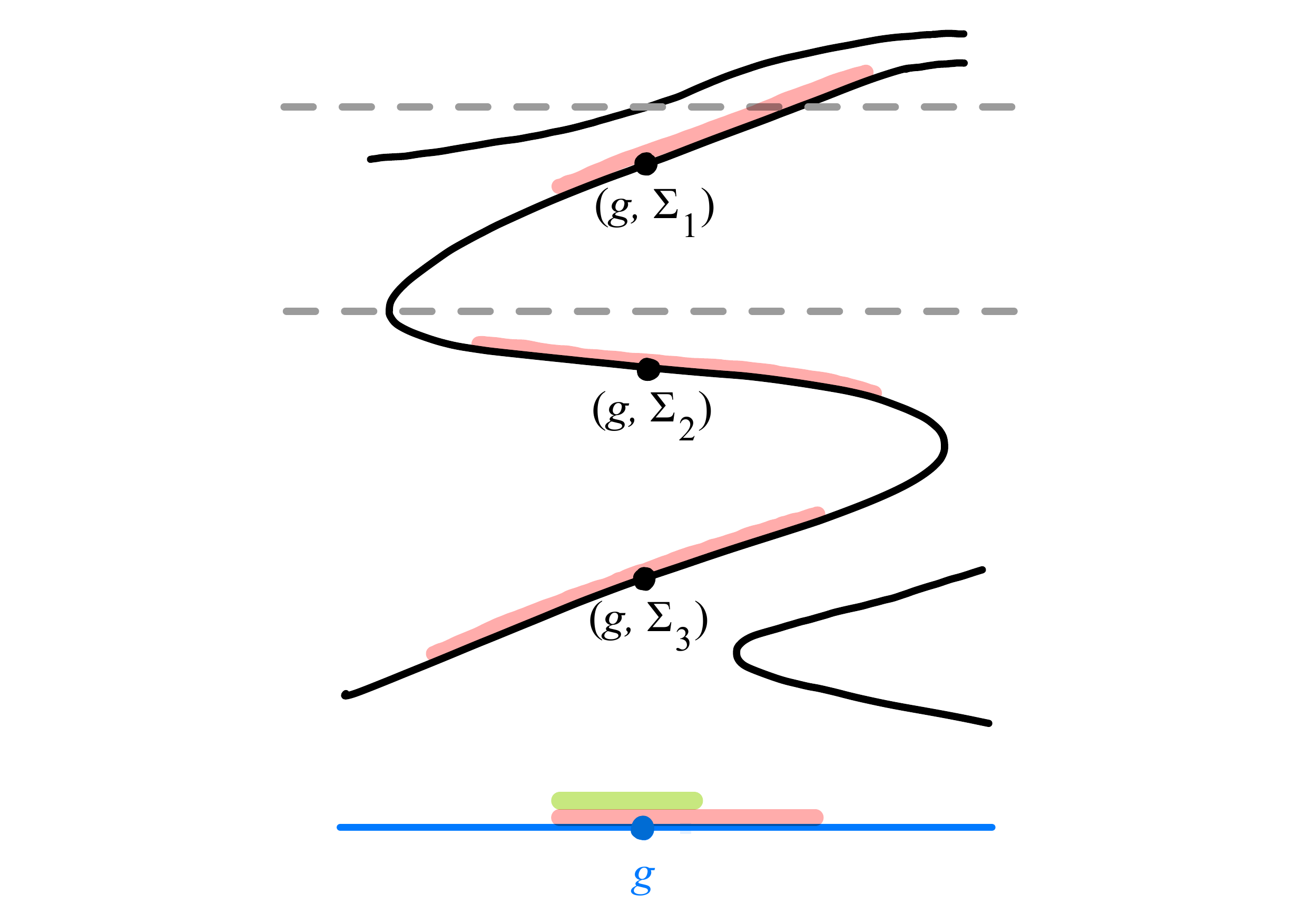}}
        \caption{The   black curves denote  $\cM^\infty_1$, in which the part below the upper dashed line is $\cM^{\infty}_{1,\epsilon}$, and the   part below the lower dashed line is $\cM^{\infty}_{1,2\epsilon}$. The blue line is $\Gamma_\infty$. 
        On $\cM^\infty_1$, the red parts denote $\cU(g,\Sigma_i)$. On $\Gamma_\infty$, the red part denotes $\cV_\epsilon(g)$, and the green part denotes $\cW(g)$.}
        \label{fig:V_epsilon}
    \end{figure}

Next, we argue that there is a smaller neighborhood $\cW(g)\subset \cV_{\epsilon}(g)$ of $g$ such that
\begin{equation}\label{w.nbd.def}
    \pi^{-1}(\cW(g))\cap \cM_{1,2\epsilon}^{\infty}\subset \bigcup_{i=1}^m\cU(g,\Sigma_i)
\end{equation}
(see Figure \ref{fig:V_epsilon}).
Indeed, this follows from the observation that, for any sequence $g_j\to g$ converging smoothly to $g$ and any $(g_j,S_j)\in \cM_{1,2\epsilon}^{\infty}$, the uniform bounds on $\area(S_j)$ and $|A_{S_j}|_{g_j}$ force $S_j$ to convergence smoothly, along a subsequence, to a minimal surface $S\subset (N,g)$ satisfying the non-strict inequalities
$$\lambda_1(J_S)+2\epsilon\leq 0\leq \lambda_2(J_S)-2\epsilon$$
and
$$\area(S)+\|A_S\|_{C^0(S,g)}\leq \frac{1}{2\epsilon}.$$
If no neighborhood $\cW(g)$ satisfying \eqref{w.nbd.def} existed, then we could find such a sequence $(g_j,S_j)$ for which the limit $S\notin \cU(g,\Sigma_i)$ for any $1\leq i\leq m$. On the other hand, any such limit $S$ clearly satisfies $(g,S)\in \cM_{1,\epsilon}^{\infty}$, so we must have $S=\Sigma_i$ for some $1\leq i\leq m.$ Thus, some neighborhood $\cW(g)$ satisfying \eqref{w.nbd.def} must exist.

We then see that $\cV_{\epsilon,k}(g)\cap \cW(g)$ is open and dense in $\cW(g)$, and satisfies
$$\pi^{-1}(\cV_{\epsilon,k}(g)\cap \cW(g))\cap \cM_{1,2\epsilon}^{\infty}\subset \bigcup_{i=1}^m\cU_{i,k}\subset \bigcup_{n\geq k}\cS_n.$$
Hence, $\cV_{\epsilon,k}(g)\cap \cW(g)\subset \cG_{k,2\epsilon}$, so in particular $\cG_{k,2\epsilon}\cap \cW(g)$ is open and dense in $\cW(g)$.

We have then shown that every metric $g\in \Gamma_{\infty}$ has a neighborhood $\cW(g)$ in which $\cG_{k,2\epsilon}$ is open and dense, from which the desired result follows.

\end{proof}

We can now complete the proof of Theorem \ref{thm:good.generic}.

\begin{proof}[Proof of Theorem \ref{thm:good.generic}]
First, it follows from Proposition \ref{geps.open.dense} that the set
$$\tilde{\cG}:=\bigcap_{k\in \mathbb{N}}\bigcap_{\epsilon\in \mathbb{Q}_+}\cG_{k,2\epsilon}$$
is a countable intersection of dense open subsets of $\Gamma_{\infty}$, hence is comeager in $\Gamma_{\infty}$. Unraveling the definition, we see that if $g\in \tilde{\cG}$, then for every $\epsilon>0$ and $k\in \mathbb{N}$,
$$\pi^{-1}(g)\cap \cM_{1,\epsilon}^{\infty}\subset \bigcup_{n\geq k}\cS_n,$$
so that
$$\pi^{-1}(g)\cap \cM_1^{\infty}\subset \bigcap_{k\in \mathbb{N}}\bigcup_{n\geq k}\cS_n.$$

Thus, each $g\in \tilde{\cG}$ has the property that for every two-sided, embedded, \emph{nondegenerate} minimal surface $\Sigma\subset N$ with Morse index one,
$$(g,\Sigma)\in \bigcap_{k\in \mathbb{N}}\bigcup_{n\geq k}\cS_n.$$
Finally, letting $\cG$ denote the intersection of $\tilde{\cG}$ with the comeager set of bumpy metrics, we see that $\cG$ is again comeager, and all embedded minimal surfaces in $(N,g)$ are nondegenerate for any $g\in \cG$, so that $(g,\Sigma)\in \bigcap_{k\in \mathbb{N}}\bigcup_{n\geq k}\cS_n$ for every two-sided, embedded minimal surface $\Sigma\subset (N,g)$ with Morse index one. Thus, $\cG$ has the desired properties.
\end{proof}

The remainder of the section is now devoted to the proof of Theorem \ref{thm:perturb}, carrying out the perturbation argument used  to establish the density of $\bigcup_{n\geq k}\cS_n$ in $\cM_1^{\infty}$.

\subsection{Effects of a conformal change on the variational problem}\label{sect:generalFactConformal}
\subsubsection{Interaction energies under a conformal change}

Fix a closed surface $\Sigma$ with a smooth metric $g_0$, and some positive function $V_0\in C^\infty(\Sigma)$. As in \S \ref{sect:KMtoEnergyFunctional},  we consider the associated Schrödinger operator $J_0:= d^*_{g_0}d-V_0$, Green's function $G_0$,  Robin's function $R_0$, and, for any $n$, the  functionals $\cH_0$ and $\hat\cH_0$ {\it associated to the pair $(g_0,V_0)$}.

To a given $u\in C^{\infty}(\Sigma)$, we then associate a new pair $(g_1,V_1)$ by setting $g_1:=e^{2u}g_0$ and $V_1:=e^{-2u}V_0$, with associated Schrödinger operator $J_1:= d^*_{g_1}d+V_1$. By the conformal covariance of the Laplacian in dimension two, we then see that
$$J_1\psi=e^{-2u}d_{g_0}^*d_{g_0}\psi-V_1\psi=e^{-2u}J_0\psi$$
for any $\psi\in C^{\infty}(\Sigma)$, and in particular,
\begin{equation}\label{j.covar}
   \int \phi J_0\psi dA_{g_0}=\int \phi e^{2u}J_1\psi dA_{g_0}=\int \phi J_1\psi dA_{g_1} 
\end{equation}
for any $\phi,\psi\in C^{\infty}(\Sigma)$. Defining the Green's function $G_1$, the Robin's function $R_1$, and interaction energies $\cH_1$ and $\hat \cH_1$ with respect to the pair $(g_1,V_1)$, the covariance \eqref{j.covar} leads to the following useful relations.

\begin{lem}\label{lem:confChangeH} In the notation of the preceding paragraphs, we have:
\begin{enumerate}
\item $G_0=G_1$.
\item For any $(\bp,\btau)\in X_n$,
\begin{align*}
\cH_1(\bp,\btau)&=\cH_0(\bp,\btau)+\frac 1{2\pi}\sum_{i=1}^n \tau^2_iu(p_i).  
\end{align*}
\item For any $(\bp,\bomega)\in\hat X_n$,
$$\hat\cH_1(\bp,\bomega)=\hat\cH_0(\bp,\bomega)\cdot e^{2\sum_{i=1}^n\omega_i^2u(p_i)}. $$
\end{enumerate}

\end{lem}
\begin{proof}
To show that $G_0=G_1$, note that \eqref{j.covar} holds whenever $\psi\in C^{\infty}(\Sigma)$ and $\phi\in L^1(\Sigma)$, and applying this with $\phi(y):=G_1(x,y)$ for an arbitrary $x\in \Sigma$ gives, by definition of the Green's function,
\begin{eqnarray*}
    \int G_1(x,y)J_0\psi(y) dA_{g_0}(y)&=&\int G_1(x,y) J_1\psi(y) dA_{g_1}(y)\\
    &=&\psi(x)\\
    &=&\int G_0(x,y) J_0\psi_0(y) dA_{g_0}(y).
\end{eqnarray*}
Since this holds for arbitrary $\psi$ and $J_0$ is invertible, it follows that $G_0=G_1$.

As a result, it follows directly from the definitions of $\cH_0$ and $\cH_1$ that
\begin{equation}\label{eq:H1H0R}
\cH_1(\bp,\btau)=\cH_0(\bp,\btau)+\sum_{i}\tau^2_i(R_1(p_i,p_i)-R_0(p_i,p_i))
\end{equation}
Moreover, note that
\begin{align*}
    R_1(x,y)-R_0(x,y)&=[G_1(x,y)+\frac{1}{2\pi}\log d_{g_1}(x,y)]-[G_0(x,y)+\frac{1}{2\pi}\log d_{g_0}(x,y)]\\
    &=\frac{1}{2\pi}\log \frac{d_{g_1}(x,y)}{d_{g_0}(x,y)}.
\end{align*}
Taking $y\to x$ and using the estimate
$$|\log\frac{d_{g_1}(x,y)}{d_{g_0}(x,y)}-u(x)-\frac 12 d_xu((\exp^{g_0}_x)^{-1}(y))|\leq C(g_0,u) d_{g_0}(x,y)^2$$
as in \cite[Lemma 3.13]{kapouleasMcGrath2023generalDoubling}, we obtain 
$$R_1(x,x)-R_0(x,x)=\frac{1}{2\pi}u(x).$$
Returning to \eqref{eq:H1H0R}, we arrive at
$$\cH_1(\bp,\btau)=\cH_0(\bp,\btau)+\frac 1{2\pi}\sum_{i=1}^n \tau^2_iu(p_i).$$

Finally, for any $(\bp,\bomega)\in\hat X_n$,  we can directly compute using  \eqref{hath.gibbs} and the above equation that
\begin{align*}\hat\cH_1(\bp,\bomega)&=\frac{1}{4\pi e}e^{4\pi\cH_1(\bp,\bomega) }\\&=\frac 1{4\pi e}e^{4\pi\cH_0(\bp,\bomega) +2\sum_i\omega^2_iu(p_i)}  \\&=\hat\cH_0(\bp,\btau)\cdot e^{2\sum_i\omega^2_i u(p_i) }.\end{align*}
\end{proof}


\subsubsection{Prescribing the Jacobi operator through an ambient conformal change}\label{sect:prescribed} 
Now, suppose we are given  a Riemannian 3-manifold $(N,g_0)$, in which  $\Sigma$ is an embedded minimal surface. Writing the Jacobi operator of $\Sigma$ as $d_{g_0}^*d-V_0$, the lemma below roughly says, given any $u\in C^\infty(\Sigma)$, the change $(g_0,V_0)\mapsto (e^{2u}g_0,e^{-2u}V_0)$ can be realized through a conformal change of the ambient metric $g_0$ on $N$.

\begin{lem}\label{lem:prescribedChange}
Let $(N,g_0)$ be a Riemannian $3$-manifold, and $\Sigma\subset N$ a two-sided embedded minimal surface. For any $u\in C^{\infty}(\Sigma)$, there exists a function  $\rho\in C^\infty(N)$ such that:
\begin{itemize}
\item $\rho|_{\Sigma}=u$.
\item $\Sigma$ stays minimal under the metric $e^{2\rho}g_0$ on $N$.
\item Denoting by $J_0$ the Jacobi operator for $\Sigma$ in $(N,g_0)$, the Jacobi operator $J_1$ for $\Sigma$ in $(N, e^{2\rho}g_0)$ is given by
$$J_1=e^{-2u}J_0.$$
\item  There exists some $\epsilon_0(g_0,\Sigma)>0$ such that for each $k=0,1,2,...$, $$\\\| \rho\|_{C^k(N,g_0)}<C(g_0,\Sigma,k)\epsilon_0^{-k}\|u\|_{C^{k+2}(\Sigma,g_0)}.$$ 
\end{itemize}
\end{lem}
\begin{proof} Denote by  $T_\epsilon(\Sigma)\subset N$ the closed $\epsilon$-neighborhood of $\Sigma$ in $N$. Fix a small $\epsilon_0>0$ such that the nearest-point projection $P:T_{2\epsilon_0}(\Sigma)\to\Sigma$ is smooth and well-defined, and let $d_\Sigma:T_{2\epsilon_0}(\Sigma)\to [0,\infty)$ denote the distance function to $\Sigma$. We first define $\rho$ on $T_{\epsilon_0}(\Sigma)$ by  
\begin{equation}\label{rho.def}
\rho(x):=u(P(x))-\frac{1}{4}[(\Delta_\Sigma u)(P(x))+|(\nabla_\Sigma u)(P(x))|^2)]d_\Sigma^2(x),
\end{equation}
where $\Delta_{\Sigma}=-d_{g_0|_{\Sigma}}^*d_{g_0|_{\Sigma}}$ is the negative-spectrum Laplacian on $\Sigma$. Denoting by $\nabla$ the gradient and $\Delta_N$ the negative-spectrum Laplacian with respect to the original metric $g_0$ on $N$, we then see that:
\begin{itemize}
        \item $\rho|_\Sigma=u$;
        \item On $\Sigma$, $(\nabla\rho)^\perp=0$;
        \item On $\Sigma$,  $\Hess(\rho)(\nu,\nu)=-\frac{1}{2}(\Delta_\Sigma u+|\nabla_{\Sigma}u|^2)$, where $\nu$ is a unit normal to $\Sigma$. 
    \end{itemize}

In general, the mean curvature vector of $\Sigma$ under the new metric $g_1:=e^{2\rho}g_0$ is given by 
$$e^{-2\rho}({\bf H}-2(\nabla\rho)^\perp),$$
where $\bf H$ is the mean curvature vector with respect to $g_0$ \cite[\S 1.163]{besse2007einstein}. Hence, by the minimality of $\Sigma$  in $g_0$ and the vanishing $(\nabla\rho)^\perp=0$ on $\Sigma$, we see that $\Sigma$ remains minimal in $T_{\epsilon_0}(\Sigma)$ with respect to $g_1$.

Let $J_0$ (resp. $J_1$) denote the Jacobi operator of $\Sigma$ under $g_0$ (resp. $g_1$).  For any $\Sigma$ and $\rho$ in general, a standard computation (see Appendix \ref{appendix:conf}) gives the following relation:
$$J_1=e^{-2\rho}(J_0+2\bH\cdot\nabla\rho+\Hess(\rho)(\nu,\nu)+\Delta_N\rho+|\nabla\rho|^2-3|(\nabla\rho)^\perp|^2).$$
Note again that $\rho|_{\Sigma}=u$, $(\nabla\rho)^\perp=0$, $\bH=0$, and by minimality of $\Sigma$, $\Delta_N\rho=\Delta_{\Sigma}u+\Hess(\rho)(\nu,\nu)$ along $\Sigma$, so that
\begin{eqnarray*}
    |\nabla \rho|^2+\Hess(\rho)(\nu,\nu)+\Delta_N\rho&=&|\nabla_{\Sigma}u|^2+\Delta_{\Sigma}u+2\Hess(\rho)(\nu,\nu)\\
    &=&|\nabla_{\Sigma}u|^2+\Delta_{\Sigma}u-(\Delta_{\Sigma}u+|\nabla_{\Sigma}u|^2)\\
    &=&0.
\end{eqnarray*}
Using this in the preceding computation, we see that $J_1=e^{-2u} J_0$, as desired.

Finally, we extend $\rho$ outside the tubular neighborhood $T_{\epsilon_0}(\Sigma)$ via a simple cutoff. More precisely, we can fix a cutoff function $\chi$ such that $\chi=0$ on $N\backslash T_{2\epsilon_0}(\Sigma)$, $\chi=1$ in $T_{\epsilon_0}$, and $|D^k\chi|\leq C \epsilon_0^{-k}$, then let $\rho$ be given by multiplying the right-hand side of \eqref{rho.def} by $\chi$. It is then straightforward to check that $\rho$ satisfies all the desired properties.
\end{proof}

\subsection{Proof of Theorem \ref{thm:perturb}}\label{proof_thm:perturb}

On the given closed $3$-manifold $N$, fix $(g,\Sigma)\in \cM_1^{\infty}$, so that $\Sigma$ is a non-degenerate, two-sided, embedded minimal surface in $(N,g)$ of Morse index one. Denote by $J_0=d^*d-V_0$ the Jacobi operator of $\Sigma$ in $(N,g)$.

To prove Theorem \ref{thm:perturb}, for each $n\geq n_0(g,\Sigma)$ sufficiently large, we need to find a conformal metric $g_n=e^{2\rho_n}g$ such that $(g_n,\Sigma)\in \cS_n$, with $\rho_n\to 0$ in $C^{\infty}(\Sigma)$ as $n\to\infty$.

\subsubsection{Defining $\rho_n$}
Since the Jacobi operator $J_0$ is an invertible operator with $\lambda_1(J_0)<0$, we can apply Theorem \ref{thm:existGoodMinimizers} and Proposition \ref{prop:piTauiAlmostMinimum} to the triple $(\Sigma,g|_{\Sigma},V_0)$ to obtain constants $\overline n(g,\Sigma), \overline C(g,\Sigma),\overline\Lambda(g,\Sigma)>0$ such that for every $n\geq  \overline n$, we have:
\begin{itemize}
\item 
$ n\cI - \overline C\log n<\inf_{\hat X_n}  \cH_0< n\cI+\overline C n^{3/4},$ where $\cI=\cI(g_0,V_0)$ is the infimum of $\cE(\mu):= \int G_0d\mu d\mu$ over probability measures on $\Sigma$.
\item Every minimizing sequence for $\cH_0$ on $\hat X_n$ is precompact in $\hat X_n$.
\item  Every minimizer $(\bp,\bomega)$ for $\cH_0$ on $\hat X_n$ satisfies $\frac 1{ \overline\Lambda\sqrt n} <\omega_i<\frac{\overline\Lambda}{\sqrt n}$ for each $i$,  and $d(p_i,p_j)> n^{-3}$ for each $i\ne j$.
\end{itemize}

Now, for each  $n\geq \overline n$, we fix some critical point $(\bp^n,\btau^n)$ of $\cH_0$ in $X_n$ such that $(\bp^n, {\btau^n}/{|\btau^n|})$ is a minimizer for $\cH_0$ on $\hat X_n$. Moreover, fix
$$a_0(g,\Sigma):=-10^{-6}\cI=10^{-6}|\cI|.$$

Next, fix a nondecreasing function $\chi\in C^{\infty}(\mathbb{R})$ satisfying $\chi(t)=t^2$ for $t<\frac{1}{2}$ and $\chi(t)=1$ for $t\geq 1$. Using the bound $\min_{i\ne j}d(p^n_i,p^n_j)>n^{-3}$,  we can then define $u_n\in C^{\infty}(\Sigma)$ by setting 
$$u_n(x):=e^{-a_0 n}\chi(2n^3d_{p_i^n})\text{ for }x\in B_{1/2n^3}(p^n_i),$$
$$u_n\equiv e^{-a_0 n}\text{ on }\Sigma\setminus \bigcup_{i=1}^nB_{1/2n^3}(p^n_i).$$
It is then straightforward to check that: 
\begin{itemize}
\item  \begin{equation}\label{eq:uC0Bound}
    0\leq u_n\leq 2e^{-a_0 n}.
\end{equation}
\item $u_n$ has exactly $n$  minimum points on $\Sigma$, located at $p^n_1,...,p^n_n$, with $u_n(p^n_i)=0$ for each $i$.
\item Each $p^n_i$ is a non-degenerate minimum point of $u_n$, with 
\begin{equation}\label{eq:D2u}
    D^2u_n(p^n_i)(v,v)= 8e^{-a_0 n}n^6h(v,v)
\end{equation} for every $v\in T_{p^n_i}\Sigma$.
\item For each $k=0,1,2,\ldots$, there is a constant $C_2(g,\Sigma,k)>0$, \begin{equation}\label{eq:uBound}
  \| u_n\|_{C^{k}(\Sigma,g)}<C_2e^{-a_0 n}n^{3k}.
\end{equation}
\end{itemize}

Applying Lemma \ref{lem:prescribedChange} to $(N,g)$, $\Sigma$, and $u_n\in C^\infty(\Sigma)$ for each $n$, we obtain a function $\rho_n\in C^\infty(N)$ such that:
\begin{itemize}
\item $\rho_n|_{\Sigma}=u_n$.
\item $\Sigma$ remains minimal in the metric $g_n=e^{2\rho_n}g$ on $N$.
\item The Jacobi operator  for $\Sigma$ in $(N, g_n)$ is given by $e^{-2u_n}J_0.$
\item There exists some $\epsilon_0(g,\Sigma)>0$ such that for each $k=0,1,2,...$, $$\max_{x\in N }|\nabla^k\rho_n|(x)<C(g,\Sigma,k)\epsilon_0^{-k}\|u\|_{C^{k+2}(\Sigma,g)}.$$
\end{itemize}
In particular, combining the last item with \eqref{eq:uBound}, we see that for each $k=0,1,2,...$, 
$$\|\rho_n\|_{C^k(N,g)}<C(g,\Sigma,k)\epsilon_0^{-k} e^{-a_0 n}n^{3(k+2)}.$$
In particular, since $e^{-a_0n}n^{3(k+2)}\to 0$ as $n\to\infty$, it follows that the metric $g_n:=e^{2\rho_n}g$ converges to $g$ in $C^k$ as $n\to \infty$; thus, $g_n\to g$ in the space $\Gamma_{\infty}$ of smooth metrics on $N$.

Now, fixing some $n\geq \overline n$, we denote by $\cH_1$ the functional $ \cH_{(g_n,\Sigma)}:X_n\to \mathbb{R}$, and by $\hat\cH_1$ the functional $\hat \cH_{(g_n,\Sigma)}:\hat X_n\to (0,\infty)$. In other words,  these  interaction energies for $\Sigma$ are defined with respect to the metric  $e^{2u_n}g|_{\Sigma}$ and the Schr\"odinger operator $J=d^*_{e^{2u_n}g|_{\Sigma}}d-e^{-2u_n}V_0=e^{-2u_n}J_0$.

Recall that, to prove Theorem \ref{thm:perturb}, we just need to show that for every $n$ sufficiently large, $(g_n,\Sigma)\in \cS_n$, which amounts to checking that Property $D(g_n,\Sigma,n)$ holds. 
 
We first characterize the minimizers of $\cH_1$ on $\hat X_n$ in terms of the chosen critical points $(\bp^n,\btau^n)$ of $\cH_0$.

\begin{lem}\label{lem:newMinimizer}
Fix $n\geq \overline n$, and let  $(\bq,\bomega)\in \hat X_n$. Then $(\bq,\bomega)$ is a minimizer for $\cH_1$ on $\hat X_n$ if and only if $\bq=\bp^n$ and $(\bq,\bomega)$ is a minimizer for $\cH_0$ on $\hat X_n$, with $\cH_1(\bq,\bomega)=\cH_0(\bq,\bomega).$ In particular, if $(\bp,\btau)\in X_n$ is the critical point for $\cH_1$ such that $(\bp,\btau/|\btau|)=(\bp,\bomega)$, then $(\bp,\btau)$ is also the critical point for $\cH_0$ associated to $(\bp,\bomega)$. 
\end{lem}
\begin{proof} By definition, $u_n(x)\geq 0$ for all $x\in \Sigma$, with equality only when $x\in\{p_1^n,\ldots,p_n^n\}$.
Thus, by Lemma \ref{lem:confChangeH}, for any $(\bq,\bomega)\in \hat X_n$,  we have
$$\cH_1(\bq,\bomega)=\cH_0(\bq,\bomega)+\frac 1{2\pi}\sum_i \omega^2_iu_n(p_i)\geq \inf_{\hat X_n}\cH_0,$$
with equality if and only if $\bq=\bp^n$ and $\cH_1(\bq,\bomega)=\cH_0(\bq,\bomega)=\inf_{\hat X_n}\cH_0$. 

Since the $\cH_i$-critical point $(\bp,\btau)\in X_n$ associated to an $\hat\cH_i$-minimizer $(\bp,\bomega)$ is determined only by $\btau=e^{2\pi \cH_i(\bp,\bomega)-1/2}$, it follows from the equality $\cH_1(\bp,\bomega)=\cH_0(\bp,\bomega)$ that the associated critical points $(\bp,\btau)$ for $\cH_0$ and $\cH_1$ also coincide.
\end{proof}

Since any minimizing sequence for $\cH_0$ is precompact in $\hat{X}_n$, it also follows from the proof of Lemma \ref{lem:newMinimizer} that any minimizing sequence for $\cH_1$ in $\hat{X}_n$ is also precompact, and in particular converges to a minimizer of the form $(\bp^n,\bomega)$. Thus, to check   $D(g_n,\Sigma,n)$, it remains to prove the following: For each $n\geq  n_0$, with $n_0(g,\Sigma)$ to be determined, for every critical point of $\cH_1$ such that $(\bp,\btau/|\btau|)$ is a minimizer of $\cH_1$ on $\hat X_n$, the point $(\bp,\btau)$ satisfies items (\ref{tau.comparable}) - (\ref{quant.nondeg}) of Theorem \ref{thm:doubling.precise} with respect to the metric $g_n|_{\Sigma}$ and the Jacobi operator $e^{-2u_n}J_0$.

From now on, let  $(\bp,\btau)$ be {\it any} critical point of $\cH_1$ on $X_n$ such that $(\bp,\btau/|\btau|)$ is a minimizer of $\cH_1$ on $\hat X_n$. For convenience, we will denote $\hat \btau:=\btau/|\btau|$.  Note, by Lemma \ref{lem:newMinimizer}, we know $\bp=\bp^n$, and $(\bp,\hat \btau)$ is also a minimizer of the original energy $\cH_0$ on $\hat X_n$, with $(\bp,\btau)$ the associated critical point for $\cH_0$. We will check that $(\bp,\btau)$ satisfies items (\ref{tau.comparable}) - (\ref{quant.nondeg}) of Theorem \ref{thm:doubling.precise}, for every sufficiently large $n$.

\subsubsection{Item   (\ref{tau.comparable}) of  Theorem \ref{thm:doubling.precise}} Since $(\bp,\hat \btau)$ is a minimizer for $\cH_0$ on $\hat X_n$  and $n\geq \overline{n}$, we know already that
\begin{equation}\label{eq:hatTauComparable}
\frac{1}{\overline{\Lambda}\sqrt n}<\hat\tau_i<\frac{\overline{\Lambda}}{\sqrt n}.  
\end{equation} 
In particular, noting 
$\tau_i/\overline{\btau}=\hat\tau_i/\overline{\hat\btau}$, we see that
 $1/n<\tau_i/\overline{\btau}<n$ when $n>\overline{\Lambda}(g,\Sigma)^2$.

Now, since $\btau$ is given by
$$\btau=e^{ {2\pi \cH_0(\bp,\hat \btau)}  -1/2}\hat \btau,$$
and since $\cH_0(\bp,\hat\btau)=n\cI+O(n^{3/4})$ and $\cI<0$, we then have 
\begin{equation}\label{eq:btau}
   \btau=e^{2\pi n\cI+O(n^{3/4})}\hat\btau,
\end{equation}
where here and in the remainder of \S \ref{proof_thm:perturb}, $O(\epsilon_n)$ refers to any quantity bounded in absolute value by $C\epsilon_n$ for a constant $C=C(g,\Sigma)$ depending only on the initial pair $(g,\Sigma)\in \cM_1^{\infty}$. Together with (\ref{eq:hatTauComparable}), the preceding estimate implies that
\begin{equation}\label{eq:barTau}
    \bar \btau=e^{2\pi n\cI+O(n^{3/4})}.
\end{equation}
In particular, we can guarantee that $\bar\btau\in (0,n^{-100,000})$ for all sufficiently large $n$. 

In conclusion, we can ensure that condition (\ref{tau.comparable}) of Theorem \ref{thm:doubling.precise} holds for $n\geq n_0(g,\Sigma)$ sufficiently large.

\subsubsection{Item (\ref{del.low}) of Theorem \ref{thm:doubling.precise} } This is immediate, since $$\delta(\bp)=\min_{i\ne j}d(p_i,p_j)=\min_{i\ne j}d(p^n_i,p^n_j)>n^{-3}.$$ Thus, by \eqref{eq:barTau}, we certainly have $\delta(\bp)>\bar\btau^{1/100,000}$ when $n$ is large enough, so Theorem \ref{thm:doubling.precise} (\ref{del.low}) also holds.

\subsubsection{Item (\ref{quant.nondeg}) of Theorem \ref{thm:doubling.precise} }\label{sect:item_quant.nondeg}

Coming now to the heart of the proof, we need to prove that for every eigenvalue $\lambda$ of 
the self-adjoint linear map 
$$S_{\btau}D^2\cH_1(\bp,\btau)S_\btau: \cP\to \cP$$
on the space
$$\cP:=T_{(\bp,\btau)}X_n=(\bigoplus_{i=1}^n T_{p_i}\Sigma)\oplus \R^n$$
we have  $|\lambda|>\bar{\btau}^{1/1000}$. 

As a first step, note that since
$$\cH_1(\bq,\bsigma)=\cH_0(\bq,\bsigma)+\frac{1}{2\pi}\sum_{i=1}^n\sigma_i^2u_n(q_i),$$
and we have $u_n(p_i^n)=du_n(p_i^n)=0$ for every $1\leq i\leq n$, a straightforward computation gives
\begin{eqnarray*}
    D^2\cH_1(\bp,\btau)(\bv,\bs)&=& D^2\cH_0(\bp,\btau)(\bv,\bs)+\frac{1}{2\pi}\sum_{i=1}^n\tau_i^2D^2 u_n(p_i^n)(v_i,\cdot)\\
    &=&D^2\cH_0(\bp,\btau)(\bv,\bs)+\frac{1}{2\pi}8e^{-a_0n}n^6\sum_{i=1}^n\tau_i^2v_i
\end{eqnarray*}
for any $(\bv,\bs)\in \cP$, where in the last line we use equation \eqref{eq:D2u}. In particular, taking the composition $S_{\btau}D^2\cH_1(\bp,\btau)S_{\btau}$, we see that
\begin{equation}\label{hess.decomp}
S_{\btau}D^2\cH_1(\bp,\btau)S_{\btau}(\bv,\bs)=S_{\btau}D^2\cH_0(\bp,\btau)S_{\btau}(\bv,\bs)+\frac{4n^6e^{-a_0 n}}{\pi}(\bv,0).
\end{equation}

Now, as in Section \ref{sect:wt.nondeg}, we consider the orthogonal decomposition
$$\cP=\langle (0,\btau)\rangle\oplus\langle (0,\btau)\rangle^\perp$$
of $\cP$ into the 1-dimensional span $\langle (0,\btau)\rangle$, and its orthogonal complement $\langle (0,\btau)\rangle^\perp=T_{(\bp,\hat \btau) }\hat X_n$. 

By Lemma \ref{lem:spanTau}, we know already that $(0,\btau)$ is an eigenvector for $S_{\btau}D^2\cH_1(\bp,\btau)S_{\btau}$ with eigenvalue $-\frac{1}{\pi}$, so all that remains is to obtain suitable lower bounds for the spectrum of $S_{\btau}D^2\cH_1(\bp,\btau)S_{\btau}$ on $\langle (0,\btau)\rangle^{\perp}.$

To this end, fix an arbitrary $(\bv,\bs)\in \langle (0,\btau)\rangle^{\perp}$, and consider two possibilities.

First, suppose that $|\bv|<n^{-5}|\bs|$. Then applying Proposition \ref{prop:wt.nondeg} at $(\bp,\btau)$ with respect to the original Jacobi operator $J_0$, it follows that for $n\geq n_0(g,\Sigma)$ sufficiently large, 
\begin{eqnarray*}
    \langle S_{\btau}D^2\cH_0(\bp,\btau)S_{\btau}(\bv,\bs),(\bv,\bs)\rangle &\geq & c_0 n|\bs|^2-n^5|\bv||\bs|\\
    &\geq & (c_0 n-1)|\bs|^2\\
    &\geq &\frac{c_0 n}{2}(|\bv|^2+|\bs|^2).
\end{eqnarray*}
In the opposite case, where $|\bv|\geq n^{-5}|\bs|$, note that Proposition \ref{prop:wt.nondeg} still gives the nonnegativity
$$\langle S_{\btau}D^2\cH_0(\bp,\btau)S_{\btau}(\bv,\bs),(\bv,\bs)\rangle \geq 0,$$
so that, by \eqref{hess.decomp}, we have
\begin{eqnarray*}
    \langle S_{\btau}D^2\cH_1(\bp,\btau)S_{\btau}(\bv,\bs),(\bv,\bs)\rangle &\geq & \frac{4 n^6 e^{-a_0 n}}{\pi}|\bv|^2\\
    &\geq &n^{-4}e^{-a_0 n}(|\bv|^2+|\bs|^2).
\end{eqnarray*}
Thus, we see that
$$\langle S_{\btau}D^2\cH_1(\bp,\btau)S_{\btau}(\bv,\bs),(\bv,\bs)\rangle\geq \min\{n^{-4}e^{-a_0 n},\frac{c_0 n}{2}\}(|\bv|^2+|\bs|^2)$$
for all $(\bv,\bs)\in \langle (0,\btau)\rangle^{\perp}$. 

Combining the preceding observations, we now see that every eigenvalue $\lambda\in Spec(S_{\btau}D^2\cH_1(\bp,\btau)S_{\btau})$ must satisfy
$$|\lambda|\geq \min\{\frac{1}{\pi},n^{-4}e^{-a_0 n},\frac{c_0 n}{2}\}=n^{-4}e^{-a_0 n}$$
for $n\geq n_0(g,\Sigma)$ sufficiently large. Finally, recall from \eqref{eq:barTau} that
$$\bar{\btau}=e^{2\pi n\cI+O(n^{3/4})},$$
so that
$$\bar{\btau}^{1/1000}=e^{2\pi n 10^{-3}\cI+O(n^{3/4})},$$
while we chose $a_0(g,\Sigma)=-10^{-6}\cI,$ so that clearly
$$n^{-4}e^{-a_0 n}=n^{-4}e^{10^{-6}\cI}>e^{2\pi n 10^{-3}\cI+O(n^{3/4})}$$
for $n$ sufficiently large. 

Thus, we have confirmed that $(g_n,\Sigma)$ satisfies condition (\ref{quant.nondeg}) of Theorem \ref{thm:doubling.precise} for all sufficiently large $n$, completing the proof of Theorem \ref{thm:perturb}.

\appendix
\section{Basic properties of Green's and Robin functions} \label{green.app}

Here we briefly recall how to prove the key properties of the Green's functions $G(x,y)$ and Robin functions $R(x,y)$ used throughout the paper, proving Lemma \ref{green.lem} in particular. Everything that follows should be well-known to experts, but we have opted to give a brief, self-contained treatment for the reader's convenience.

Let $J=d^*d-V$ be a Schr\"odinger operator on $(\Sigma,g)$ with $V\in C^{\infty}(\Sigma)$ such that $\ker(J)=0$. Note that $J$ is self-adjoint and Fredholm, and by standard elliptic theory, $J$ defines an invertible map $W^{k+2,p}(\Sigma)\to W^{k,p}(\Sigma)$ with bounded inverse for any $k\in \mathbb{Z}$ and $p\in (1,\infty)$. 

Fixing any $p\in \Sigma$, let $d_p\in Lip(\Sigma)$ be the distance function $d_p(x)=\dist_g(p,x)$, and recall that standard computations (cf. \cite[Section 4.2]{Aubin}) give an estimate of the form 
$$\|d^*d(\frac{1}{2\pi}\log d_p)+\delta_p\|_{C^0(B_{r_0}(p))}\leq C(g),$$
where $r_0=\frac{1}{2}\InjRad(\Sigma,g)$, while the fact that $d_p$ is Lipschitz gives
$$d^*d \log d_p\in W^{-1,\infty}(\Sigma).$$
Since the potential $V$ is smooth, it follows that
$$J(\frac{1}{2\pi}\log d_p)+\delta_p\in W^{-1,\infty}(\Sigma)\subset \bigcap_{q\in [1,\infty)}W^{-1,q}(\Sigma)$$
as well. As a consequence, we can find a unique $\psi_p\in \bigcap_{1\leq q<\infty}W^{1,q}(\Sigma)$ solving
$$J\psi_p=J(\frac{1}{2\pi}\log d_p)+\delta_p,$$
so that $\phi_p:=-\frac{1}{2\pi}\log d_p+\psi_p\in L^q(\Sigma)$ solves 
$$J\phi_p=\delta_p.$$
We then claim that the Green's function $G(x,y):=\phi_x(y)$ and Robin function $R(x,y)=\psi_x(y)$ satisfy the conclusions of Lemma \ref{green.lem}. 

Indeed, since $J\phi_p=\delta_p$, the equation
$$u(x)=\int_{\Sigma}G(x,y)(Ju)(y)dA_g(y)=\int_{\Sigma}\phi_x Ju$$
for any $u\in C^{\infty}(\Sigma)$ follows immediately, and by self-adjointness of $J$,
$$\int_{\Sigma\times \Sigma} G(x,y)u(x)v(y)=\langle J^{-1}u,v\rangle_{L^2}=\langle u, J^{-1}v\rangle_{L^2}=\int_{\Sigma\times \Sigma}G(x,y)v(x)u(y)$$
for any $u,v\in C^{\infty}(\Sigma)$, giving the symmetry $G(x,y)=G(y,x)$. Likewise, the smoothness $G\in C^{\infty}_{loc}(\Sigma\times \Sigma\setminus diag(\Sigma))$ away from the diagonal $diag(\Sigma)$ follows easily from local elliptic regularity for $J$ and the fact that $J\phi_p=0$ on $\Sigma\setminus\{p\}$. 

Since $R(x,y)=G(x,y)+\frac{1}{2\pi}\log d(x,y),$ it follows immediately that $R\in Lip_{loc}(\Sigma\times \Sigma\setminus diag(\Sigma))$. Moreover, on $B_{r_0}(p)$, we see that
$$|J\psi_p|=|d^*d(\frac{1}{2\pi}\log d_p)+\delta_p-\frac{V}{2\pi}\log d_p)|\leq C(g)(1+|\log d_p|),$$
and since the right-hand side belongs to $L^q$ for any $q\in [1,\infty)$, it follows from local elliptic regularity that
$$\|\psi_p\|_{W^{2,q}(B_{r_0}(p))}\leq C_q(g,V).$$
In view of the embedding $W^{2,q}(B_{r_0}(p))\subset C^{1,1-2/q}(B_{r_0}(p))$ for $q>2$, the weak convergence $\psi_x\to \psi_p$ as $x\to p$, and the symmetry $R(x,y)=R(y,x)$, it is then straightforward to deduce that $R$ is $C^1$ on the neighborhood 
$$\{(x,y)\in \Sigma\times \Sigma\mid d(x,y)<r_0/2\}$$
of the diagonal.

Since $G\in C_{loc}^{\infty}(\Sigma\times \Sigma\setminus diag(\Sigma))$ and $R(x,y)$ is $C^1$ on a neighborhood of the diagonal, it follows immediately that $G=-\frac{1}{2\pi}\log (d(x,y))+R(x,y)$ satisfies an estimate of the form
$$|\nabla G|(x,y)\leq \frac{C(g,V)}{d(x,y)}.$$
Then, for any ball $B_r(x)\subset \Sigma\setminus \{p\}$, since $J\phi_p=0$, this bound together with standard elliptic estimates leads to bounds of the form
$$|D^k\phi_p|(x)\leq C_k(g,V)r^{-k},$$
which together with the symmetry of the Green's function implies that
$$|D^kG(x,y)|\leq \frac{C_k(g,V)}{d(x,y)^k}.$$

This completes the proof of Lemma \ref{green.lem}. Next, we explain how to establish the smoothness of the \emph{diagonal} Robin function $R_D(x):=R(x,x).$

\begin{lem}\label{robin.reg}
    For any Schr\"odinger operator $J=d_g^*d-V$ with smooth potential $V$ on the closed surface $(\Sigma,g)$ as above, $R_D\in C^{\infty}(\Sigma)$.
\end{lem}
\begin{proof}
Fix an arbitrary point $p_0\in \Sigma$, and recall that we can find local isothermal coordinates $\Phi=(x_1,x_2):U\to B_{2\epsilon}^2(0)\subset\mathbb{R}^2$ on a neighborhood $U$ of $p_0$, with respect to which $\Phi(p_0)=0$ and the metric $g$ takes the form 
$$g=e^{2u}(dx_1^2+dx_2^2)=e^{2u}g_0$$
for some $u\in C^{\infty}(U)$. With respect to these coordinates, it follows from the computations of Section \ref{sect:generalFactConformal} that the Euclidean Green's functions
$$\sigma_x(y):=-\frac{1}{2\pi}\log |x-y|$$
solve $d_g^*d_g\psi_x=\delta_x$ on $U$, and we can decompose the Green's function $\phi_x(y)=G(x,y)$ for $J$ as 
$$\phi_x=\sigma_x+\xi_x,$$
where 
$$J\xi_x=J(\phi_x-\sigma_x)=V\sigma_x.$$
Note that elliptic regularity gives $\xi_x\in W^{2,q}(B_{2\epsilon}(0))$ for every $q\in [1,\infty)$.

Now, for $z\in B_{\epsilon}(0)$, define $u_z: B_{\epsilon}(0)\to \mathbb{R}$ by
$$v_z(y):=\xi_z(y+z)-\xi_0(y),$$
and observe that, for the Euclidean Laplacian $d_{g_0}^*d$, we have
$$d_{g_0}^*dv_z=(d_{g_0}^*d\xi_z)(y+z)-(d_{g_0}^*d\xi_0)(y);$$
using that $d_g^*d=e^{-2u}d_{g_0}^*d$ and $J\xi_x=V\sigma_x$ then gives
\begin{eqnarray*}
    d_g^*d v_z(y)&=&e^{-2u(y)}[(d_{g_0}^*d\xi_z)(y+z)-(d_{g_0}^*d\xi_0)(y)]\\
    &=&e^{2(u(y+z)-u(y))}(d_g^*d_g\xi_z)(y+z)-d_g^*d_g\xi_0(y)\\
    &=&e^{2(u(y+z)-u(y))}(V(y+z)\sigma_z(y+z)+V(y+z)\xi_z(y+z))\\
    &&-V(y)\sigma_0(y)-V(y)\xi_0(y).
\end{eqnarray*}
We will show next that $v_z(y)$ is infinitely differentiable with respect to the $z$ variable. By the translation-invariance of the Euclidean Green's function, we see that $\sigma_z(y+z)=\sigma_0(y)$, so that the preceding computation gives
\begin{eqnarray*}
    Jv_z(y)&=&[e^{2[u(y+z)-u(y)]}V(y+z)-V(y)]\sigma_0(y)+V(y+z)\xi_z(y+z)-V(y)\xi_0(y)\\
    &&-V(y)v_z(y)\\
    &=&[e^{2[u(y+z)-u(y)]}V(y+z)-V(y)]\sigma_0(y)\\
    &&+[V(y+z)-V(y)](v_z(y)+\xi_0(y)].
\end{eqnarray*}
It is then easy to see that $v_z\in W^{2,q}(B_{\epsilon}(0))$ for any $q\in [1,\infty)$, and applying the preceding computation to the difference quotients $\frac{v_{z+te_i}-v_z}{t}$ and taking $t\to 0$, we observe that the derivatives $\partial_{z_i}v_z$ of $v_z$ in the $z$ parameter solve
\begin{eqnarray*}
    J(\partial_{z_i}v_z)(y)&=&e^{2(u(y+z)-u(y))}(2\partial_iu(y+z)V(y+z)+\partial_iV(y+z))\sigma_0(y)\\
    &&+\partial_iV(y+z)(v_z(y)+\xi_0(y))+(V(y+z)-V(y))\partial_{z_i}v_z,
\end{eqnarray*}
and by the smoothness of $u$ and $V$, we again find $W^{2,q}$ estimates of the form
$$\|\partial_{z_i}v_z\|_{W^{2,q}(B_{\epsilon}(0)}\leq C_q(g,V).$$
Continuing in this way by induction, we deduce that the partial derivatives $D^k_zv_z$ of all orders in $z$ belong to $\bigcap_{q\in [1,\infty)}W^{2,q}(B_{\epsilon}(0))\subset C^1(B_{\epsilon}(0))$, with $C^1(B_{\epsilon}(0))$ estimates depending only on $g$, $V$, and $k$. Finally, for any small $0\neq y\in B_{\epsilon}(0)$, observe that
\begin{eqnarray*}
    R(y+z,z)-R(y,0)&=&G(y+z,z)-G(y,0)+\frac{1}{2\pi}\log(d(y+z,z)/d(y,0))\\
    &=&\xi_z(y+z)-\xi_0(y)+\frac{1}{2\pi}\log(d(y+z,z)/d(y,0))\\
    &=&v_z(y)+\frac{1}{2\pi}\log(d(y+z,z)/d(y,0)),
\end{eqnarray*}
and taking the limit $y\to 0$ gives
$$R(z,z)-R(0,0)=v_z(0)+\frac{1}{2\pi}\log(e^{u(z)}/e^{u(0)})=v_z(0)+\frac{1}{2\pi}(u(z)-u(0)).$$

The estimates we obtained for $v_z$ above imply smoothness of the function $z\mapsto v_z(0)$, so it follows from this formula that the diagonal Robin function $R_D(x)=R(x,x)$ must be smooth on a neighborhood of $p_0$. Since $p_0$ was arbitrary, this completes the proof.
    
\end{proof}

\begin{remark}\label{smth.green.cvg}
From the arguments above, it is easy to see that the key estimates of Lemmas \ref{green.lem} and \ref{robin.reg} hold uniformly under small smooth perturbations of the metric $g$ on $\Sigma$ and the potential $V\in C^{\infty}(\Sigma)$. Thus, given a sequence of metrics $g_j$ on $\Sigma$ and potentials $V_j\in C^{\infty}(\Sigma)$ converging smoothly $(g_j,V_j)\to (g_0,V_0)$, as in the proof of Lemma \ref{sn.open}, the associated Green's functions $G_j$ satisfy uniform $C^{k+1}$ estimates on any compact subset $K\subset \Sigma\times \Sigma\setminus diag(\Sigma)$, so that an Arzela-Ascoli argument gives $\lim_{k\to\infty}\|G_j-G_0\|_{C^k(K)}=0$. Likewise, uniform $C^k$ estimates for the associated diagonal Robin functions implies their smooth convergence on $\Sigma$ as $j\to\infty$. Finally, note that uniform Lipschitz estimates for the Robin functions $R_j(x,y)$ on $\Sigma\times \Sigma$ and another application of Arzela-Ascoli give
$$\lim_{j\to\infty}\|R_j-R_0\|_{C^0(\Sigma\times \Sigma)}=0,$$
which together with the uniform convergence
$$\frac{d_{g_j}(x,y)}{d_{g_0}(x,y)}\to 1$$
implies that the difference
$$G_j(x,y)-G_0(x,y)=-\frac{1}{2\pi}\log(d_{g_j}(x,y)/d_{g_0}(x,y))+R_j(x,y)-R_0(x,y)$$
satisfies $\|G_j-G_0\|_{C^0(\Sigma\times \Sigma)}\to 0$ as $j\to\infty$.
\end{remark}

\subsection{Jacobi operator under a conformal change}\label{appendix:conf}

Let $(N,g)$ be a smooth Riemannian $3$-manifold, with $\Sigma\subset N$ be an arbitrary embedded, two-sided surface; we are primarily interested in the case where $\Sigma$ is minimal, but we do not assume minimality in the following computations.

Write $h=g|_{\Sigma}$ for the induced metric on $\Sigma$, and consider on $\Sigma$ the Schr\"odinger operator 
$$J:=d^*_hd-V ,\quad V:=|A|^2+\Ric(\nu,\nu),$$
where  $\nu$ is a unit normal vector field and $A$ the second fundamental form defined by $ A(X,Y):= g( \nabla_X Y,\nu)$, so that $J$ gives the Jacobi operator when $\Sigma$ is minimal.

Let $\rho$ be a smooth  function on $N$, and consider the conformal metric $\tilde g:=e^{2\rho}g$. Let $\tilde J$ be the  operator $d_{\tilde h}^*d-(|\tilde A|_{\tilde{g}}^2+\tilde {\Ric}(\tilde \nu,\tilde\nu))$ on $\Sigma$, with $\tilde A$, $\tilde {\Ric}$, $\tilde \nu$, and $\tilde{h}=\tilde{g}|_{\Sigma}$ all defined with respect to  $\tilde g$.

Below,  ${\bf H}$ denotes the mean curvature vector of $\Sigma$ in $(N,g)$, $\Delta$ denotes $ -d^*_gd=\tr D^2$, and $\nabla$ is taken with respect to the metric $g$.
\begin{lem}\label{lem:derivativeWRTs}Under the above convention, we have 
\begin{align*} 
\tilde J&=e^{-2\rho}J+e^{-2\rho}(2  {\bf H}\cdot  \nabla\rho+|\nabla \rho|^2-3|(\nabla \rho)^{\perp}|^2+\Hess(\rho)(\nu,\nu)+\Delta\rho),
\end{align*}
where $(\nabla\rho)^\perp=g(\nabla \rho,\nu)\nu$.
\end{lem}  
\begin{proof}
By definition,
$$\tilde J=d_{\tilde h}^*d-|\tilde A|^2-\tilde \Ric(\tilde\nu,\tilde\nu),$$
where  $\tilde\nu$ is the unit normal vector field on $\Sigma$ under $\tilde g$  defined by  $\tilde\nu=e^{-\rho}\nu$, and  $\tilde A(X,Y):=\tilde g(\tilde \nabla_X Y,\tilde\nu)$.

Recall (cf. \cite[Theorem 1.159]{besse2007einstein}) that the Ricci tensor $\tilde \Ric$ for $\tilde{g}=e^{2\rho}g$ on the $3$-manifold $N^3$ is given by
$$\tilde \Ric =\Ric-(\Hess_g(\rho)-d\rho\otimes d\rho) +(d^*_gd\rho-|d\rho|^2) g ,$$
so that
\begin{eqnarray*}
\tilde\Ric (\tilde\nu,\tilde\nu)&=&e^{-2\rho}\tilde\Ric(\nu,\nu)\\
&=&e^{-2\rho}(\Ric(\nu,\nu)-(\Hess_g(\rho)(\nu,\nu)-|(\nabla\rho)^{\perp}|^2))\\
&&+e^{-2\rho}(d_g^*d\rho-|d\rho|_g^2),
\end{eqnarray*}
while on the surface $\Sigma$, conformal covariance of the Laplacian gives $d^*_{\tilde h}d=e^{-2\rho}d_h^*d$ for the induced Laplacians. Therefore 
\begin{eqnarray*}
    \tilde J&=&e^{-2\rho}d^*_hd-|\tilde A|_{\tilde{g}}^2-e^{-2\rho}(\Ric(\nu,\nu)-(\Hess_g(\rho)(\nu,\nu)-|(\nabla\rho)^{\perp}|^2))\\
&&-e^{-2\rho}(d_g^*d\rho-|d\rho|_g^2)\\
&=&e^{-2\rho}(J+|A|_g^2)-|\tilde A|_{\tilde{g}}^2\\
&&+e^{-2\rho}(\Hess_g(\rho)(\nu,\nu)+|d\rho|_g^2-|(\nabla\rho)^{\perp}|^2-d_g^*d_g\rho)
\end{eqnarray*}
It remains to compute $|\tilde A|^2.$

Fix an arbitrary point $p\in\Sigma$ and choose an orthonormal frame $e_1,e_2$ in a neighborhood of $p$ with respect to the metric $g$, so that $\tilde e_1:=e^{-\rho}e_1,\tilde e_2:=e^{-\rho}e_2$ form an orthonormal frame with respect to $\tilde g$.
Now, let $\tilde h(X,Y):=\tilde g(\tilde \nabla_X Y,\tilde\nu)\tilde\nu.$ Then from \cite[\S 1.163]{besse2007einstein}, 
\begin{equation}\label{eq:2ndFF}
   \tilde  h(X,Y)=h(X,Y)-g(X,Y)(\nabla\rho)^\perp ,
\end{equation}
Hence,
\begin{align*}
    &\tilde g(\tilde h(\tilde e_i,\tilde e_j),\tilde\nu)\\
    &=\tilde g(h(\tilde e_i,\tilde e_j)-g(\tilde e_i,\tilde e_j)(\nabla\rho)^\perp  ,\tilde\nu)\\
    &=e^{2\rho}g(e^{-2\rho}h(e_i,e_j)-e^{-2\rho}g(e_i,e_j)(\nabla\rho)^\perp,e^{-\rho}\nu)\\
    &=e^{-\rho} g( h(e_i,e_j)- g(e_i,e_j)(\nabla\rho)^\perp,\nu)\\
    &=e^{-\rho}g(h(e_i,e_j),\nu)-e^{-\rho}g(e_i,e_j)g(\nabla\rho,\nu).
\end{align*} 
Hence, denoting $A_{ij}:=A(e_i,e_j)$,
\begin{align*}
    |\tilde A|_{\tilde{g}}^2&=e^{-2\rho}(A_{11}-g(\nabla\rho,\nu ))^2+e^{-2\rho}(A_{22}-g(\nabla\rho, \nu))^2+2e^{-2\rho}A_{12}\\
    &=e^{-2\rho}|A|_g^2-2e^{-2\rho}(A_{11}+A_{22})g(\nabla\rho,\nu)+2e^{-2\rho}g(\nabla\rho,\nu)^2\\
    &=e^{-2\rho}(|A|^2-2  {\bf H}\cdot  \nabla\rho +2 |(\nabla\rho)^\perp|^2).
\end{align*}

Combining this with the previous computations, we have
\begin{align*}
\tilde J=&e^{-2\rho}(J+|A|_g^2)-e^{-2\rho}(|A|^2-2  {\bf H}\cdot  \nabla\rho +2 |(\nabla\rho)^\perp|^2)\\
&+e^{-2\rho}(\Hess_g(\rho)(\nu,\nu)+|d\rho|_g^2-|(\nabla\rho)^{\perp}|^2-d_g^*d_g\rho)\\
=&e^{-2\rho}J+e^{-2\rho}(2  {\bf H}\cdot  \nabla\rho+|\nabla \rho|^2-3|(\nabla \rho)^{\perp}|^2)\\
&+e^{-2\rho}(\Hess(\rho)(\nu,\nu)+\Delta\rho),
\end{align*}
as claimed.
\end{proof}

\printbibliography

\end{document}